\documentclass[acmlarge, manuscript, screen, review, natbib=false, oneside]{acmart}

\usepackage{hyperref}
\usepackage{subfig}  
\usepackage{tikz}
\usetikzlibrary{positioning, decorations.pathreplacing}
\usepackage{float}   
\usepackage{multirow}
\usepackage{placeins} 

\gdef\@copyrightpermission{}     
\setcopyright{none}              
\acmDOI{}                        
\authorsaddresses{}              
\makeatletter
\renewcommand\@formatdoi[1]{}               
\def\@ACM@checkaffil{}                      
\long\def\footnotetextcopyrightpermission#1{}
\makeatother
\makeatletter
\@twosidefalse
\@mparswitchfalse
\makeatother
\acmVolume{}
\acmArticle{}
\acmMonth{}
\acmYear{}
\AtBeginDocument{%
  \fancyfoot{}%
  \fancyhead{}%
  \fancyhead[RO,LE]{\thepage}%
  \fancyhead[LO,RE]{\shortauthors}%
  \fancypagestyle{firstpagestyle}{\fancyhf{}\fancyfoot{}}%
}

\RequirePackage[
  style=authoryear-comp,
  backend=bibtex,
  giveninits=false,
  uniquename=false
  ]{biblatex}

\newtheorem{theorem}{Theorem}[section]

\newtheorem{corollary}[theorem]{Corollary}
\bibliography{references.bib}
\DeclareMathOperator{\CVaR}{CVaR}
\DeclareMathOperator{\VaR}{VaR}

\usepackage{algorithm}
\usepackage{algorithmicx}
\usepackage{algpseudocode}

\begin{document}
\raggedbottom

\title{Accelerated Online Risk-Averse Policy Evaluation in POMDPs with Theoretical Guarantees and Novel CVaR Bounds}

\author{Yaacov Pariente}
\affiliation{%
  \institution{Technion - Israel Institute of Technology}
  \city{Haifa}
  \country{Israel}
}

\author{Vadim Indelman}
\affiliation{%
  \institution{Technion - Israel Institute of Technology}
  \city{Haifa}
  \country{Israel}}

\renewcommand{\shortauthors}{Pariente \& Indelman}

\begin{abstract}
    Risk-averse decision-making under uncertainty in partially observable domains represents a central challenge in artificial intelligence and is essential for developing reliable autonomous agents. The formal framework for such problems is typically the partially observable Markov decision process (POMDP), where risk sensitivity is introduced through a risk measure applied to the value function. Among these, the Conditional Value-at-Risk (CVaR) has emerged as a particularly significant criterion. However, solving POMDPs is computationally intractable in general, and approximate solution methods rely on computationally expensive simulations of possible future agent trajectories. 
    
    This work introduces a theoretical framework for accelerating the evaluation of CVaR value functions in POMDPs while providing formal performance guarantees. As the mathematical foundation, we derive new bounds on the CVaR of a random variable $X$ using an auxiliary random variable $Y$, under assumptions relating their respective cumulative distribution and density functions; these bounds yield interpretable concentration inequalities and converge as the distributional discrepancy vanishes. Building on this foundation, we establish upper and lower bounds on the CVaR value function computable from a simplified belief-MDP, accommodating general simplifications of the transition dynamics—including reduced-complexity observation and transition models. We develop estimators for computing these bounds during online policy evaluation within a particle-belief MDP framework and provide probabilistic performance guarantees. These bounds are employed for computational acceleration via action elimination: actions whose bounds indicate suboptimality under the simplified model are safely discarded while ensuring consistency with the original POMDP. 
    
    Empirical evaluation across multiple POMDP domains confirms that the bounds reliably separate safe from dangerous policies while achieving substantial computational speedups under the simplified model.
\end{abstract}

\maketitle

\section{Introduction}
Autonomous agents have emerged as a cornerstone of modern intelligent systems, with applications spanning healthcare, disaster response, education, industrial automation, and decentralized systems. From autonomous aerial vehicles coordinating firefighting efforts to multi-agent systems executing complex financial transactions, these agents are designed to perceive, reason, and act without continuous human oversight. However, ensuring that such systems behave safely and reliably in uncertain and partially observable environments remains a central challenge. In these settings, risk-aware decision-making frameworks, particularly those grounded in partially observable Markov decision processes (POMDPs), provide a principled foundation for modeling uncertainty and enforcing safety through explicit risk-sensitive optimization \cite{marecki2010risk, hou2016solving, ahmadi2023risk}.

Exact solution methods for POMDPs become computationally intractable when the state, observation, and action spaces are large \cite{papa_mdp_complexity}. To address this fundamental limitation, approximate planning algorithms utilize forward search trees constructed via simulation of future agent trajectories \cite{silver2010monte, sunberg2018online}. The sampling procedure is guided by the underlying reward function, which provides evaluative feedback for action selection. This approach enables practical online planning in partially observable domains by trading off solution optimality for computational tractability.

Risk-averse decision-making frameworks in the literature encompass three primary approaches: distributionally robust optimization, which hedges against distributional uncertainty \cite{xu2010distributionally}; chance-constrained programming, which enforces probabilistic feasibility constraints \cite{santana2016aaai, moss2024constrainedzero}; and risk metric integration, wherein a risk measure—such as Conditional Value-at-Risk (CVaR)—is incorporated directly into the value function \cite{chow2015risk}. 
Risk metric integration is particularly appealing when the risk measure is coherent, as coherent risk measures satisfy fundamental axioms—monotonicity, translation invariance, positive homogeneity, and subadditivity—that ensure rational and consistent behavior under aggregation and scaling of uncertain returns. These properties are especially important in sequential decision-making, where returns accumulate over time and diversification across stochastic outcomes should not be penalized. Moreover, coherence guarantees convexity of the resulting objective, enabling tractable optimization and facilitating theoretical analysis. Among coherent risk measures, CVaR has emerged as a widely adopted choice due to its explicit focus on tail risk and its compatibility with robust and distributional formulations.
In the latter framework, CVaR is computed over the cumulative return, defined as the sum of stage-wise costs across the planning horizon. 
Chance-constrained and CVaR-based formulations reflect different perspectives on risk in sequential decision problems. 

Chance constraints enforce probabilistic restrictions on state visitation, thereby preventing the agent from entering undesirable regions of the state space. Conversely, CVaR optimization targets the tail risk of the cumulative return, prioritizing mitigation of poor reward realizations. These objectives are complementary rather than interchangeable. Consider, for instance, an online algorithmic trading problem wherein the environment reward is defined as the percentage portfolio gain relative to the preceding portfolio value, and the state is characterized by current asset prices. In this setting, CVaR optimization of returns directly addresses the trader's primary objective—managing downside risk in cumulative gains—whereas formulating the problem as a chance constraint on state space would be less natural and potentially misaligned with the risk management goal. Another example arises in sequential decision-making problems with multiple objectives, such as a navigating robot that must account for fuel availability, obstacles, and related operational constraints. In such settings, optimizing the cumulative reward provides a principled means of jointly accommodating all constraints.

Current state-of-the-art POMDP planning algorithms remain computationally demanding, as they rely on extensive sampling from the original POMDP dynamics. To mitigate this issue, the simplification approach replaces the true POMDP dynamics with a computationally tractable surrogate model, which is then employed during planning while still admitting performance guarantees relative to the original dynamics. Simplification techniques for expectation-based value functions are already available \cite{LevYehudi24aaai, barenboim2023NIPS, andrey_ori} and have been shown to accelerate POMDP planning. However, risk-averse simplification has remained largely unexamined; the only existing contribution in this direction employs a non-coherent risk measure \cite{Zhitnikov22ai}, which offers weaker justification for risk-averse decision making \cite{majumdar2020should}.

We begin by establishing the mathematical foundations of this work, deriving several CVaR bounds for a random variable $X$ using an auxiliary random variable $Y$, under assumptions relating their respective CDFs and PDFs. Although not all of these theoretical results are employed in our practical applications, they form a foundational framework that we expect will support future research on simplification.
Building on these foundations, we analyze the relationship between the original value function—defined as the CVaR of the return—and its simplified belief transition-model counterpart. In particular, we establish lower and upper bounds on the original value function in terms of the simplified value function and provide corresponding performance guarantees. We then develop estimators for computing these bounds during online planning and derive performance guarantees for these estimators as well. The primary contributions of this work are as follows:

\begin{itemize}
	\item Derivation of bounds on the CVaR of a random variable $X$ given another random variable $Y$. These bounds enable the approximation of $\CVaR_\alpha(X)$ when direct access to $X$ is limited. They also yield interpretable concentration inequalities that characterize the CVaR of $X$ through the CVaR of $Y$ under an adjusted confidence level, and we establish convergence of these bounds as the distributional discrepancy vanishes.
	
	\item Derivation of lower and upper bounds on the theoretical value function in risk-averse POMDPs expressed through the action-value function from a simplified belief-MDP transition model, with corresponding bounds for simplified observation models.
	
	\item Estimators for evaluating the simplified bounds during online planning, with probabilistic performance guarantees on the deviation between the theoretical action-value function and its estimated bounds computed using simplified belief-MDP and observation models within a particle-belief MDP framework.

	\item Application of the derived bounds to computational acceleration via action elimination, in which actions whose bounds indicate suboptimality are safely discarded, and empirical demonstration of substantial speedups with negligible degradation in policy performance across multiple POMDP domains.
\end{itemize}

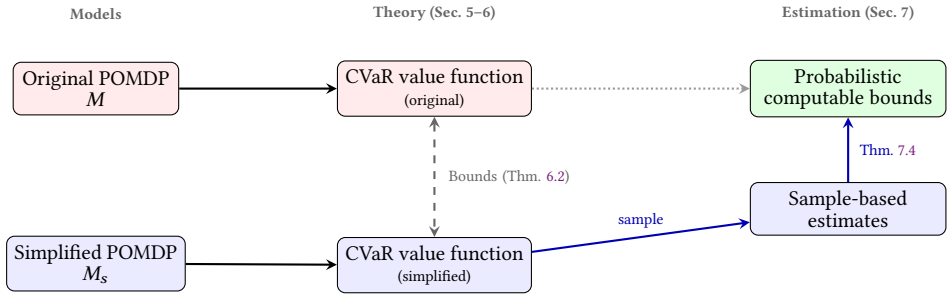
\begin{figure}[t]
	\centering
	\begin{tikzpicture}[
		node distance=1.0cm and 1.6cm,
		box/.style={rectangle, draw, rounded corners=3pt, minimum width=2.2cm, minimum height=0.7cm, align=center, font=\small},
		arr/.style={->, thick, >=stealth},
		darr/.style={<->, thick, >=stealth, dashed},
		lbl/.style={font=\scriptsize, midway},
		phase/.style={font=\scriptsize\bfseries, text=black!60},
	]
		\node[phase] at (0, 1.0) {Models};
		\node[phase] at (4.5, 1.0) {Theory (Sec.~5--6)};
		\node[phase] at (10.0, 1.0) {Estimation (Sec.~7)};

		\node[box, fill=red!8] (origM) at (0, 0) {Original POMDP\\[-2pt] $M$};
		\node[box, fill=red!8] (origQ) at (4.5, 0) {CVaR value function\\[-2pt] \scriptsize(original)};

		\node[box, fill=blue!8, below=1.6cm of origM] (simpM) {Simplified POMDP\\[-2pt] $M_s$};
		\node[box, fill=blue!8, below=1.6cm of origQ] (simpQ) {CVaR value function\\[-2pt] \scriptsize(simplified)};

		\node[box, fill=blue!8, minimum width=2.6cm] (est) at (10.0, -1.6) {Sample-based\\[-2pt] estimates};
		\node[box, fill=green!12, minimum width=2.6cm] (bounds) at (10.0, 0) {Probabilistic\\[-2pt] computable bounds};

		\draw[arr] (origM) -- (origQ);
		\draw[arr] (simpM) -- (simpQ);

		\draw[darr, black!60] (origQ) -- (simpQ) node[lbl, right, xshift=2pt, text width=1.8cm, align=left] {\scriptsize Bounds (Thm.~\ref{thm:uniform_lower_and_upper_bounds_for_v_and_q})};

		\draw[arr, blue!70!black] (simpQ) -- (est) node[lbl, above] {\scriptsize sample};

		\draw[arr, blue!70!black] (est) -- (bounds) node[lbl, right, xshift=1pt] {\scriptsize Thm.~\ref{thm:guarantees}};

		\draw[arr, black!40, densely dotted] (origQ) -- (bounds);
	\end{tikzpicture}
	\caption{Overview of the proposed framework. The original POMDP yields the original CVaR value function. The simplified POMDP $M_s$ admits a simplified value function, related to the original via Theorem~\ref{thm:uniform_lower_and_upper_bounds_for_v_and_q}. In practice, the simplified value function and distributional discrepancy are estimated from sample trajectories, yielding probabilistic computable bounds on the original value function (Theorem~\ref{thm:guarantees}).}
	\label{fig:framework_overview}
\end{figure}
An overview of the proposed framework is depicted in Figure~\ref{fig:framework_overview}. The remainder of this paper is organized as follows. Section~2 reviews related work. Section~3 introduces the necessary preliminaries on POMDPs, particle belief MDPs, and CVaR estimation. Section~4 formulates the problem of bounding the risk-averse value function under a simplified belief-transition model. Section~5 derives the foundational CVaR bounds using auxiliary random variables. Section~6 establishes value function bounds for the static CVaR value function under general and observation-model simplifications. Section~7 develops online estimators for computing these bounds during planning and provides corresponding performance guarantees. Section~8 discusses the limitations of the framework. Section~9 presents the experimental evaluation across multiple POMDP domains. Section~10 concludes the paper. Proofs and additional experimental details are provided in the appendix.

\section{Related Work}
Conditional Value at Risk (CVaR) \cite{rockafellar2000optimization} is a principled and extensively studied risk measure with broad applicability across domains involving uncertainty. A central property of CVaR is its dual representation \cite{artzner1999coherent}, which allows it to be interpreted as the expected loss in the worst-case tail of a cost distribution \cite{chow2015risk}. Unlike Value at Risk (VaR), which identifies only the quantile threshold of extreme losses at a chosen confidence level, CVaR quantifies both the probability and the magnitude of such losses, yielding a more comprehensive assessment of risk. This property makes CVaR especially relevant in safety-critical decision-making contexts, where both the occurrence and severity of adverse outcomes are significant. Moreover, CVaR satisfies the axioms of coherent risk measures—such as subadditivity and positive homogeneity—ensuring consistent and rational aggregation of risk across time and decision stages \cite{artzner1999coherent}. Practical estimators of CVaR are supported by concentration inequalities and deviation guarantees from the true risk value \cite{brown2007large,pmlr-v97-thomas19a}, further reinforcing its reliability in applied decision-making problems.

Risk can be incorporated into planning under uncertainty through multiple paradigms, including chance constraints \cite{ono2015chance}, exponential utility functions \cite{koenig1994risk}, distributionally robust optimization \cite{xu2010distributionally, osogami2015robust, cubuktepe2021robust}, and quantile regression methods \cite{dabney2018distributional}. Among these, distributionally robust formulations are particularly aligned with CVaR optimization, as both emphasize resilience to rare but high-impact events \cite{chow2015risk}. Risk measures map random cost variables to real-valued evaluations and are expected to satisfy fundamental axioms to ensure interpretability and consistency \cite{majumdar2020should}. Coherent risk measures, by satisfying these axioms, provide a principled foundation for incorporating CVaR into sequential decision-making. General coherent risk formulations have been employed as optimization objectives in POMDPs, constrained MDPs, and shortest-path problems \cite{ahmadi2021constrained, ahmadi2020risk, ahmadi2021risk, dixit2023risk}, where CVaR arises as a key special case. Specifically, \cite{chow2015risk} introduced the CVaR-MDP framework, defining the value function as the CVaR of the return and deriving a value-iteration-based solution with formal error guarantees.

Simplification techniques for POMDPs aim to reduce computational complexity while retaining performance guarantees, thereby facilitating real-time deployment of POMDP-based policies in practical systems \cite{Papadimitriou1987TheCO}. Simplification refers to replacing one or more POMDP components with computationally tractable approximations while providing formal performance guarantees that ensure decision quality with respect to the original model. For instance, \cite{LevYehudi24aaai} examined simplification of the observation model by introducing a less expensive surrogate while establishing finite-sample convergence bounds. Similarly, \cite{BARENBOIM2026104442} studied simplification of the state and observation spaces, providing deterministic performance guarantees and integrating them into state-of-the-art solvers. Furthermore, \cite{Zhitnikov22ai} examined the simplification of belief-dependent rewards and provided deterministic guarantees under the Value at Risk (VaR) criterion. However, VaR is not a coherent risk measure and therefore lacks several essential properties for principled risk-averse decision making. More broadly, the simplification of risk-averse planning remains largely unexplored, with \cite{Zhitnikov22ai} being the only prior work in this direction. To the best of our knowledge, no prior work has addressed the problem of bounding the CVaR of a random variable using an auxiliary random variable with a known distributional discrepancy, which constitutes a key contribution of this work.

\section{Preliminaries}
\subsection{Partially Observable Markov Decision Process}\label{sec:POMDP_preliminaries}
A finite-horizon Partially Observable Markov Decision Process (POMDP) is formally defined as the tuple $(X, A, Z, T, O, c, b_0)$, where $X$, $A$, and $Z$ denote the state, action, and observation spaces, respectively. The transition model $T(x_{t+1} \mid x_t, a_t) \triangleq P(x_{t+1} \mid x_t, a_t)$ specifies the conditional probability of transitioning from state $x_t$ to $x_{t+1}$ given action $a_t$, while the observation model $O(z_t \mid x_t) \triangleq P(z_t \mid x_t)$ defines the likelihood of observing $z_t$ when the system is in state $x_t$. The belief space $B$ is the set of all probability distributions over $X$, and the stage-wise cost function is given by $c: B \times A \to \mathbb{R}$.

At each time step $t$, the agent maintains a belief $b_t \in B$, representing the posterior distribution over the latent state conditioned on the history of observations and actions. The history is denoted by $H_t = \{z_{1:t}, a_{0:t-1}, b_0\}$, and the belief update is defined as $b_t(x_t) \triangleq P(x_t \mid H_t)$ for each $x_t \in X$. A policy $\pi_t: B \to A$ is a measurable mapping that prescribes the action $a_t = \pi_t(b_t)$ based on the current belief state. The immediate expected cost under belief $b_t$ and action $a_t$ is
\begin{equation}
	c(b_t, a_t) \triangleq \mathbb{E}_{x \sim b_t}[c_x(x, a_t)],
\end{equation}
where $c_x(x, a_t)$ denotes the state-dependent cost, bounded as $R_{\min} \le c_x(x, a_t) \le R_{\max}$.

The cumulative cost, or \emph{return}, over a finite horizon $T \in \mathbb{N}$ is given by
\begin{equation}
	R_{t:T} \triangleq \sum_{\tau=t}^T c(b_\tau, a_\tau),
\end{equation}
which serves as the performance criterion starting at time $t$. The value function associated with policy $\pi$ and initial belief $b_k$ is
\begin{equation}\label{Eq:regular_V_function}
	V^\pi(b_k) \triangleq \mathbb{E}[R_{k:T} \mid b_k, \pi] = \sum_{t=k}^T \mathbb{E}[c(b_t, a_t) \mid b_k, \pi],
\end{equation}
and the corresponding action-value function is
\begin{equation}\label{Eq:regular_Q_function}
	Q^\pi(b_k, a_k) \triangleq \mathbb{E}_{z_{k+1}}\!\left[c(b_k, a_k) + V^\pi(b_{k+1}) \mid b_k, a_k \right].
\end{equation}
A POMDP can equivalently be formulated as a fully observable \emph{belief-MDP} (BMDP), in which the belief $b_t$ serves as the state variable and is a sufficient statistic for the interaction history. In this formulation, the dynamics are governed by a belief transition kernel induced by the latent state-transition and observation models. Formally, the belief-MDP associated with a POMDP $M$ is defined as
\begin{equation}
	M_B \triangleq (\mathcal{B}, A, \tau, \rho, \gamma),
\end{equation}
where $\mathcal{B} \triangleq \Delta(X)$ denotes the space of probability distributions over the latent state space $X$, $A$ is the action space, $\tau$ denotes the belief-transition kernel, $\rho$ is the expected cost, and $\gamma$ is the discount factor. Given the state-transition and observation models, the belief-transition model of the belief-MDP is given by
\begin{equation}
	P(b_{t+1}|b_{t},a_t)=\int_{x_t \in X}\int_{x_{t+1}\in X}\int_{z_{t+1}\in Z}P(x_{t+1}|x_t,a_t)P(z_{t+1}|x_{t+1})b_t(x_t)dx_{t}dx_{t+1}dz_{t+1}.
\end{equation}
This BMDP formulation is exact but generally intractable, since the belief space $\mathcal{B}$ is infinite-dimensional, motivating approximate representations such as particle belief MDPs.

\subsection{Particle Belief MDP }
In order to estimate the theoretical action-value function, we consider a particle belief MDP (PB-MDP) setting. Formally, denote by $M_P \triangleq (\Sigma, A, \tau, \rho, \gamma)$ the PB-MDP that is defined with respect to the POMDP $M$ \cite{lim2023optimality} and $N_p \in \mathbb{N}$, where 
\begin{itemize}
	\item $\Sigma \triangleq \{\bar{b}:\bar{b}=\{(x_i,w_i)\}_{i=1}^{N_p},x_i\in X, \forall i,w_i\geq 0,\exists i\text{ such that } w_i>0\}$ is the state space over the particle beliefs.
	\item $A$ is the action space as defined in the POMDP $M$.
	\item $\tau(\bar{b}_{t+1}|\bar{b}_{t}, a)$ is the belief transition probability, for $a\in A,\bar{b}\in \Sigma$.
	\item $\rho(\bar{b},a) \triangleq \frac{\sum_{j=1}^{N_p} w_j c(x_j,a)}{\sum_{j=1}^{N_p} w_j}$ is the belief-dependent cost.
	\item $\gamma$ as defined in the POMDP $M$.
\end{itemize}
This belief estimator maintains a finite set of state particles that serve as an approximation to the theoretical belief, which, in principle, may involve an infinite number of states.

\subsection{Conditional Value-at-Risk}\label{sec:cvar_preliminaries}
Let $X$ be a random variable defined on a probability space $(\Omega, \mathcal{F}, P)$, where $\mathcal{F}$ is the $\sigma$-algebra, and $P : \mathcal{F} \to [0,1]$ is a probability measure. Assume further that $\mathbb{E}[|X|] < \infty$. We denote the cumulative distribution function (CDF) of the random variable $X$ by $F_X(x)=P(X\leq x)$. The value at risk at confidence level $\alpha \in (0,1)$ is the $1-\alpha$ quantile of $X$, i.e., \begin{equation}
	\VaR_\alpha (X) \triangleq \inf\{x\in \mathbb{R}:F_X(x) > 1-\alpha\}
\end{equation}
For simplicity we denote $q_\alpha^X=\VaR_\alpha(X)$ or $q_\alpha^F=\VaR_\alpha(X)$. The conditional value at risk (CVaR) at confidence level $\alpha$ is defined as \cite{rockafellar2000optimization}
\begin{equation}
	\CVaR_\alpha(X) \triangleq \inf_{w\in \mathbb{R}}\{w+\frac{1}{\alpha}\mathbb{E}[(X-w)^+]|w\in \mathbb{R}\},
\end{equation}
where $(x)^+=\max{(x,0)}$. For a smooth $F$, it holds that \cite{pflug2000some}
\begin{equation}\label{def:cvar_as_conditional_expectation}
	\CVaR_\alpha(X)=\mathbb{E}[X|X>\VaR_\alpha(X)]=\frac{1}{\alpha}\int_{1-\alpha}^1 F^{-1}(v)dv.
\end{equation}
Let $X_i \overset{\text{iid}}{\sim} F$ for $i\in\{1,\dots,n\}$. Denote by
\begin{equation}\label{eq:brown_cvar_estimator}
	\hat{C}_\alpha(X) \triangleq \hat{C}_\alpha(\{X_i\}^{n}_{i=1}) \triangleq \inf_{x\in \mathbb{R}} \Bigl{\{}x+\frac{1}{n\alpha} \sum_{i=1}^n(X_i-x)^+\Bigr{\}}
\end{equation}
the estimate of $\CVaR_\alpha(X)$ \cite{brown2007large}. Theorem \ref{thm:brown_bounds}, 
that bounds the deviation of the estimated CVaR and the true CVaR, was proved in \cite{brown2007large}.
\begin{theorem}\label{thm:brown_bounds}
	If $\text{supp}(X) \subseteq [a, b]$ and $X$ has a continuous distribution function, then for any $\delta \in (0, 1]$, 
	\begin{equation}
		P\Bigl{(}\CVaR_\alpha(X)-\hat{C}_\alpha(X)>(b-a)\sqrt{\frac{5\ln(3/\delta)}{\alpha n}}\Bigr{)}\leq \delta,
	\end{equation}
	\begin{equation}
		P\Bigl{(}\CVaR_\alpha(X)-\hat{C}_\alpha(X)<-\frac{(b-a)}{\alpha}\sqrt{\frac{\ln(1/\delta)}{2n}}\Bigr{)}\leq \delta.
	\end{equation}
\end{theorem}

The estimator in~\eqref{eq:brown_cvar_estimator} can be expressed as
\begin{equation}\label{eq:cvar_estimator_sorted_sample}
	\hat{C}_\alpha(X)=X^{(n)}-\frac{1}{\alpha}\sum_{i=1}^{n-1} (X^{(i)}-X^{(i-1)})\Bigl{(} \frac{i}{n}-(1-\alpha) \Bigr{)}^+,
\end{equation}
where $X^{(i)}$ is the $i$th order statistic of $X_1,\dots,X_n$ in ascending order \cite{pmlr-v97-thomas19a}. The results presented in Theorems~\ref{thm:thomas_upper_bound} and~\ref{thm:thomas_lower_bound}, following the work of \cite{pmlr-v97-thomas19a}, yield tighter bounds on the CVaR compared to those established by \cite{brown2007large}.
\begin{theorem}\label{thm:thomas_upper_bound}
	If $X_1, \ldots, X_n$ are independent and identically distributed random variables and $\Pr(X_1 \le b) = 1$ for some finite $b$, then for any $\delta \in (0, 0.5]$,
	\begin{equation}
		\begin{aligned}
			&\Pr\Bigg(
			\CVaR_\alpha(X_1)
			\le Z_{n+1}
			- \frac{1}{\alpha} \sum_{i=1}^n (Z_{i+1} - Z_i) 
			\Bigg(
			\frac{i}{n} - \sqrt{\frac{\ln(1/\delta)}{2n}}
			- (1 - \alpha)\Bigg)^{+}
			\Bigg) \ge 1 - \delta,
		\end{aligned}
	\end{equation}
	where $Z_1, \ldots, Z_n$ are the order statistics (i.e., $X_1, \ldots, X_n$ sorted in ascending order), $Z_{n+1} = b$, and $x^+ \triangleq  \max\{0, x\}$ for all $x \in \mathbb{R}$.
\end{theorem}

\begin{theorem}\label{thm:thomas_lower_bound}
	If $X_1, \ldots, X_n$ are independent and identically distributed random variables and $\Pr(X_1 \ge a) = 1$ for some finite $a$, then for any $\delta \in (0, 0.5]$,
	\begin{equation}
		\begin{aligned}
			&\Pr\Bigg(
			\CVaR_\alpha(X_1)
			\ge Z_n
			- \frac{1}{\alpha} \sum_{i=0}^{n-1} (Z_{i+1} - Z_i)
			\Big(
			\min \Big\{
			1,\, \frac{i}{n} + \sqrt{\frac{\ln(1/\delta)}{2n}}
			\Big\}
			- (1 - \alpha) \Big)^{+}
			\Bigg) \ge 1 - \delta,
		\end{aligned}
	\end{equation}
	where $Z_1, \ldots, Z_n$ are the order statistics (i.e., $X_1, \ldots, X_n$ sorted in ascending order), $Z_0 = a$, and $x^+ \triangleq \max\{0, x\}$ for all $x \in \mathbb{R}$.
\end{theorem}
Notably, Theorems~\ref{thm:thomas_upper_bound} and~\ref{thm:thomas_lower_bound} require only one-sided boundedness of the support and do not assume continuity of the distribution function, in contrast to Theorem~\ref{thm:brown_bounds}, which requires both two-sided boundedness and a continuous distribution. The restriction to $\delta \in (0, 0.5]$ is mild in practice, as it corresponds to confidence levels of at least $50\%$.

\section{Problem Formulation}

In this work, we adopt a risk-averse formulation in which, instead of optimizing the expected return as in \eqref{Eq:regular_V_function} and \eqref{Eq:regular_Q_function}, the value function is defined as the CVaR of the return. Let $\alpha \in (0,1)$, and denote the original POMDP by $M$. The value and action-value functions at time $k$ are defined by
\begin{equation}\label{Eq:cvar_value_function_def}
	\begin{aligned}
		&V^{\pi}_{M}(b_k,\alpha) \! \triangleq \! \CVaR_\alpha^P[\sum_{t=k}^T \! c(b_t, \pi(b_t))|b_k, \pi],
	\end{aligned}
\end{equation}
\begin{equation}\label{Eq:cvar_Q_function_def}
	\begin{aligned}
		Q^{\pi}_{M}(b_k, a_k, \alpha)
		\triangleq \CVaR_\alpha^P[c(b_k, a_k)
		+ \sum_{t=k+1}^T c(b_t, \pi(b_t))|b_k, \pi].
	\end{aligned}
\end{equation}
Figure~\ref{fig:v_value_vs_cvar_v_value} illustrates the difference between the standard expected value function and the CVaR value function.

We consider cases in which the belief-MDP transition model is simplified for computational tractability. Let $\Omega_B$ be the sample space of the belief random variable and let $\Omega \triangleq \Omega_B^{T-k+1}$ denote the corresponding product space, endowed with a $\sigma$-algebra $\mathcal{F}$. We define two probability measures $P, P_s : \mathcal{F} \to [0,1]$ on the measurable space $(\Omega,\mathcal{F})$, corresponding to the original and simplified belief models, respectively. The CDF of $R_{k:T}$ under the original model is
\begin{equation}\label{eq:OriginalCDFReturn}
		P(R_{k:T}\leq l|b_k,\pi) = \int_{b_{k+1:T} \in B^{T-k}} \!\! P(R_{k:T}\leq l|b_{k:T},\pi)
		\prod_{i=k+1}^T P(b_i|b_{i-1},\pi)db_{k+1:T},
\end{equation}
and the simplified CDF of $R_{k:T}$ is defined using only the simplified belief transition model:
\begin{equation}\label{eq:SimplifiedCDFReturn}
		P_s(R_{k:T}\leq l|b_k,\pi) = \int_{b_{k+1:T} \in B^{T-k}} \!\! P(R_{k:T}\leq l|b_{k:T},\pi)
		\prod_{i=k+1}^T P_s(b_i|b_{i-1},\pi)db_{k+1:T}.
\end{equation}
The belief-MDP transition model is defined using the state transition and observation models. Hence, simplification of the belief-MDP transition model is general enough to include simplifications of state and observation models simultaneously.

Denote the simplified POMDP by $M_s$. The value and action-value functions corresponding to $M_s$ are denoted by $V^{\pi}_{M_s}(b_k, \alpha)$ and $Q^{\pi}_{M_s}(b_k, a_k, \alpha)$, respectively. These definitions are identical to those in \eqref{Eq:cvar_value_function_def} and \eqref{Eq:cvar_Q_function_def}, except that the return distribution is defined with respect to the simplified belief transition model $P_s$.

While the simplified model $P_s$ is computationally tractable, directly using $Q^\pi_{M_s}(b_k, a_k, \alpha)$ as a surrogate for $Q^\pi_M(b_k, a_k, \alpha)$ may lead to incorrect risk assessments, since the return distributions under $P$ and $P_s$ differ. Our objective is to establish lower and upper bounds, denoted by $L_s$ and $U_s$, such that $L_s \leq Q_M^\pi(b_k, a_k, \alpha) \leq U_s$, depending solely on quantities computable from the simplified model~\eqref{eq:SimplifiedCDFReturn}; this conceptual approach is summarized in Figure~\ref{fig:conceptual_overview}. These bounds provide formal guarantees on the true value function, enabling reliable action selection during online planning. As illustrated in Figure~\ref{fig:problem_formulation}, the key quantity governing these bounds is $\epsilon$, an upper bound on the discrepancy between the CDFs of the return under $P$ and $P_s$. Bounding this CDF discrepancy allows us to bound the difference in CVaR, as established in subsequent sections.

\begin{figure}[t]
	\centering
	\begin{tikzpicture}[
		node distance=1.2cm and 2.0cm,
		box/.style={rectangle, draw, rounded corners=3pt, minimum width=2.6cm, minimum height=0.7cm, align=center, font=\small},
		arr/.style={->, thick, >=stealth},
		darr/.style={<->, thick, >=stealth, dashed},
		lbl/.style={font=\scriptsize, midway},
	]
		\node[box, fill=red!8] (origM) at (0,0) {Original POMDP $M$};
		\node[box, fill=red!8] (origQ) at (6,0) {$Q^\pi_M(b_k,a_k,\alpha)$};

		\node[box, fill=blue!8] (simpM) at (0,-1.8) {Simplified POMDP $M_s$};
		\node[box, fill=green!10] (bounds) at (6,-1.8) {Computable bounds\\[-2pt] $L_s \leq Q^\pi_M \leq U_s$};

		\draw[arr, black!40, densely dotted] (origM) -- (origQ);
		\draw[arr, blue!70!black] (simpM) -- (bounds);

		\draw[darr, red!70!black] (origM) -- (simpM) node[lbl, right, xshift=2pt, align=left] {discrepancy\\[-2pt] $\epsilon$};
		\draw[arr, black!50, densely dotted] (origQ) -- (bounds);
	\end{tikzpicture}
	\caption{Conceptual overview. Direct evaluation of $Q^\pi_M$ under the original model $M$ is computationally expensive. Instead, the simplified model $M_s$ is used to compute bounds on $Q^\pi_M$, where the bound quality is governed by the distributional discrepancy $\epsilon$ between their belief transition models.}
	\label{fig:conceptual_overview}
\end{figure}
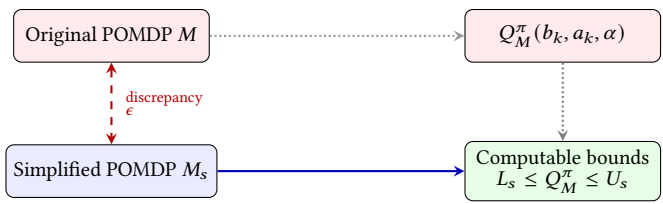

\begin{figure}[t]
	\centering
	\subfloat[Comparison of expected value $V^\pi_M(b_k)$ and CVaR value $V^\pi_M(b_k,\alpha)$. The shaded region represents the worst $\alpha$ fraction of outcomes.]{%
		\begin{tikzpicture}[scale=0.975]
			\def\skewpdf#1{2.5*pow(#1,1.8)*exp(-1.2*(#1))}

			\draw[->] (-0.3,0) -- (6.5,0) node[below] {\textbf{Return}};
			\draw[->] (0,-0.3) -- (0,4.2) node[midway, left, rotate=90, anchor=south] {\textbf{Probability Density}};

			\def\xmean{1.67}
			\def\xvar{3.2}
			\def\xcvar{4.1}

			\fill[red!25] plot[domain=\xvar:5.8, samples=100] (\x, {\skewpdf{\x}}) -- (5.8,0) -- (\xvar,0) -- cycle;

			\draw[blue, very thick, domain=0.05:5.8, samples=150] plot (\x, {\skewpdf{\x}});

			\draw[green!60!black, dashed, very thick] (\xmean,0) -- (\xmean,3.5);

			\draw[red!70!black, dashed, very thick] (\xcvar,0) -- (\xcvar,3.5);

			\node[anchor=north east, fill=white, draw=black!30, rounded corners, inner sep=4pt, font=\small] at (6.3,4.0) {
				\begin{tabular}{cl}
					\tikz\fill[red!25] (0,0) rectangle (0.3,0.2); & Upper $\alpha$ tail \\[1pt]
					\tikz\draw[green!60!black, dashed, thick] (0,0.1) -- (0.3,0.1); & $V^\pi_M(b_k)$ \\[1pt]
					\tikz\draw[red!70!black, dashed, thick] (0,0.1) -- (0.3,0.1); & $V^\pi_M(b_k, \alpha)$
				\end{tabular}
			};
		\end{tikzpicture}%
		\label{fig:v_value_vs_cvar_v_value}
	}
	\quad
	\subfloat[CDF discrepancy between original ($F_X$) and simplified ($F_Y$) models. The $\epsilon$-band bounds the discrepancy, enabling CVaR bounds.]{%
		\begin{tikzpicture}[scale=0.975,
			declare function={
				cdfY(\x) = 1/(1+exp(-2.5*(\x-2)));
				cdfX(\x) = 1/(1+exp(-2.4*(\x-2.15)));
				epsval = 0.12;
			}]

			\draw[->] (-0.3,0) -- (5.5,0) node[midway, below] {\textbf{Return}};
			\draw[->] (0,-0.3) -- (0,4.5) node[midway, left, rotate=90, anchor=south] {\textbf{Cumulative Probability}};

			\node[left] at (0,0) {\small $0$};
			\node[left] at (0,4.26) {\small $1$};

			\fill[green!15] plot[domain=0:5, samples=80] (\x, {4.26*max(0,min(1,cdfY(\x)-epsval))})
				-- plot[domain=5:0, samples=80] (\x, {4.26*min(1,cdfY(\x)+epsval)}) -- cycle;

			\draw[blue, very thick, domain=0:5, samples=80] plot (\x, {4.26*cdfY(\x)});

			\draw[green!50!black, thick, dashed, domain=0:5, samples=80] plot (\x, {4.26*min(1,cdfY(\x)+epsval)});
			\draw[green!50!black, thick, dashed, domain=0:5, samples=80] plot (\x, {4.26*max(0,cdfY(\x)-epsval)});

			\draw[red, very thick, domain=0:5, samples=80] plot (\x, {4.26*cdfX(\x)});

			\draw[black!50, thin, dashed] (0,4.26) -- (5,4.26);

			\draw[<->, thick, black!70] (4.2,{4.26*cdfY(4.2)}) -- (4.2,{4.26*min(1,cdfY(4.2)+epsval)});
			\node[right, font=\small] at (4.2,{4.26*(cdfY(4.2)+epsval/2)}) {$\epsilon$};

			\node[anchor=south east, fill=white, fill opacity=0.5, text opacity=1, draw=black!30, rounded corners, inner sep=4pt, font=\small] at (5.3,0.2) {
				\begin{tabular}{cl}
					\tikz\draw[blue, thick] (0,0.1) -- (0.3,0.1); & $F_Y$ (simplified) \\[1pt]
					\tikz\draw[red, thick] (0,0.1) -- (0.3,0.1); & $F_X$ (original) \\[1pt]
					\tikz\fill[green!15] (0,0) rectangle (0.3,0.2); & $\epsilon$-band
				\end{tabular}
			};
		\end{tikzpicture}%
		\label{fig:cdf_discrepancy}
	}
	\caption{Problem formulation illustrations. (a) The CVaR value function focuses on the tail of the return distribution, unlike the expected value. (b) The simplified model $P_s$ yields a computable CDF $F_Y$, while the original CDF $F_X$ is intractable. When $|F_X(r) - F_Y(r)| \leq \epsilon$, bounds on $\CVaR_\alpha(X)$ can be derived from $Y$.}
	\label{fig:problem_formulation}
\end{figure}
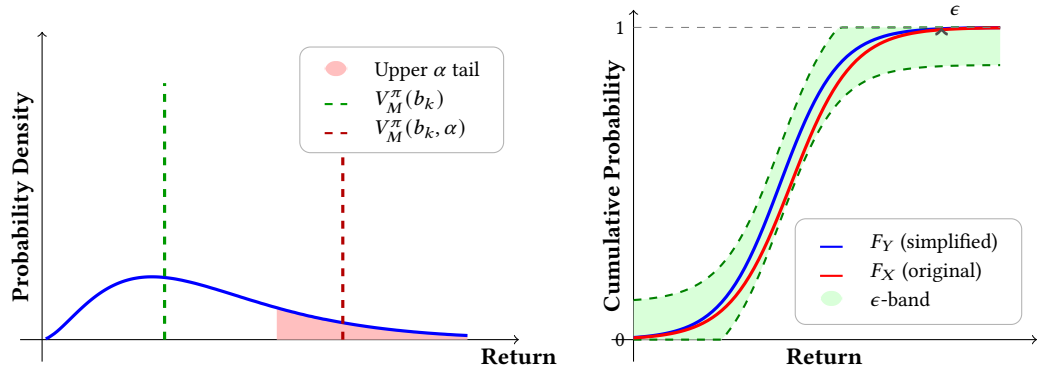

\section{CVaR Bounds}
In the previous section, we introduced a simplified belief-transition model with distribution $P_s$ as a computationally tractable approximation of the original belief-transition model $P$. This approximation induces two corresponding random variables: $X$, representing the return under the original model $P$, and $Y$, representing the return under the simplified model $P_s$. While direct evaluation of risk-sensitive criteria for $X$ is often computationally infeasible, sampling and analysis of $Y$ can be carried out efficiently.

The goal of this section is to characterize how $\CVaR_\alpha(X)$ can be bounded and estimated using information derived from $Y$. To this end, we first establish distributional bounds that relate $X$ and $Y$, and then derive finite-sample guarantees that allow these bounds to be estimated from data. Together, these results provide a principled framework for bounding the CVaR of the true return using samples generated from the simplified model.

\subsection{Theoretical CVaR Bounds}\label{sec:theoretical_cvar_bounds}
In this section, we establish bounds for the CVaR of the random variable $X$ by leveraging an auxiliary random variable $Y$. Two forms of distributional relationships between their respective CDFs are considered: a uniform bound and a non-uniform bound, as illustrated in Figure~\ref{fig:eps_bound} and Figure~\ref{fig:g_bound}. These results provide the theoretical basis for subsequent sections, in which the derived bounds are applied to estimate the CVaR of random variables that are either computationally intractable or prohibitively expensive to sample directly. 

Theorem~\ref{thm:cvar_bound_v2} bounds $\CVaR_\alpha(X)$ in terms of $\CVaR_\alpha(Y)$ under the sole condition that the cumulative distribution functions of $X$ and $Y$ differ by at most a known uniform bound. This representation enhances the interpretability of the bound and constitutes a novel aspect of the result, made possible by framing the bounding problem in terms of distributional discrepancies. In the next section, we demonstrate that the bounds established in \cite{pmlr-v97-thomas19a} (theorems \ref{thm:thomas_upper_bound} and \ref{thm:thomas_lower_bound}) arise as a special case of Theorem \ref{thm:cvar_bound_v2}, thereby providing an interpretation for existing CVaR bounds that are otherwise difficult to interpret.
\begin{theorem}\label{thm:cvar_bound_v2}
	Let $X$ and $Y$ be random variables and $\epsilon \in [0, 1]$. 
	\begin{enumerate}
		\item \textbf{Upper Bound:} assume that $P(X\leq b_X)=1, P(Y\leq b_Y)=1$ and $\forall z\in \mathbb{R},F_Y(z)-F_X(z)\leq \epsilon$, then 
		\begin{enumerate}
			\item If $\alpha > \epsilon$, \begin{equation}\label{eq:general_cvar_upper_bound_1}
				\CVaR_\alpha(X) \leq \frac{\epsilon}{\alpha}\max(b_X,b_Y) + (1 - \frac{\epsilon}{\alpha})\CVaR_{\alpha-\epsilon}(Y).
			\end{equation}
			
			\item If $\alpha \leq \epsilon$, then $\CVaR_\alpha(X)\leq \max(b_X,b_Y)$.
		\end{enumerate}
		
		\item \textbf{Lower Bound:} If $P(X\geq a_X)=1, P(Y\geq a_Y)=1$ and $\forall z\in \mathbb{R},F_X(z)-F_Y(z)\leq \epsilon$, then
		\begin{enumerate}
			\item If $\alpha+\epsilon \leq 1$,  \begin{equation}\label{eq:general_cvar_lower_bound_1}
				\CVaR_{\alpha}(X)\geq (1+\frac{\epsilon}{\alpha})\CVaR_{\alpha+\epsilon}(Y)-\frac{\epsilon}{\alpha} \CVaR_{\epsilon}(Y)
			\end{equation}
			
			\item If $\alpha+\epsilon > 1$ and $a_{min} = \min(a_X, a_Y)$, \begin{equation}
				\CVaR_\alpha(X) \geq \frac{1}{\alpha}\big(\mathbb{E}[Y] - \epsilon \CVaR_\epsilon(Y) + (\alpha+\epsilon-1)a_{min}\big)
			\end{equation}
		\end{enumerate}
	\end{enumerate}
\end{theorem}
\begin{proof}
	The proof is available in Appendix \ref{prf:cvar_bound_v2}.
\end{proof}
The parameter $\epsilon$ quantifies the discrepancy between the random variables $X$ and $Y$, and is defined as a uniform bound on the difference between their cumulative distribution functions. By construction, $\epsilon$ takes values in the interval $[0,1]$. In the limiting case where $\epsilon = 1$, $Y$ provides no information about $X$. In this setting, the bounds provided by Theorem~\ref{thm:cvar_bound_v2} reduce to trivial bounds on $\CVaR_\alpha(X)$ involving the essential supremum of both $X$ and $Y$, rendering them ineffective for practical use. When $\epsilon = 0$, the cumulative distribution functions of $X$ and $Y$ are identical, and thus the bound becomes exact, yielding $\CVaR_\alpha(X) = \CVaR_\alpha(Y)$. 
\begin{theorem}\label{thm:cvar_bound_convergence}
	Under the definitions of $X$, $Y$, $\epsilon$, and $\alpha$ specified in Theorem~\ref{thm:cvar_bound_v2}, the lower and upper bounds established therein converge to $\CVaR_\alpha(X)$ as $\epsilon \to 0$.
\end{theorem}
\begin{proof}
	The proof is available in Appendix \ref{prf:cvar_bound_convergence}.
\end{proof}
For both the upper and lower bounds to be informative, Theorem~\ref{thm:cvar_bound_v2} shows that $\alpha > \epsilon$ and $\alpha + \epsilon \leq 1$ must hold, thereby specifying the required relation between the distributional discrepancy $\epsilon$ and the CVaR confidence level $\alpha$. Under these conditions, the bounding problem reduces to \eqref{eq:general_cvar_upper_bound_1} and \eqref{eq:general_cvar_lower_bound_1}. For $\alpha > \epsilon$, the upper bound on $\CVaR_\alpha(X)$ is a weighted average of the $\CVaR$ of $Y$ at a confidence level shifted by the distributional discrepancy $\epsilon$ and the maximum of the supports of $X$ and $Y$, where the weights are proportional to the amount of distributional discrepancy. For $\alpha+\epsilon \leq 1$, the bound in \eqref{eq:general_cvar_lower_bound_1} satisfies
\begin{equation}
	\begin{aligned}
		\CVaR_{\alpha}(X) &\geq \CVaR_{\alpha+\epsilon}(Y) 
		+ \frac{\epsilon}{\alpha}(\CVaR_{\alpha+\epsilon}(Y)- \CVaR_{\epsilon}(Y)).
	\end{aligned}
\end{equation}
The lower bound corresponds to the $\CVaR$ of $Y$ evaluated at a confidence level adjusted by the distributional discrepancy $\epsilon$, augmented by a correction term that is proportional to the distributional discrepancy.

Theorem~\ref{thm:cvar_bound_v2} assumes that the parameter $\epsilon$ provides an upper bound on the pointwise difference between the cumulative distribution functions $F_X$ and $F_Y$. As illustrated in Figure~\ref{fig:eps_bound}, this bound is particularly conservative in the vicinity of $x = -1$, where the actual discrepancy between $F_X$ and $F_Y$ is significantly smaller than the global bound $\epsilon$. Ideally, a tighter bound on $F_X(x)$ would allow for variation with respect to $x$, rather than relying on a uniform constant. Specifically, one seeks a pointwise bound of the form $F_X(x) \leq F_Y(x) + g(x)$ for some non-negative function $g$, as illustrated in Figure~\ref{fig:g_bound}. In this paper, we defer the study of specific choices of the function $g$ to future work, and instead provide general assumptions under which a feasible tighter bound can be established using such a function.

\begin{figure*}[htbp] 
	\centering
	\subfloat[]{%
		\includegraphics[width=0.45\textwidth]{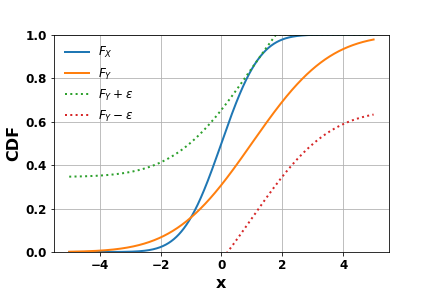}%
		\label{fig:eps_bound}
	}%
	\hfill
	\subfloat[]{%
		\includegraphics[width=0.45\textwidth]{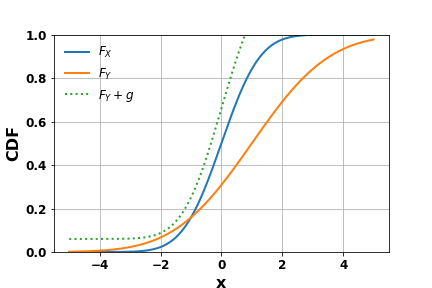}%
		\label{fig:g_bound}
	}
	\caption{Illustrations of bounds on $F_X(x)$.
		(a) Bounds from Theorem~\ref{thm:cvar_bound_v2}.
		(b) Bounds from Theorem~\ref{thm:tight_cvar_lower_bound}, where $g$ depends on $x$ and yields a tighter result than $F_Y(x)+\epsilon$.}
	\label{fig:all_bounds}
\end{figure*}

\begin{theorem}\label{thm:tight_cvar_lower_bound}
	(Tighter CVaR Lower Bound) Let $\alpha \in (0,1)$, $X$ and $Y$ be random variables. Define the random variable $Y^L$ such that $F_{Y^L}(y) \triangleq \min(1, F_Y(y) + g(y))$ for $g:\mathbb{R}\rightarrow [0, \infty)$. Assume 
	$\lim_{x \rightarrow -\infty}g(x)=0$, $g$ is continuous from the right and monotonic increasing.
	If $\forall x\in \mathbb{R}, F_X(x)\leq F_Y(x)+g(x)$, then $F_{Y^L}$ is a CDF and $\CVaR_\alpha(Y^L)\leq \CVaR_\alpha(X).$
\end{theorem}
\begin{proof}
	The proof is available in Appendix \ref{prf:tight_cvar_lower_bound}.
\end{proof}
Theorem \ref{thm:tight_cvar_lower_bound} defines a random variable $Y^L$, constructed from $Y$ and the function $g$, such that the distributional discrepancy between $Y^L$ and $X$ is determined explicitly by $g$ rather than being uniformly bounded as in Theorem \ref{thm:cvar_bound_v2}. In addition, it offers a criterion for determining whether a given function $g$ can be used to derive a lower bound on the CVaR. This bound extends the lower bound established in Theorem \ref{thm:cvar_bound_v2}, which is obtained when $g(x)$ is constant and equal to $\epsilon$ for all $x \in \mathbb{R}$. Intuitively, $F_{Y^L}$ is the largest valid CDF that remains within the pointwise discrepancy band $[F_Y(y),\, F_Y(y) + g(y)]$; since $F_X \leq F_{Y^L}$ by construction, monotonicity of CVaR yields the lower bound. Figure~\ref{fig:all_bounds}(c) illustrates this tighter construction.

Note that in Theorem \ref{thm:tight_cvar_lower_bound}, the function $g$ is assumed to be non-decreasing and right-continuous. If one assumes only that $|F_X(x) - F_Y(x)| \leq g(x)$ for some function $g : \mathbb{R} \to [0, \infty)$, which is not necessarily monotonic or continuous, then the most general form of the CVaR bounds is given by
\begin{equation}\label{eq:general_g_cvar_lower_bound}
	\begin{aligned}
		&\CVaR_\alpha(X)=\frac{1}{\alpha}\int_{1-\alpha}^1 \inf\{z\in \mathbb{R}:F_X(z)\geq \tau\}d\tau 
		\geq \frac{1}{\alpha}\int_{1-\alpha}^1 \inf\{z\in \mathbb{R}:F_Y(z)+g(z)\geq \tau\}d\tau,
	\end{aligned}
\end{equation}
\begin{equation}\label{eq:general_g_cvar_upper_bound}
	\begin{aligned}
		&\CVaR_\alpha(X)=\frac{1}{\alpha}\int_{1-\alpha}^1 \inf\{z\in \mathbb{R}:F_X(z)\geq \tau\}d\tau 
		\leq \frac{1}{\alpha}\int_{1-\alpha}^1 \inf\{z\in \mathbb{R}:F_Y(z)-g(z)\geq \tau\}d\tau.
	\end{aligned}
\end{equation}
Another option is to specify the distributional discrepancy through the density functions underlying the cumulative distribution functions. Let $f_x$ and $f_y$ be the probability density functions of $X$ and $Y$, respectively, and let $h : \mathbb{R} \to [0, \infty)$ describe the pointwise discrepancy between them. Theorem \ref{thm:lower_cvar_bound_using_density_bound} specifies conditions on the function $h$ under which a lower bound for the CVaR of $X$ can be obtained. A key advantage is that this bound takes the form of the CVaR of a random variable, enabling its estimation with performance guarantees via CVaR concentration bounds given in Theorems \ref{thm:brown_bounds}, \ref{thm:thomas_lower_bound}, \ref{thm:thomas_upper_bound} and Theorem \ref{thm:ecdf_cvar_bound}.

\begin{theorem}\label{thm:lower_cvar_bound_using_density_bound}
	Let $\alpha \in (0,1)$, $X$ and $Y$ be random variables. Define $h:\mathbb{R}\rightarrow [0,\infty)$ to be a continuous function, $g(z)\triangleq\int_{-\infty}^z h(x)dx$ and $Y^L$ to be a random variable such that $F_{Y^L}(y)\triangleq\min(1, F_{Y}(y) + g(y))$. If $\lim_{z\rightarrow -\infty} g(z)=0$ and $\forall x\in \mathbb{R},f_x(x)\leq f_y(x) + h(x)$, then $F_{Y^L}$ is a CDF and $\CVaR_\alpha(Y^L)\leq \CVaR_\alpha(X)$.
\end{theorem}
\begin{proof}
	The proof is available in Appendix \ref{prf:lower_cvar_bound_using_density_bound}.
\end{proof}
Theorem \ref{thm:lower_cvar_bound_using_density_bound} is obtained as a corollary of Theorem \ref{thm:tight_cvar_lower_bound} by defining a function $g$, as illustrated in Figure \ref{fig:g_bound}, in terms of the function $h$. This construction demonstrates that the function $g$ in Theorem \ref{thm:tight_cvar_lower_bound} can be generated from a broad class of density discrepancy functions.

\subsection{Concentration Inequalities}\label{sec:con_ineq}
In this section, we derive concentration inequalities for $\CVaR_\alpha(X)$ based on samples drawn from an auxiliary random variable $Y$. A notable special case of these inequalities arises when $Y$ is taken to follow the empirical cumulative distribution function (ECDF) of $X$. 

Let $X_1, \ldots, X_n \overset{i.i.d.}{\sim} F_X$, where $F_X$ is the CDF of a random variable $X$. The ECDF based on these samples is defined by $\hat{F}_X(x)=\frac{1}{n} \sum_{i=1}^n 1_{X_i\leq x}$,
for $x\in \mathbb{R}$. Let $E_{\hat{F}_X}[X]$ denote the expectation with respect to $\hat{F}_X$, and let $C_\alpha^{\hat{F}_X}$ denote the CVaR computed under the empirical distribution $\hat{F}_X$.

\begin{theorem}
	\label{thm:ecdf_cvar_bound}
	Let $X$ be a random variable, $\alpha\in (0, 1],\delta \in (0, 0.5), \epsilon=\sqrt{\ln(1/\delta)/(2n)}$. Let $X_1, \dots, X_n \overset{iid}{\sim} F_X$ be random variables that define the ECDF $\hat{F}_X$.
	\begin{enumerate}
		\item \textbf{Upper Bound:} If $P(X\leq b)=1$, then
		\begin{enumerate}
			\item If $\alpha > \epsilon$ then $P(\CVaR_\alpha(X) \leq (1-\frac{\epsilon}{\alpha})C_{\alpha-\epsilon}^{\hat{F}_X} + \frac{\epsilon}{\alpha}b)> 1-\delta$.
			
			\item If $\alpha\leq \epsilon$ then $\CVaR_\alpha(X)\leq b$
		\end{enumerate}
		
		\item \textbf{Lower Bound:} If $P(X\geq a)=1$, then 
		\begin{enumerate}
			\item If $\alpha+\epsilon < 1$, then $P(\CVaR_\alpha(X)\geq (1+\frac{\epsilon}{\alpha})C_{\alpha+\epsilon}^{\hat{F}_X}-\frac{\epsilon}{\alpha}C_{\epsilon}^{\hat{F}_X})> 1-\delta$.
			
			\item If $\alpha+\epsilon \geq 1$, then $P(\CVaR_\alpha(X) \geq \frac{1}{\alpha}[(\alpha+\epsilon - 1)a + E_{\hat{F}_X}[X] - \epsilon C_\epsilon^{\hat{F}_X}]) > 1-\delta$.
		\end{enumerate}
	\end{enumerate}
\end{theorem}
\begin{proof}
	The proof is available in Appendix \ref{prf:ecdf_cvar_bound}.
\end{proof}
Theorem \ref{thm:ecdf_cvar_bound} provides concentration inequalities for $\CVaR_\alpha(X)$ in a form that enables the user to specify a desired bound consistency level $\delta$, which determines the probability that the bound holds. This result follows as a corollary of Theorem \ref{thm:cvar_bound_v2}, in which the auxiliary random variable $Y$ is instantiated as the ECDF of $X$, whereas $X$ denotes the underlying true random variable, which is inaccessible in practice. The distributional discrepancy required by Theorem \ref{thm:cvar_bound_v2}, denoted by $\epsilon$, is controlled in Theorem \ref{thm:ecdf_cvar_bound} via the Dvoretzky–Kiefer–Wolfowitz (DKW) inequality \cite{DKW_inequality}. The DKW inequality ensures that the supremum distance between the true CDF and the ECDF converges to zero at a rate of order $1/\sqrt{n}$ as the number of samples $n$ increases. Corollary \ref{thm:x_concentration_bound_asymptotic_convergence} establishes the asymptotic convergence of the bounds given in Theorem \ref{thm:ecdf_cvar_bound} as the sample size tends to infinity.
\begin{corollary}\label{thm:x_concentration_bound_asymptotic_convergence}
	Let $X$ be a random variable, $\alpha\in (0, 1],\delta \in (0, 0.5), \epsilon=\sqrt{\ln(1/\delta)/(2n)}, a\in \mathbb{R}, b\in \mathbb{R}$. Let $X_1, \dots, X_n \overset{iid}{\sim} F_X$ be random variables that define the ECDF $\hat{F}_X$. Denote by $U(n)$ and $L(n)$ the upper and lower bounds respectively from Theorem \ref{thm:ecdf_cvar_bound}, where $n$ is the number of samples, then
	\begin{enumerate}
		\item If $P(X\leq b)=1$, then $\lim_{n\rightarrow \infty} U(n)\overset{a.s.}{=}\CVaR_\alpha(X)$.
		\item If $P(X\geq a)=1$, then $\lim_{n\rightarrow \infty} L(n)\overset{a.s.}{=}\CVaR_\alpha(X)$.
	\end{enumerate}
	where a.s.~denotes almost sure convergence.
\end{corollary}
\begin{proof}
	The proof is available in Appendix \ref{prf:x_concentration_bound_asymptotic_convergence}.
\end{proof}
The concentration bounds established by \cite{pmlr-v97-thomas19a} coincide with those given in Theorem \ref{thm:ecdf_cvar_bound}, rendering the results of \cite{pmlr-v97-thomas19a} a special case of Theorem \ref{thm:ecdf_cvar_bound}. This equivalence arises because both Theorem \ref{thm:ecdf_cvar_bound} and \cite{pmlr-v97-thomas19a} derive concentration bounds for $\CVaR_\alpha(X)$ by constructing an alternative CDF that stochastically dominates the true distribution $F_X$, employing the DKW inequality to control the discrepancy. The principal distinction between the two results lies in the formulation of the $\CVaR$ bound: \cite{pmlr-v97-thomas19a} express the bound through a sum of reweighted order statistics (theorems \ref{thm:thomas_upper_bound} and \ref{thm:thomas_lower_bound}), resulting in a more intricate form, whereas Theorem \ref{thm:ecdf_cvar_bound} presents a more interpretable bound in terms of $\CVaR$. The interpretability of these bounds constitutes a contribution of this paper.

Theorem \ref{thm:ecdf_y_bound_cvar_x} provides concentration inequalities for the theoretical $\CVaR_\alpha(X)$ based on samples drawn from an auxiliary random variable $Y$, assuming only a bound on the distributional discrepancy between $X$ and $Y$.
\begin{theorem}
	\label{thm:ecdf_y_bound_cvar_x}
	Let $X$ and $Y$ be random variables, $\epsilon \in [0, 1], \delta\in (0,0.5)$, and $\eta = \sqrt{\ln(1/\delta)/(2n)}, \epsilon'=\min(\epsilon+\eta, 1)$. Let $Y_1, \dots, Y_n$ be independent and identically distributed samples from $F_Y$, and denote by $\hat{F}_Y$ the associated empirical cumulative distribution function.
	\begin{enumerate}
		\item \textbf{Upper Bound:} If $\forall z\in \mathbb{R},F_Y(z)-F_X(z)\leq \epsilon$ and $P(X\leq b_X)=1, P(Y\leq b_Y)=1$, then
		
		\begin{enumerate}
			\item If $\alpha > \epsilon'$ then $P\Big(\CVaR_\alpha(X) \leq \frac{\epsilon'}{\alpha}\max(b_X,b_Y) 
			+ (1 - \frac{\epsilon'}{\alpha})\CVaR_{\alpha-\epsilon'}^{\hat{F}_Y}\Big)>1-\delta$. 
			
			\item If $\alpha \leq \epsilon'$, then $\CVaR_\alpha(X)\leq \max(b_X,b_Y)$
		\end{enumerate}
		
		\item \textbf{Lower Bound:} If $\forall z\in \mathbb{R},F_X(z)-F_Y(z)\leq \epsilon$ and $P(X\geq a_X)=1,P(Y\geq a_Y)=1$, then \begin{enumerate}
			\item If $\alpha+\epsilon' \leq 1$, then \begin{equation}
				\begin{aligned}
					P\Big(\CVaR_{\alpha}(X)&\geq (1+\frac{\epsilon'}{\alpha})\CVaR_{\alpha+\epsilon'}^{\hat{F}_Y} 
					-\frac{\epsilon'}{\alpha} \CVaR_{\epsilon'}^{\hat{F}_Y}\Big)>1-\delta.
				\end{aligned}
			\end{equation}
			
			\item If $\alpha+\epsilon' > 1$, then \begin{equation}
				\begin{aligned}
					P\Big(\CVaR_\alpha(X) &\geq \frac{1}{\alpha}\big(\mathbb{E}_{\hat{F}_Y}[Y] - \epsilon' \CVaR_{\epsilon'}^{\hat{F}_Y}
					+ (\alpha+\epsilon'-1)a_{min}\big) \Big) > 1-\delta.
				\end{aligned}
			\end{equation}
		\end{enumerate}
	\end{enumerate}
\end{theorem}
\begin{proof}
	The proof is available in Appendix \ref{prf:ecdf_y_bound_cvar_x}.
\end{proof}
The parameter $\epsilon$ in Theorem \ref{thm:ecdf_y_bound_cvar_x} captures the distributional discrepancy between $X$ and $Y$, consistent with its role in Theorem \ref{thm:cvar_bound_v2}. Additionally, the theorem introduces $\epsilon'$ to represent the discrepancy between $F_X$ and $\hat{F}_Y$, where $\hat{F}_Y$ denotes the empirical CDF of $Y$ constructed from the sample $\{Y_i\}_{i=1}^n$. The parameter $\eta$ accounts for the additional distributional discrepancy beyond $\epsilon$, and captures the estimation error in approximating $F_Y$ using the empirical sample $\{Y_i\}_{i=1}^n$. By combining Theorem \ref{thm:cvar_bound_v2} with the DKW inequality \cite{DKW_inequality}, we obtain probabilistic guarantees for the estimated bounds.

\section{Value Function Bounds for Static CVaR Value Function}\label{sec:cvar_bounds}
In this section, we leverage the CVaR bounds from the previous section to establish bounds between the simplified and original value functions as defined in \eqref{Eq:cvar_value_function_def}, where the simplification is considered in a general form with respect to the belief transition model. We then derive bounds for the specific case in which the simplification applies to the observation model. The implications of these bounds for policy evaluation and open-loop POMDP planning will be examined in the Experiments section.

\subsection{Bounds for a General Belief Transition Model Simplification}\label{sec:bounds_for_simplified_belief_model}
In this section, we establish bounds on the difference between the CDFs of the returns, computed with respect to the simplified and original belief transition models. These bounds are then used to bound the original value function in terms of the simplified value function.

Theorem \ref{thm:simplification_bound_over_belief_distribution} demonstrates that the difference between the CDFs of the returns, computed with respect to the simplified and original belief transition models, can be characterized in terms of the difference between the simplified and original belief transition models. 
\begin{theorem}\label{thm:simplification_bound_over_belief_distribution} 	Let $P$ and $P_s$ denote the probability measures induced by the original and simplified belief-transition models, respectively. Denote by $\mathbb{E}^s$ the expectation taken with respect to the simplified belief-transition model. Then, the following holds:
	\begin{equation}\label{eq:simplification_bound_over_belief_distribution}
		\begin{aligned}
			\sup_{l\in \mathbb{R}}|P(R_{k:T} \leq l|b_k, a_k, \pi) \! -\! P_s(R_{k:T} \leq l|b_k, a_k, \pi)| \!
			\leq \!\! \! \sum_{t=k}^{T-1} \mathbb{E}^s[\Delta^s(b_t, a_t)|b_k, a_k],
		\end{aligned}
	\end{equation}
	where $\Delta^s$ is the TV distance that is defined by
	\begin{equation}\label{eq:TVdistBel}
		\begin{aligned}
			\Delta^s(b_{t-1},a_{t-1})
			\triangleq \int_{b_t\in B} |P(b_t|b_{t-1},a_{t-1})-P_s(b_t|b_{t-1},a_{t-1})|db_t.
		\end{aligned}
	\end{equation}
\end{theorem}
\begin{proof}
	The proof is available in Appendix \ref{prf:simplification_bound_over_belief_distribution}.
\end{proof}
The expression $\mathbb{E}^s[\Delta^s(b_t, a_t) \mid b_k, a_k]$ quantifies the expected one-step distributional discrepancy at time $t$ between the transition dynamics of the simplified and original belief-MDPs, conditioned on the planning process being initiated at time $k$. Summing this quantity over the horizon accumulates the total distributional discrepancy, leading to the bound in \eqref{eq:simplification_bound_over_belief_distribution}, which in turn bounds the cumulative distribution function of the return across the horizon. This accumulation mechanism is illustrated in Figure~\ref{fig:epsilon_accumulation}.

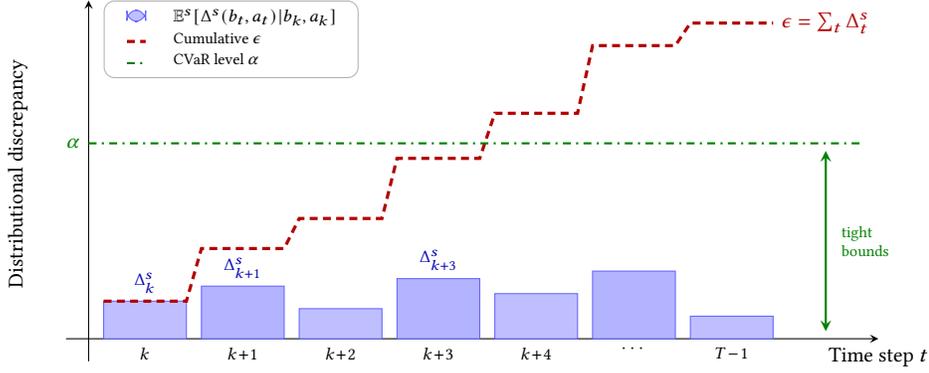
\begin{figure}[t]
	\centering
	\begin{tikzpicture}[>=stealth, font=\small]
		\draw[->] (-0.3,0) -- (10.5,0) node[below] {Time step $t$};
		\draw[->] (0,-0.3) -- (0,4.5);

		\node[rotate=90, anchor=south] at (-0.7,2.2) {Distributional discrepancy};

		\def\hk{0.5}    
		\def\hkk{0.7}   
		\def\hkkk{0.4}  
		\def\hkfour{0.8} 
		\def\hkfive{0.6} 
		\def\hksix{0.9}  
		\def\hkseven{0.3} 

		\def\bw{1.1}

		\fill[blue!25, draw=blue!60] (0.2,0) rectangle (0.2+\bw, \hk);
		\fill[blue!30, draw=blue!60] (0.2+\bw+0.2, 0) rectangle (0.2+2*\bw+0.2, \hkk);
		\fill[blue!25, draw=blue!60] (0.2+2*\bw+0.4, 0) rectangle (0.2+3*\bw+0.4, \hkkk);
		\fill[blue!30, draw=blue!60] (0.2+3*\bw+0.6, 0) rectangle (0.2+4*\bw+0.6, \hkfour);
		\fill[blue!25, draw=blue!60] (0.2+4*\bw+0.8, 0) rectangle (0.2+5*\bw+0.8, \hkfive);
		\fill[blue!30, draw=blue!60] (0.2+5*\bw+1.0, 0) rectangle (0.2+6*\bw+1.0, \hksix);
		\fill[blue!25, draw=blue!60] (0.2+6*\bw+1.2, 0) rectangle (0.2+7*\bw+1.2, \hkseven);

		\node[below, font=\scriptsize] at (0.2+0.5*\bw, 0) {$k$};
		\node[below, font=\scriptsize] at (0.2+1.5*\bw+0.2, 0) {$k\!+\!1$};
		\node[below, font=\scriptsize] at (0.2+2.5*\bw+0.4, 0) {$k\!+\!2$};
		\node[below, font=\scriptsize] at (0.2+3.5*\bw+0.6, 0) {$k\!+\!3$};
		\node[below, font=\scriptsize] at (0.2+4.5*\bw+0.8, 0) {$k\!+\!4$};
		\node[below, font=\scriptsize] at (0.2+5.5*\bw+1.0, 0) {$\cdots$};
		\node[below, font=\scriptsize] at (0.2+6.5*\bw+1.2, 0) {$T\!-\!1$};

		\node[above, font=\scriptsize, blue!70!black] at (0.2+0.5*\bw, \hk) {$\Delta^s_k$};
		\node[above, font=\scriptsize, blue!70!black] at (0.2+1.5*\bw+0.2, \hkk) {$\Delta^s_{k+1}$};
		\node[above, font=\scriptsize, blue!70!black] at (0.2+3.5*\bw+0.6, \hkfour) {$\Delta^s_{k+3}$};

		\draw[red!70!black, very thick, densely dashed]
			(0.2, 0.5) -- (0.2+\bw, 0.5)
			-- (0.2+\bw+0.2, 1.2) -- (0.2+2*\bw+0.2, 1.2)
			-- (0.2+2*\bw+0.4, 1.6) -- (0.2+3*\bw+0.4, 1.6)
			-- (0.2+3*\bw+0.6, 2.4) -- (0.2+4*\bw+0.6, 2.4)
			-- (0.2+4*\bw+0.8, 3.0) -- (0.2+5*\bw+0.8, 3.0)
			-- (0.2+5*\bw+1.0, 3.9) -- (0.2+6*\bw+1.0, 3.9)
			-- (0.2+6*\bw+1.2, 4.2) -- (0.2+7*\bw+1.2, 4.2);

		\node[right, red!70!black, font=\small] at (0.2+7*\bw+1.2, 4.2) {$\epsilon = \sum_t \Delta^s_t$};

		\draw[green!50!black, thick, dashdotted] (0, 2.6) -- (10.3, 2.6);
		\node[left, green!50!black, font=\small] at (0, 2.6) {$\alpha$};

		\draw[<->, green!50!black, thick] (9.8, 0.1) -- (9.8, 2.5);
		\node[right, green!50!black, font=\scriptsize, text width=1.5cm, align=left] at (9.9, 1.3) {tight\\bounds};

		\node[anchor=north west, fill=white, draw=black!30, rounded corners, inner sep=3pt, font=\scriptsize] at (0.2,4.5) {
			\begin{tabular}{cl}
				\tikz\fill[blue!25, draw=blue!60] (0,0) rectangle (0.25,0.15); & $\mathbb{E}^s[\Delta^s(b_t,a_t)|b_k,a_k]$ \\[1pt]
				\tikz\draw[red!70!black, very thick, densely dashed] (0,0.08) -- (0.25,0.08); & Cumulative $\epsilon$ \\[1pt]
				\tikz\draw[green!50!black, thick, dashdotted] (0,0.08) -- (0.25,0.08); & CVaR level $\alpha$
			\end{tabular}
		};
	\end{tikzpicture}
	\caption{Accumulation of distributional discrepancy over the planning horizon. Each bar represents the expected one-step TV distance $\mathbb{E}^s[\Delta^s(b_t, a_t) \mid b_k, a_k]$ at time $t$. The cumulative sum $\epsilon = \sum_{t=k}^{T-1} \Delta^s_t$ (dashed) governs the bound quality: when $\epsilon < \alpha$ (below the dash-dotted line), the CVaR bounds from Theorem~\ref{thm:uniform_lower_and_upper_bounds_for_v_and_q} remain informative; when $\epsilon \geq \alpha$, the bounds reduce to trivial extrema.}
	\label{fig:epsilon_accumulation}
\end{figure}

By leveraging Theorem \ref{thm:cvar_bound_v2} and combining it with the supremum norm bounds for the simplified and original return CDFs in \eqref{eq:simplification_bound_over_belief_distribution}, we derive the following bounds on the value functions.

\begin{theorem}\label{thm:uniform_lower_and_upper_bounds_for_v_and_q}
	Let $\delta_{d_{min}}\in [0, 1], \delta_{d_{max}}\in [0, 1]$ and denote 
	\begin{equation}\label{eq:epsilon}
		\epsilon \triangleq \epsilon(b_k, a_k) \triangleq \min(\sum_{i=k}^{T-1} \mathbb{E}^s[\Delta^s(b_i, a_i)|b_k, a_k,\pi], 1),
	\end{equation}
	\begin{enumerate}
		\item \textbf{Upper Bound:} assume that $P(R_{k:T}\leq d_{max})>1-\delta_{d_{max}}$ and $P_s(R_{k:T}\leq d_{max})>1-\delta_{d_{max}}$, then \begin{enumerate}
			\item If $\epsilon<\alpha$, then
			$U_s\triangleq (1-\frac{\epsilon}{\alpha}) Q^\pi_{M_s}(b_k,a_k,\alpha-\epsilon) + \frac{\epsilon}{\alpha}d_{max}$
			\item If $\epsilon \geq \alpha$ then $U_s\triangleq d_{max}$
		\end{enumerate}
		\item \textbf{Lower Bound:} assume that $P(R_{k:T}\geq d_{min})>1-\delta_{d_{min}}$ and $P_s(R_{k:T}\geq d_{min})>1-\delta_{d_{min}}$, then \begin{enumerate}
			\item If $\epsilon + \alpha < 1$ then 
			$
			L_s\triangleq (1+\frac{\epsilon}{\alpha}) Q^\pi_{M_s}(b_k,a_k,\alpha+\epsilon) - \frac{\epsilon}{\alpha} Q^\pi_{M_s}(b_k,a_k, \epsilon)
			$
			\item If $\epsilon + \alpha \geq 1$ then 
			$
			L_s\triangleq\frac{1}{\alpha}[-(\alpha + \epsilon - 1)
			d_{min}+Q^\pi_{M_s}(b_k,a_k)-\epsilon Q^\pi_{M_s}(b_k,a_k,\epsilon)]
			$
		\end{enumerate}
	\end{enumerate}
	Then $P(L_s \leq Q_M^\pi(b_k, a_k,\alpha)) > 1-\delta_{d_{min}}$ and $P(Q_M^\pi(b_k, a_k,\alpha)\leq U_s)> 1-\delta_{d_{max}}$.
\end{theorem}
\begin{proof}
	The proof is available in Appendix \ref{prf:uniform_lower_and_upper_bounds_for_v_and_q}.
\end{proof}
The bounds established in  Theorem~\ref{thm:uniform_lower_and_upper_bounds_for_v_and_q} provide upper and lower bounds on the original value function in terms of the simplified value function, evaluated at an adjusted confidence level that accounts for the distributional discrepancy $\epsilon$ between the simplified and original belief transition models.

\subsection{Bounds for Observation Model Simplification}\label{sec:simp_obs_bound}
Simplification of the observation model is a special case of simplification of the belief transition model. Despite that, computation of the belief transition model is still an open question in the field, and therefore the bounds from Theorem \ref{thm:uniform_lower_and_upper_bounds_for_v_and_q} should be adjusted for this specific simplification setting. 

Let \( q_z \) denote the simplified observation model associated with the POMDP \( M_s \), and let \( p_z \) denote the original observation model associated with the POMDP \( M \). The bounds for observation model simplification are exhibited in Theorem \ref{thm:observation_model_simplification_bounds}.
\begin{theorem}\label{thm:observation_model_simplification_bounds}
	In the case where the simplified and original observation models are denoted by $q_z$ and $p_z$ respectively, it holds that
	\begin{equation}\label{eq:cdf_diff_obs_simp_bound}
		\begin{aligned}
			\sup_{l\in \mathbb{R}}|P(R_{k:T} \leq l|b_k, a_k,\pi) \! -\! P_s(R_{k:T} \leq l|b_k, a_k,\pi)| \!
			\leq \!\! \! \sum_{t=k}^{T-1} \mathbb{E}^s[\Delta^s(x_{t+1})|b_k, a_k],
		\end{aligned}
	\end{equation}
	where $\Delta^s$ is the TV distance that is defined by 
	\begin{equation}\label{eq:TVdistState}
		\Delta^s(x_{t}) \triangleq \int_{z_{t}\in Z} |p_z(z_{t}|x_t)-q_z(z_t|x_t)|dz_t.
	\end{equation}
\end{theorem}
\begin{proof}
	The proof is available in Appendix \ref{prf:observation_model_simplification_bounds}.
\end{proof}
Theorem~\ref{thm:observation_model_simplification_bounds} establishes a state-dependent upper bound on the sup-norm difference between the CDFs of the return. This bound is amenable to practical computation. The state $x_{t+1}$ appearing in~\eqref{eq:cdf_diff_obs_simp_bound} can be sampled from the distribution
$P(\cdot \mid x_t, \pi_t) \prod_{i=k+1}^{t-1} P(x_{i+1} \mid x_i, \pi_i) P(x_k \mid b_k)$,
where $\pi_i$ denotes the action selected by policy $\pi$ at time $i$. 

By leveraging Theorem~\ref{thm:uniform_lower_and_upper_bounds_for_v_and_q} with Theorem~\ref{thm:observation_model_simplification_bounds}, we obtain Corollary~\ref{thm:cvar_tv_dist_bound_simp_obs}, which provides computable CVaR bounds based on state-dependent total variation distance.
\begin{corollary}\label{thm:cvar_tv_dist_bound_simp_obs}
	The bound from Theorem \ref{thm:uniform_lower_and_upper_bounds_for_v_and_q} holds for $\epsilon(b_k, a_k)=\sum_{t=k}^{T-1} \mathbb{E}^s[\Delta^s(x_{t+1})|b_k, a_k]$.
\end{corollary}
\begin{proof}
	The proof is available in Appendix \ref{prf:cvar_tv_dist_bound_simp_obs}.
\end{proof}

\section{Online Bound Estimation}\label{sec:online_bound_estimation}
In this section, we derive practical estimators corresponding to the theoretical bounds presented in Section~\ref{sec:cvar_bounds}, and describe how these bounds can be efficiently computed within the context of online policy evaluation and open-loop planning. All estimators are constructed from samples generated by particle-based belief estimators drawn under the simplified belief-transition distribution. Specifically, we develop estimators for the action–value function, the distributional discrepancy $\epsilon$ appearing in Corollary~\ref{thm:cvar_tv_dist_bound_simp_obs}, and the minimum and maximum values of the return, thereby enabling computation of the complete bound stated in Corollary~\ref{thm:cvar_tv_dist_bound_simp_obs}. Moreover, we provide performance guarantees for these estimators, establishing their reliability when computed from a finite number of samples.

\subsection{Action-Value Function Estimator}\label{sec:q_func_estimator}
The action-value function is defined as the CVaR of the return $R_{k:T} \triangleq \sum_{t=k}^T c(b_t, a_t)$, and can be estimated in a straightforward manner by generating sample trajectories of the return and computing the empirical CVaR from these samples. 

To generate return samples, it is necessary to simulate belief trajectories from time $k$ to time $T$. This is achieved using a particle-based belief estimator, where the belief is represented by a weighted set of $N_p \in \mathbb{N}$ particles. Given a particle belief $\bar{b}_k$, provided by the user to represent the agent's belief at time $k$, the agent generates $N_b$ simulated particle-belief trajectories $\{\bar{b}_t^i\}$, where $i \in \{1, \dots, N_b\}$ indexes the trajectory and $t \in \{k, \dots, T\}$ denotes the time step. For each simulated trajectory, a return sample is computed as $\bar{R}^i_{k:T} = \sum_{t=k}^T c(\bar{b}_t^i, \pi(\bar{b}_t^i))$, resulting in a collection of return samples $\{\bar{R}_{k:T}^i\}_{i=1}^{N_b}$. Let $\bar{b}_t^i \triangleq \{(x_t^{i,j}, w_t^{i,j})\}_{j=1}^{N_p}$ denote the $i$-th simulated particle belief at time $t$, where $j$ indexes the particles. Define the normalized weights by $\tilde{w}_t^{i,j} \triangleq \frac{w_t^{i,j}}{\sum_{j=1}^{N_p} w_t^{i,j}}$. Assuming the cost function is state-dependent, the cost associated with belief $\bar{b}_t^i$ and action $a$ is given by
$
c(\bar{b}_t^i, a) = \sum_{j=1}^{N_p} \tilde{w}_t^{i,j} \, c(x_t^{i,j}, a).
$ 
By estimating CVaR according to \eqref{eq:cvar_estimator_sorted_sample}, we get the action-value function estimator
\begin{equation}
	\begin{aligned}
		&\hat{Q}_{M_P}^\pi (\bar{b}_k, a, \alpha) \triangleq \hat{C}(\{R^i\}_{i=1}^{N_b}) 
		&=R^{(N_b)}-\frac{1}{\alpha}\sum_{i=1}^{N_b} (R^{(i)}-R^{(i-1)})\Bigl(\frac{i}{N_b}-(1-\alpha) \Bigr)^+.
	\end{aligned}
\end{equation}
Detailed pseudo-code for the estimation of the action–value function is presented in Algorithm \ref{alg:value_func_eval}.

\begin{algorithm*}[tb]
	\caption{CVaR Value Function Estimation for Policy $\pi$}
	\label{alg:value_func_eval}
	\textbf{Global Parameters}: $N_p, N_b, T, \pi$\\
	\textbf{Input}: Particle belief $\bar{b}$, planning horizon $T$, policy $\pi$, action $a$, risk level $\alpha$, number of belief trajectories $N_b$ and number of belief particles $N_p$.\\
	\textbf{Output}: Estimated value $\hat{V}^\pi(\bar{b})$ or $\hat{Q}^\pi(\bar{b},a)$

	\vspace{4pt}
	\noindent
	\begin{minipage}[t]{0.48\textwidth}
	\begin{algorithmic}[1]
		\Procedure{GENPF}{$\bar{b}, a$}
		\State Sample $x_0$ from $\bar{b}$ with prob.\ $w_i / \sum_i w_i$
		\State $z \gets G(x_0,a)$
		\For{$i = 1$ \textbf{to} $N_p$}
		\State $(x'_i, r_i) \gets G(x_i,a)$
		\State $w'_i \gets w_i \cdot Z(z \mid a, x'_i)$
		\EndFor
		\State $\bar{b'} \gets \{(x'_i,w'_i)\}_{i=1}^{N_p}$
		\State $\rho \gets \frac{\sum_i w_i r_i}{\sum_i w_i}$
		\State \textbf{return} $\bar{b'}, \rho$
		\EndProcedure

		\Procedure{SampleNextBeliefs}{beliefs, $\pi$}
		\State $\text{NextBeliefs} \gets list()$
		\For{$j=1$ \textbf{to} $N_b$}
		\State $\text{NextBelief}\gets \text{GENPF(beliefs[j]},$ $\pi\text{(beliefs[j]))}$
		\State append NextBelief to NextBeliefs
		\EndFor
		\State \textbf{return} NextBeliefs
		\EndProcedure

		\Procedure{GenBeliefTrajectories}{$\bar{b}, \pi, T$}
		\State $\text{trajectories} \gets list()$
		\For{$i=1$ \textbf{to} $N_b$}
		\State \text{append $\bar{b}$ to trajectories}
		\EndFor
		\For{$t = 0$ \textbf{to} $T-1$}
		\State $\text{NextBeliefs} \gets$ SampleNextBeliefs(trajectories[$t$], $\pi$)
		\State \text{append NextBeliefs to trajectories}
		\EndFor
		\State \textbf{return} trajectories
		\EndProcedure
	\end{algorithmic}
	\end{minipage}%
	\hfill
	\begin{minipage}[t]{0.48\textwidth}
	\begin{algorithmic}[1]
		\Procedure{GenerateReturnSample}{$\bar{b}, a, \pi$}
		\State trajectories $\gets$ GenBeliefTrajectories($\bar{b}, \pi, T$)
		\State ReturnSample $\gets$ list()
		\For{$i=1$ \textbf{to} $N_b$}
		\State Return $\gets$ c(trajectories[0][$i$], $a$)
		\For{$t=1$ \textbf{to} $T$}
		\State Return += c(trajectories[$t$][$i$], $\pi$(trajectories[$t$][$i$]))
		\EndFor
		\State append Return to ReturnSample
		\EndFor
		\State \textbf{return} ReturnSample
		\EndProcedure

		\Procedure{CVaREstimate}{samp, $\alpha$}
		\State Sort samp in ascending order as $z_1,\dots,z_n$ with $z_0=0$
		\State \textbf{return} $z_n - \tfrac{1}{\alpha}\sum_{i=0}^{n-1}(z_{i+1}-z_i)\bigl(\tfrac{i}{n}-(1-\alpha)\bigr)^+$
		\EndProcedure

		\Procedure{Estimate$V^\pi$}{$\bar{b},\alpha$}
		\State \textbf{return} \Call{EstimateQ}{$\bar{b},\pi(\bar{b}),\alpha$}
		\EndProcedure

		\Procedure{EstimateQ}{$\bar{b},a,\alpha$}
		\State samp $\gets$ GenerateReturnSample($\bar{b}, a, \pi$)
		\State \textbf{return} \Call{CVaREstimate}{samp,$\alpha$}
		\EndProcedure
	\end{algorithmic}
	\end{minipage}
\end{algorithm*}

\subsection{Online CDF Bound Estimation}\label{sec:cdf_bound_estimation}
The bound on the difference between return simplified and original CDFs, established in Theorem~\ref{thm:observation_model_simplification_bounds} and denoted by $\epsilon$ in Corollary~\ref{thm:cvar_tv_dist_bound_simp_obs}, plays a central role in bounding the discrepancy between the simplified and original action-value functions. Notably, it is the only term in the bound that explicitly captures the impact of the difference between the simplified and original observation models. A challenge arises from the fact that the quantity $\epsilon$ cannot be directly estimated during planning, as its computation requires access to the full observation model $p_z$, which is typically too expensive to evaluate online. For the purpose of online computation of the bound, we adopt the methodology proposed in~\cite{LevYehudi24aaai}, which enables evaluation without real-time access to the original observation model. This is achieved by decoupling the offline sampling phase—conducted using the original observation model—from the online sampling phase, which relies solely on the simplified observation model. Specifically, the offline–online decoupling strategy and the state-level importance-sampling estimator in~\eqref{eq:m_state_act} were introduced by~\cite{LevYehudi24aaai}; the belief-level aggregation, horizon-level accumulation, and the probabilistic guarantees developed in the remainder of this section are novel contributions of the present work. This offline--online decoupling procedure is illustrated in Figure~\ref{fig:offline_online_decoupling}.

Let $Q_0$ be a distribution over the state space, and denote the TV-distance expectation given the previous state and action by $m$. \eqref{eq:m_state_act} shows how the one-step distributional discrepancy bound $m$ can be computed using importance sampling.
\begin{equation}
	\begin{aligned}
		m(x_t, a) &\triangleq \mathbb{E}_{x_{t+1} \sim P(\cdot | x_t, a)}[\Delta^s(x_{t+1})|x_t, a]
		&=\mathbb{E}_{x_{t+1} \sim Q_0}[\frac{P(x_{t+1}|x_t, a)}{Q_0(x_{t+1})}\Delta^s(x_{t+1})|x_t, a].
	\end{aligned}
	\label{eq:m_state_act}
\end{equation} 
Define the belief-dependent one-step distributional discrepancy $m$, and the belief-dependent total distributional discrepancy bound of the CDFs difference at time $t$ by $\epsilon_t$.
\begin{equation}\label{eq:m_belief}
	m(b_t, a) \triangleq \mathbb{E}[m(x_t, a)|b_t, a],
\end{equation}
\begin{equation}
	\epsilon_t(b_t, a) \triangleq \sum_{\tau=t}^{T-1} \mathbb{E}[\Delta^s(x_{\tau+1})|b_t, a] = \sum_{\tau=t}^{T-1} \mathbb{E}[m(b_\tau, a)|b_t, a].
\end{equation}
In practice, a set of states $\{x_n^\Delta\}_{n=1}^{N^\Delta} \overset{\text{i.i.d.}}{\sim} Q_0$ is drawn from a user-specified distribution $Q_0$, and the TV-distance corresponding to each sampled state is computed offline. During online planning, the precomputed TV distance estimates are reweighted via importance sampling, thereby adapting the offline-computed distributional discrepancy to reflect the actual distributional discrepancy encountered by the agent during execution, as described in~\eqref{eq:m_est_state}. By utilizing $\{x_n^\Delta\}$, \eqref{eq:m_state_act} and \eqref{eq:m_belief} can be estimated as follows \begin{equation}\label{eq:m_est_state}
	\hat{m}(x_t, a)=\frac{1}{N^\Delta}\sum_{n=1}^{N^\Delta} \frac{P(x_n^\Delta|x_t, a)}{Q_0(x_n^\Delta)}\Delta^s(x_n^\Delta),
\end{equation}
\begin{equation}
	\hat{m}(\bar{b}_t, a)=\sum_{i=1}^{N_p} \hat{m}(x_t^i, a)\tilde{w}_{t}^i,
\end{equation}
for $\bar{b}_t\triangleq \{(x_t^i, w_t^i)\}$ where $\tilde{w}_t^i$ are the normalized weights. Using the particle-belief trajectories $\{\bar{b}_\tau^i\}_{i=1}^{N_b}$ introduced in Section~\ref{sec:q_func_estimator}, the TV-distance at time $\tau+1$ is estimated by using $\{\bar{b}_\tau^i\}_{i=1}^{N_b}$ as samples from the conditional distribution of $\bar{b}_\tau$ given $\bar{b}_t$, as described below.
\begin{equation}
	\hat{\mathbb{E}}[m(b_\tau, a)|\bar{b}_t, a] = \frac{1}{N_b} \sum_{i=1}^{N_b} \hat{m}(\bar{b}_\tau^{i}, \pi(\bar{b}_\tau^{i})).
\end{equation}
The estimate of $\epsilon$ is then computed directly by
\begin{equation}\label{eq:eps_estimator}
	\hat{\epsilon}(\bar{b}_t, a)=\sum_{\tau=t}^{T-1} \hat{\mathbb{E}}[m(\bar{b}_\tau, \pi_\tau)|\bar{b}_t, a].
\end{equation}

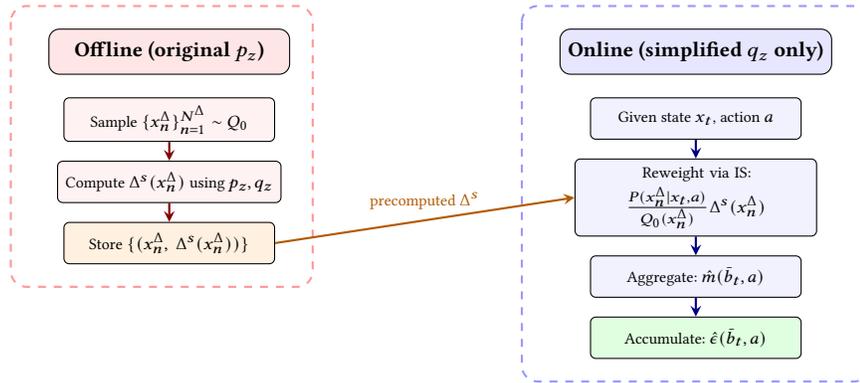
\begin{figure}[t]
	\centering
	\begin{tikzpicture}[
		>=stealth, font=\small,
		phase/.style={rectangle, draw, rounded corners=4pt, minimum width=3.2cm, minimum height=0.6cm, align=center, font=\small\bfseries},
		sbox/.style={rectangle, draw, rounded corners=2pt, minimum width=2.8cm, minimum height=0.55cm, align=center, font=\scriptsize},
		arr/.style={->, thick},
		lbl/.style={font=\scriptsize, midway},
	]
		\node[phase, fill=red!10] (offhdr) at (0, 0) {Offline (original $p_z$)};
		\node[sbox, fill=red!5, below=0.3cm of offhdr] (off1) {Sample $\{x_n^\Delta\}_{n=1}^{N^\Delta} \sim Q_0$};
		\node[sbox, fill=red!5, below=0.25cm of off1] (off2) {Compute $\Delta^s(x_n^\Delta)$ using $p_z, q_z$};
		\node[sbox, fill=orange!12, below=0.25cm of off2, minimum width=2.8cm] (store) {Store $\{(x_n^\Delta,\; \Delta^s(x_n^\Delta))\}$};


		\node[phase, fill=blue!10] (onhdr) at (7, 0) {Online (simplified $q_z$ only)};
		\node[sbox, fill=blue!5, below=0.3cm of onhdr] (on1) {Given state $x_t$, action $a$};
		\node[sbox, fill=blue!5, below=0.25cm of on1, text width=3.0cm] (on2) {Reweight via IS:\\[1pt] $\frac{P(x_n^\Delta|x_t,a)}{Q_0(x_n^\Delta)}\Delta^s(x_n^\Delta)$};
		\node[sbox, fill=blue!5, below=0.25cm of on2] (on3) {Aggregate: $\hat{m}(\bar{b}_t, a)$};
		\node[sbox, fill=green!12, below=0.25cm of on3] (on4) {Accumulate: $\hat{\epsilon}(\bar{b}_t, a)$};

		\draw[red!40, thick, rounded corners=6pt, dashed]
			([xshift=-0.5cm, yshift=0.3cm]offhdr.north west) rectangle ([xshift=0.5cm, yshift=-0.3cm]store.south east);
		\draw[blue!40, thick, rounded corners=6pt, dashed]
			([xshift=-0.5cm, yshift=0.3cm]onhdr.north west) rectangle ([xshift=0.7cm, yshift=-0.3cm]on2.south east |- on4.south);

		\draw[arr, orange!70!black, thick] (store.east) -- (on2.west) node[lbl, above, yshift=1pt, font=\scriptsize] {precomputed $\Delta^s$};

		\draw[arr, red!50!black] (off1) -- (off2);
		\draw[arr, red!50!black] (off2) -- (store);
		\draw[arr, blue!50!black] (on1) -- (on2);
		\draw[arr, blue!50!black] (on2) -- (on3);
		\draw[arr, blue!50!black] (on3) -- (on4);
	\end{tikzpicture}
	\caption{Offline--online decoupling for distributional discrepancy estimation. In the offline phase, states are sampled from a proposal distribution $Q_0$ and the TV distances $\Delta^s$ are computed using the original observation model $p_z$. During online planning, the precomputed $\Delta^s$ values are reweighted via importance sampling using only the state-transition model, eliminating the need to evaluate $p_z$ in real time.}
	\label{fig:offline_online_decoupling}
\end{figure}

Theorem~\ref{thm:m_i_sum_bound_guarantees} establishes performance guarantees for the distributional discrepancy estimator defined in~\eqref{eq:eps_estimator}, and demonstrates that it converges exponentially fast to the true distributional discrepancy as the number of sampled beliefs $N_b$ increases.
\begin{theorem}\label{thm:m_i_sum_bound_guarantees}
	Let $\eta>0,\delta\in (0,1)$. If $\hat{\Delta}$ is unbiased, then 
	\begin{equation}
		\begin{aligned}
			&P(|\hat{\epsilon}_t(\bar{b}_t, a)-\epsilon_t(\bar{b}_t, a)|\geq \eta) 
			&\leq 2\exp\Big( -\frac{2\eta^2 N_b}{D^2(T-k)^2(R_{\max}-R_{\min})^2} \Big)
		\end{aligned}
	\end{equation}
	for $D=\max_{i\in \{k+1,\dots,T-1\}} D_i$ where $D_i=\sup_{x_{i+1}}P(x_{i+1}|x_{i}, a_i)/Q_0(x_{i+1})$.
\end{theorem}
\begin{proof}
	The proof is available in Appendix \ref{prf:m_i_sum_bound_guarantees}.
\end{proof}

\subsection{Online Return Bound Estimation}\label{sec:return_bound_estimation}
The bound in Theorem \ref{thm:uniform_lower_and_upper_bounds_for_v_and_q} requires knowledge of the upper and lower bounds on the return under both the original and simplified distributions. Formally, we need estimators $\hat{d}_{\min}$ and $\hat{d}_{\max}$ for $d_{\min}$ and $d_{\max}$ in Theorem \ref{thm:uniform_lower_and_upper_bounds_for_v_and_q}. A simple choice is
\begin{equation}\label{eq:trivial_d}
	\hat{d}_{\min} = (T-k+1) R_{\min}, \qquad
	\hat{d}_{\max} = (T-k+1) R_{\max},
\end{equation}
which, however, ignores the dependence of the return on the agent's belief $b_k$ and policy $\pi$. For example, if the belief indicates that the agent is far from any obstacle and the planning horizon is too short for a collision to occur, taking the maximal possible return obscures this contextual information and results in a looser bound than necessary.  

To obtain tighter estimates, we define
\begin{equation}
	\hat{d}_{\max}^\pi(b_k) = \max_{i \in \{1,\dots,N_b\}} R^i, 
	\qquad
	\hat{d}_{\min}^\pi(b_k) = \min_{i \in \{1,\dots,N_b\}} R^i,
\end{equation}
where $\{R^i\}_{i=1}^{N_b} \sim P_s(\cdot | b_k, \pi)$ are the return samples used to estimate the action--value function. Theorem \ref{thm:d_guarantees} provides performance guarantees for these estimators.

\begin{theorem}\label{thm:d_guarantees}
	For the estimators $\hat{d}_{\max}^\pi(b_k) = \max_{i} R^i$ and $\hat{d}_{\min}^\pi(b_k) = \min_{i} R^i$, where $i \in \{1,\dots,N_b\}$, $R^1,\dots,R^{N_b} \sim P_s$, and for $\epsilon$ as defined in~\eqref{eq:epsilon}, the following bounds hold.
	\begin{equation}
		P_s(R_{k:T} \geq \hat{d}_{min}^\pi(b_k) | b_k ,\pi)\geq\frac{N_b}{N_b+1}, \qquad P_s(R_{k:T} \leq \hat{d}_{max} | b_k ,\pi)\geq\frac{N_b}{N_b+1}
	\end{equation}
	\begin{equation}
		P\!\left(R_{k:T} \geq \hat{d}_{min}^\pi(b_k) \Big| b_k ,\pi \right)
		\geq \frac{N_b}{N_b+1} - \epsilon, \qquad
		P\!\left(R_{k:T} \leq \hat{d}_{max}^\pi(b_k) \Big| b_k ,\pi \right)
		\geq \frac{N_b}{N_b+1} - \epsilon.
	\end{equation}
\end{theorem}
\begin{proof}
	The proof is available in Appendix \ref{prf:d_guarantees}.
\end{proof}
The bounds for both the simplified and the original returns are computed using return samples generated from the simplified belief-transition model. The uncertainty in the upper bound on the original return, induced by using samples drawn from the simplified return distribution, is quantified by $\epsilon$, which represents the distributional discrepancy between the original and simplified returns.


\subsection{Performance Guarantees}\label{sec:performance_guarantees}
In this section we provide performance guarantees for the bounds exhibited in previous sections, making them reliable in practice.

\subsubsection{Guarantees with Known Return Bounds}
Theorem~\ref{thm:trivial_d_guarantees} establishes performance guarantees for the original action--value function in terms of the simplified action--value function, under the assumption that the simplification is applied to the belief--transition model. These bounds hold for a general $\rho$-POMDP, where the cost is belief-dependent.

\begin{theorem}\label{thm:trivial_d_guarantees}
	Let $\delta,\delta_{d_{\min}}, \delta_{d_{\max}} \in [0,1]$, and let $b_k$ denote the initial belief. Consider samples $R^1,\dots,R^{N_b} \sim P_s(\cdot \mid b_k,\pi)$ generated using the simplified belief-transition model. Define
	\begin{equation}
		\epsilon \triangleq \epsilon(b_k, a_k) \triangleq \min(\sum_{i=k}^{T-1} \mathbb{E}^s[\Delta^s(b_i, a_i)|b_k, a_k,\pi], 1),
	\end{equation}
	\begin{equation}
		\eta = \sqrt{\ln(1/\delta)/(2N_b)}, \qquad \epsilon' = \epsilon + \eta.
	\end{equation}
	\begin{enumerate}
		\item \textbf{Upper Bound:} assume that $P(R_{k:T}\leq d_{max})>1-\delta_{d_{max}}$ and $P_s(R_{k:T}\leq d_{max})>1-\delta_{d_{max}}$, then \begin{enumerate}
			\item If $\epsilon'<\alpha$, then
			$U_s\triangleq (1-\frac{\epsilon'}{\alpha}) \hat{Q}^\pi_{M_s}(b_k,a_k,\alpha-\epsilon') + \frac{\epsilon'}{\alpha}d_{max}$
			\item If $\epsilon' \geq \alpha$ then $U_s\triangleq d_{max}$
		\end{enumerate}
		\item \textbf{Lower Bound:} assume that $P(R_{k:T}\geq d_{min})>1-\delta_{d_{min}}$ and $P_s(R_{k:T}\geq d_{min})>1-\delta_{d_{min}}$, then \begin{enumerate}
			\item If $\epsilon' + \alpha < 1$ then
			$
			L_s\triangleq (1+\frac{\epsilon'}{\alpha}) \hat{Q}^\pi_{M_s}(b_k,a_k,\alpha+\epsilon') - \frac{\epsilon'}{\alpha} \hat{Q}^\pi_{M_s}(b_k,a_k, \epsilon')
			$
			\item If $\epsilon' + \alpha \geq 1$ then 
			$
			L_s\triangleq\frac{1}{\alpha}[-(\alpha + \epsilon' - 1)
			d_{min}+\hat{Q}^\pi_{M_s}(b_k,a_k)-\epsilon' \hat{Q}^\pi_{M_s}(b_k,a_k,\epsilon')]
			$
		\end{enumerate}
	\end{enumerate}
	Then $P(L_s \leq Q_M^\pi(b_k, a_k,\alpha)) > (1-\delta)(1-\delta_{d_{min}})$ and $P(Q_M^\pi(b_k, a_k,\alpha)\leq U_s)> (1-\delta)(1-\delta_{d_{max}})$.
\end{theorem}
\begin{proof}
	The proof is available in Appendix \ref{prf:trivial_d_guarantees}.
\end{proof}

Estimation of the belief-transition model remains an open problem in the literature; consequently, the distributional discrepancy $\epsilon$ in Theorem~\ref{thm:trivial_d_guarantees}, when defined in terms of the belief-transition model, is currently intractable to estimate. In the special case where only the observation model is simplified, the one-step distributional discrepancy $\Delta^s$ is defined with respect to the observation model rather than the belief-transition model (see Section \ref{sec:simp_obs_bound}). In this case, Theorem~\ref{thm:trivial_d_guarantees} holds when $\Delta^s$ is defined using the original and simplified observation models, as characterized in Theorem~\ref{thm:cvar_tv_dist_bound_simp_obs}, thereby rendering the estimation of $\epsilon$ tractable.

To attain standard bounds, one can set trivial bounds to the return $d_{min}=(T-k+1)R_{\min}$ and $d_{max}=(T-k+1)R_{\max}$. These bounds would produce $\delta_{d_{min}}=\delta_{d_{max}}=0$ and guarantees of $1-\delta$ for the bounds in Theorem \ref{thm:trivial_d_guarantees}.

\subsubsection{Guarantees with Estimated Return Bounds}
In practice, it is preferable to estimate the return bounds $d_{\min}$ and $d_{\max}$ with respect to the current belief and policy in order to obtain tighter bounds (see Section~\ref{sec:return_bound_estimation}). Theorem~\ref{thm:guarantees} establishes guarantees for the resulting simplified bounds, in which $d_{\min}$ and $d_{\max}$ are replaced by their estimators.
\begin{theorem}\label{thm:guarantees}
	Let $b_k$ be an initial belief and let $R^1,\dots,R^{N_b} \sim P_s(\cdot | b_k,\pi)$ be a sample that utilizes the simplified belief-transition model. Denote \begin{equation}
		\epsilon \triangleq \epsilon(b_k, a_k) \triangleq \min(\sum_{i=k}^{T-1} \mathbb{E}^s[\Delta^s(b_i, a_i)|b_k, a_k,\pi], 1)
	\end{equation}
	\begin{equation}
		\eta = \sqrt{\ln(1/\delta)/(2N_b)}, \qquad \epsilon' = \epsilon + \eta
	\end{equation}
	\begin{equation}
		\hat{d}_{\max}^\pi(b_k) = \max_{i \in \{1,\dots,N_b\}} R^i, 
		\qquad
		\hat{d}_{\min}^\pi(b_k) = \min_{i \in \{1,\dots,N_b\}} R^i,
	\end{equation}
	
	\begin{enumerate}
		\item \textbf{Upper Bound:} Under the following conditions, it holds that $P(Q^\pi_{M}(\bar{b}_k, a_k, \alpha)\leq U_s|b_k,\pi)\geq (1-\delta)\Big(\frac{N_b}{N_b+1} -\epsilon \Big)$
		\begin{enumerate}
			\item If $\epsilon'<\alpha$, then
			$U_s\triangleq (1-\frac{\epsilon'}{\alpha}) \hat{Q}^\pi_{M_s}(b_k,a_k,\alpha-\epsilon') + \frac{\epsilon'}{\alpha}\hat{d}_{max}^\pi(b_k)$
			\item If $\epsilon' \geq \alpha$ then $U_s\triangleq \hat{d}_{max}^\pi(b_k)$
		\end{enumerate}
		\item \textbf{Lower Bound:} Under the following conditions, it holds that $P(Q^\pi_{M}(\bar{b}_k, a_k, \alpha)\geq L_s|b_k,\pi)\geq (1-\delta)\Big(\frac{N_b}{N_b+1} -\epsilon \Big)$
		\begin{enumerate}
			\item If $\epsilon' + \alpha < 1$ then 
			$
			L_s\triangleq (1+\frac{\epsilon'}{\alpha}) \hat{Q}^\pi_{M_s}(b_k,a_k,\alpha+\epsilon') - \frac{\epsilon'}{\alpha} \hat{Q}^\pi_{M_s}(b_k,a_k, \epsilon')
			$
			\item If $\epsilon' + \alpha \geq 1$ then 
			$
			L_s\triangleq\frac{1}{\alpha}[-(\alpha + \epsilon' - 1)\hat{d}_{min}^\pi(b_k)+\hat{Q}^\pi_{M_s}(b_k,a_k)-\epsilon' \hat{Q}^\pi_{M_s}(b_k,a_k,\epsilon')]
			$
		\end{enumerate}
	\end{enumerate}
\end{theorem}
\begin{proof}
	The proof is available in Appendix \ref{prf:guarantees}.
\end{proof}
Recall that the theoretical bounds derived in Section~\ref{sec:cvar_bounds} impose constraints on the relationship between the distributional discrepancy $\epsilon$ and the CVaR confidence level $\alpha$. When these conditions are violated, the resulting bounds on the action-value function become trivial. Employing estimators for the minimum and maximum return enables the agent to act according to the guarantees in Theorem~\ref{thm:guarantees}, even when these constraints are not satisfied. 

As an example (illustrated in Figure~\ref{fig:collapsed_bounds_example}), consider $\epsilon=\alpha=0.1$, $N_b=1000$, and $\delta=0.01$, with planning horizon $T=2$. The agent must choose between a dangerous path that passes through an obstacle and a safe path that avoids it. The agent incurs a cost of $1$ for remaining in place and a cost of $100$ upon colliding with an obstacle. In this setting, even under partial observability, the probability of encountering an obstacle along the safe path, as well as the probability of not encountering an obstacle along the dangerous path, may be small. Consequently, the estimators satisfy $\hat{d}^\pi_{\max}(b_0)=\hat{d}^\pi_{\min}(b_0)\approx 1\times 2$ for the safe path and $\hat{d}^\pi_{\max}(b_0)=\hat{d}^\pi_{\min}(b_0)\approx 100\times 2$ for the dangerous path. This allows the agent to distinguish between the two paths, in the sense that the CVaR of the dangerous path is strictly larger than that of the safe path.

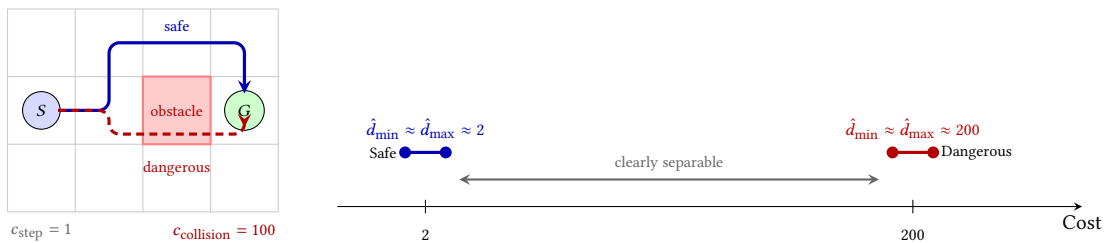
\begin{figure}[t]
	\centering
	\subfloat[An agent at position $S$ must choose between a safe path (blue, cost $1$ per step) and a dangerous path (red, passing through an obstacle with collision cost $100$). The goal's position is $G$.]{%
		\begin{tikzpicture}[>=stealth, font=\small, scale=0.9]
			\draw[black!20, thin] (0,0) grid (4,3);

			\fill[red!20, draw=red!50, thick] (2,1) rectangle (3,2);
			\node[font=\scriptsize, red!70!black] at (2.5,1.5) {obstacle};

			\node[circle, draw, fill=blue!15, minimum size=0.5cm, font=\scriptsize\bfseries] (start) at (0.5,1.5) {$S$};
			\node[circle, draw, fill=green!20, minimum size=0.5cm, font=\scriptsize\bfseries] (goal) at (3.5,1.5) {$G$};

			\draw[blue!70!black, very thick, ->, rounded corners=4pt]
				(0.75,1.5) -- (1.5,1.5) -- (1.5,2.5) -- (3.5,2.5) -- (3.5,1.75);
			\node[above=1pt, font=\scriptsize, blue!70!black] at (2.5,2.5) {safe};

			\draw[red!70!black, very thick, ->, densely dashed, rounded corners=4pt]
				(0.75,1.5) -- (1.5,1.5) -- (1.5,1.15) -- (3.5,1.15) -- (3.5,1.25);
			\node[below=2pt, font=\scriptsize, red!70!black] at (2.5,0.95) {dangerous};

			\node[font=\scriptsize, black!60] at (0.5,-0.3) {$c_{\text{step}}=1$};
			\node[font=\scriptsize, red!60!black] at (3.2,-0.3) {$c_{\text{collision}}=100$};
		\end{tikzpicture}%
		\label{fig:collapsed_bounds_paths}
	}
	\qquad
	\subfloat[When $\epsilon' \geq \alpha$, the CVaR bounds collapse to ${[\hat{d}_{\min}, \hat{d}_{\max}]}$. The return estimators for the safe path (${\approx 2}$) and dangerous path (${\approx 200}$) remain well-separated.]{%
		\begin{tikzpicture}[>=stealth, font=\small, scale=0.9]
			\draw[->] (-0.5,0) -- (10.5,0) node[below] {Cost};

			\foreach \x/\lbl in {0.8/$2$, 8.0/$200$} {
				\draw (\x, -0.1) -- (\x, 0.1);
				\node[below=3pt, font=\scriptsize] at (\x, -0.1) {\lbl};
			}

			\draw[blue!70!black, very thick] (0.5, 0.8) -- (1.1, 0.8);
			\fill[blue!70!black] (0.5, 0.8) circle (2.5pt);
			\fill[blue!70!black] (1.1, 0.8) circle (2.5pt);
			\node[above=2pt, font=\scriptsize, blue!70!black] at (0.8, 0.8) {$\hat{d}_{\min} \approx \hat{d}_{\max} \approx 2$};
			\node[left, font=\scriptsize] at (0.5, 0.8) {Safe};

			\draw[red!70!black, very thick] (7.7, 0.8) -- (8.3, 0.8);
			\fill[red!70!black] (7.7, 0.8) circle (2.5pt);
			\fill[red!70!black] (8.3, 0.8) circle (2.5pt);
			\node[above=2pt, font=\scriptsize, red!70!black] at (8.0, 0.8) {$\hat{d}_{\min} \approx \hat{d}_{\max} \approx 200$};
			\node[right, font=\scriptsize] at (8.3, 0.8) {Dangerous};

			\draw[<->, black!60, thick] (1.3, 0.4) -- (7.5, 0.4);
			\node[above, font=\scriptsize, black!60] at (4.4, 0.4) {clearly separable};
		\end{tikzpicture}%
		\label{fig:collapsed_bounds_intervals}
	}
	\caption{Illustration of bound-based action selection when $\epsilon' \geq \alpha$. (a)~The agent chooses between a safe path (cost $\approx 1 \times 2 = 2$) and a dangerous path through an obstacle (cost $\approx 100 \times 2 = 200$). (b)~Although the CVaR bounds collapse to $[\hat{d}_{\min}, \hat{d}_{\max}]$, the return estimators remain well-separated, enabling the agent to distinguish between the two paths.}
	\label{fig:collapsed_bounds_example}
\end{figure}

When the constraint $\epsilon' < \alpha$ required for the original bounds is violated---as in this example where $\epsilon = \alpha = 0.1$ and hence $\epsilon' \geq \epsilon \geq \alpha$---the bounds in Theorem~\ref{thm:guarantees} reduce to $\hat{d}_{\max}$ and $\hat{d}_{\min}$, causing the upper and lower bounds to coincide. Nevertheless, decision making remains possible: the estimators for the safe and dangerous paths differ substantially ($\approx 2$ versus $\approx 200$), enabling the agent to distinguish between them.

By Corollary~\ref{thm:cvar_tv_dist_bound_simp_obs}, the guarantees stated in Theorem~\ref{thm:guarantees} and Theorem \ref{thm:trivial_d_guarantees} also apply to the specific setting in which the observation model is simplified, with the distributional discrepancy $\epsilon$ between the original and simplified observation models defined as in Corollary~\ref{thm:cvar_tv_dist_bound_simp_obs}. As shown in Section \ref{sec:online_bound_estimation}, this setting can be resolved within the framework of online planning.
\section{Limitations}\label{sec:limitations}
The bounds in Theorem~\ref{thm:uniform_lower_and_upper_bounds_for_v_and_q} depend on the accumulation of one-step distributional discrepancies over the planning horizon, which may limit their effectiveness for long-horizon problems.

Importantly, this accumulation does not necessarily grow linearly with the horizon and is highly problem dependent. For example, in a Light–Dark environment, the agent receives observations only in light regions and relies exclusively on the state-transition model in dark regions. As a result, no distributional discrepancy is accumulated in dark regions, since the observation model is not invoked. Consequently, the effective distributional discrepancy can grow sublinearly with the horizon. This behavior is examined empirically in the experiments section, where we also consider a more challenging variant of the Light–Dark environment in which distributional discrepancy accumulates even in dark regions.

Moreover, even when the distributional discrepancy becomes large relative to $\alpha$, the agent can still act. In this case, the value function bounds collapse to their extrema; specifically, the upper value function bound reduces to $\hat{d}_{\max}$ (Theorem~\ref{thm:guarantees}), which is expected to be low in safe regions and high in dangerous ones. Analogous considerations hold for the lower bound. Consequently, decision making relies on minimum and maximum return samples conditioned on the path and policy, while the simplified observation model continues to provide computational acceleration, as it remains necessary for sampling return realizations.

Finally, the proposed simplification framework yields computational gains only when sampling from the state-transition model is cheaper than sampling from the observation model. The bounds in Theorem~\ref{thm:uniform_lower_and_upper_bounds_for_v_and_q} require samples from both the simplified observation model and the state-transition model to compensate for the absence of the original observation model. Consequently, if the state-transition model is computationally more expensive than the original observation model, the simplification approach would increase rather than decrease planning time. This limitation is shared by other observation model simplification frameworks~\cite{LevYehudi24aaai}.

\section{Experiments}
In this section, we evaluate our derived bounds across several standard benchmark environments. We demonstrate the efficacy of our bounds by examining their performance on selected action sequences, illustrating their capacity to distinguish between safe and unsafe trajectories while achieving computational acceleration relative to the baseline model. We subsequently demonstrate the integration of these bounds into open-loop planning frameworks. All simulations were conducted using the POMDPPlanners package \cite{pariente2026pomdpplannersopensourcepackagepomdp}. Detailed simulation configurations are provided in Appendix~\ref{sec:simulations_appendix}, and hardware specifications in Appendix~\ref{sec:hardware}.

\subsection{Theoretical CVaR Bounds}
\begin{figure*}[htbp]
	\centering
	\subfloat[]{\includegraphics[width=0.36\textwidth]{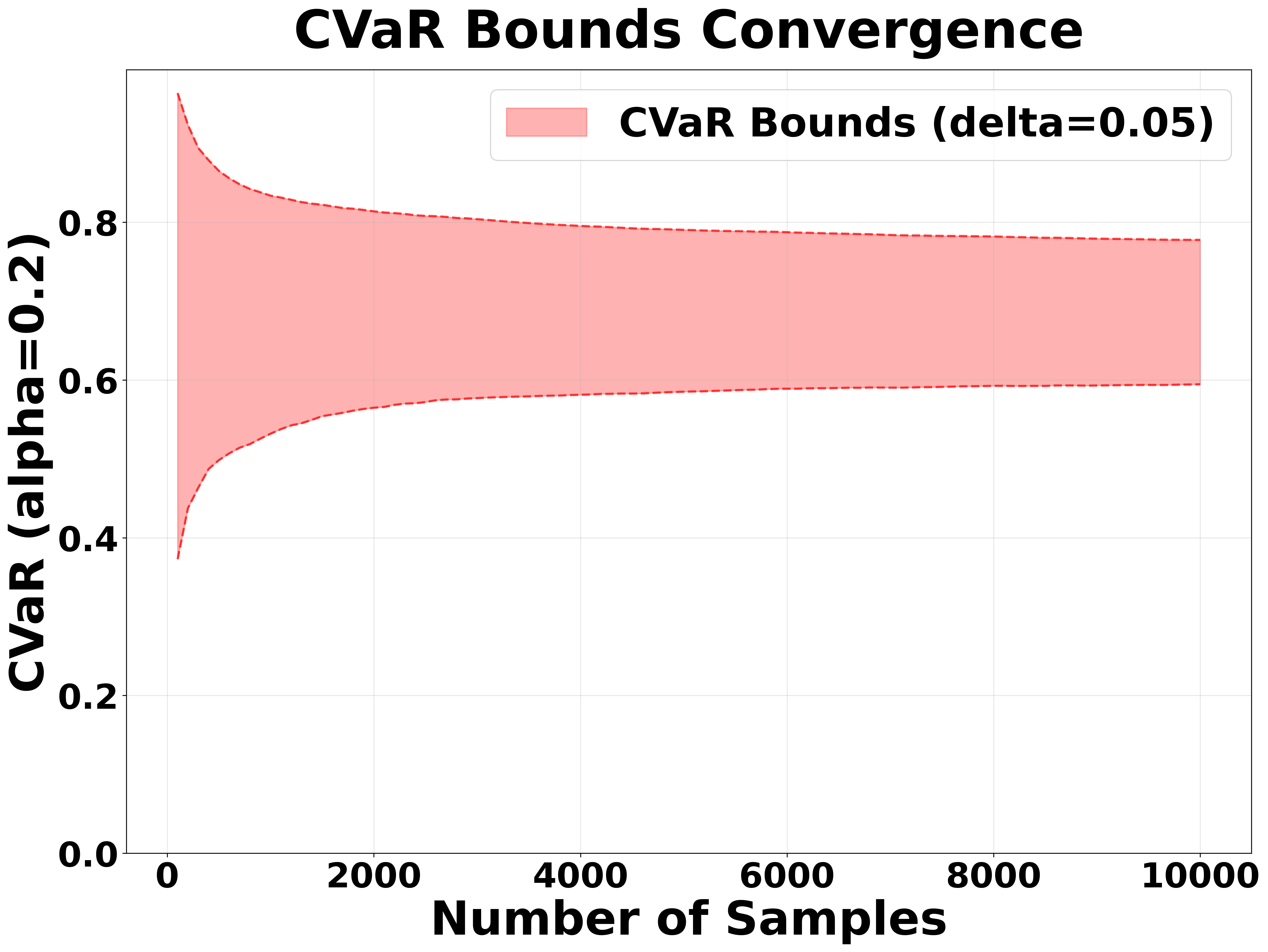}\label{fig:cvar_bounds_convergence}}
	\quad
	\subfloat[]{\includegraphics[width=0.36\textwidth]{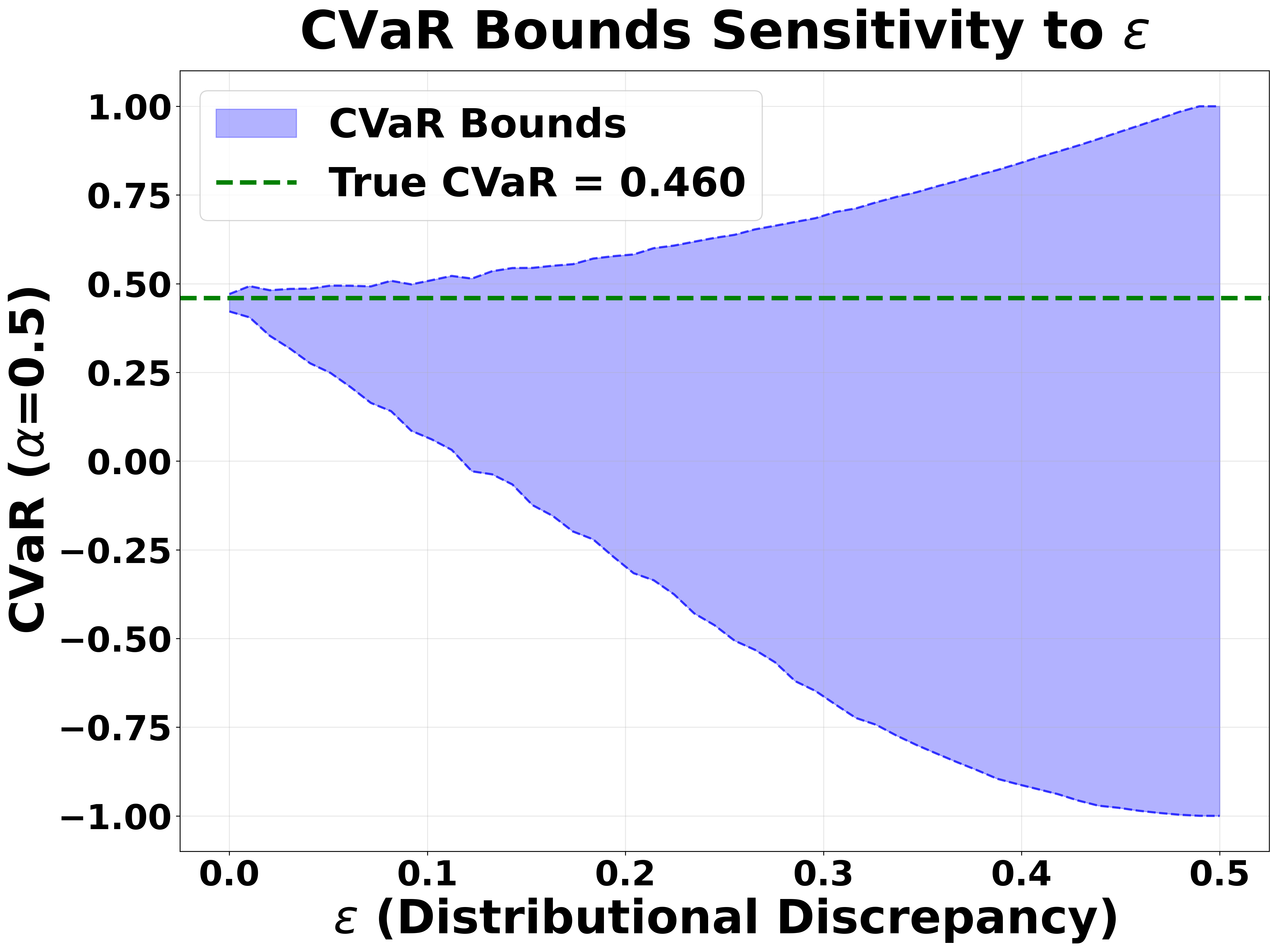}\label{fig:bounds_sensitivity_to_epsilon}}\\[1em]
	\subfloat[]{\includegraphics[width=0.36\textwidth]{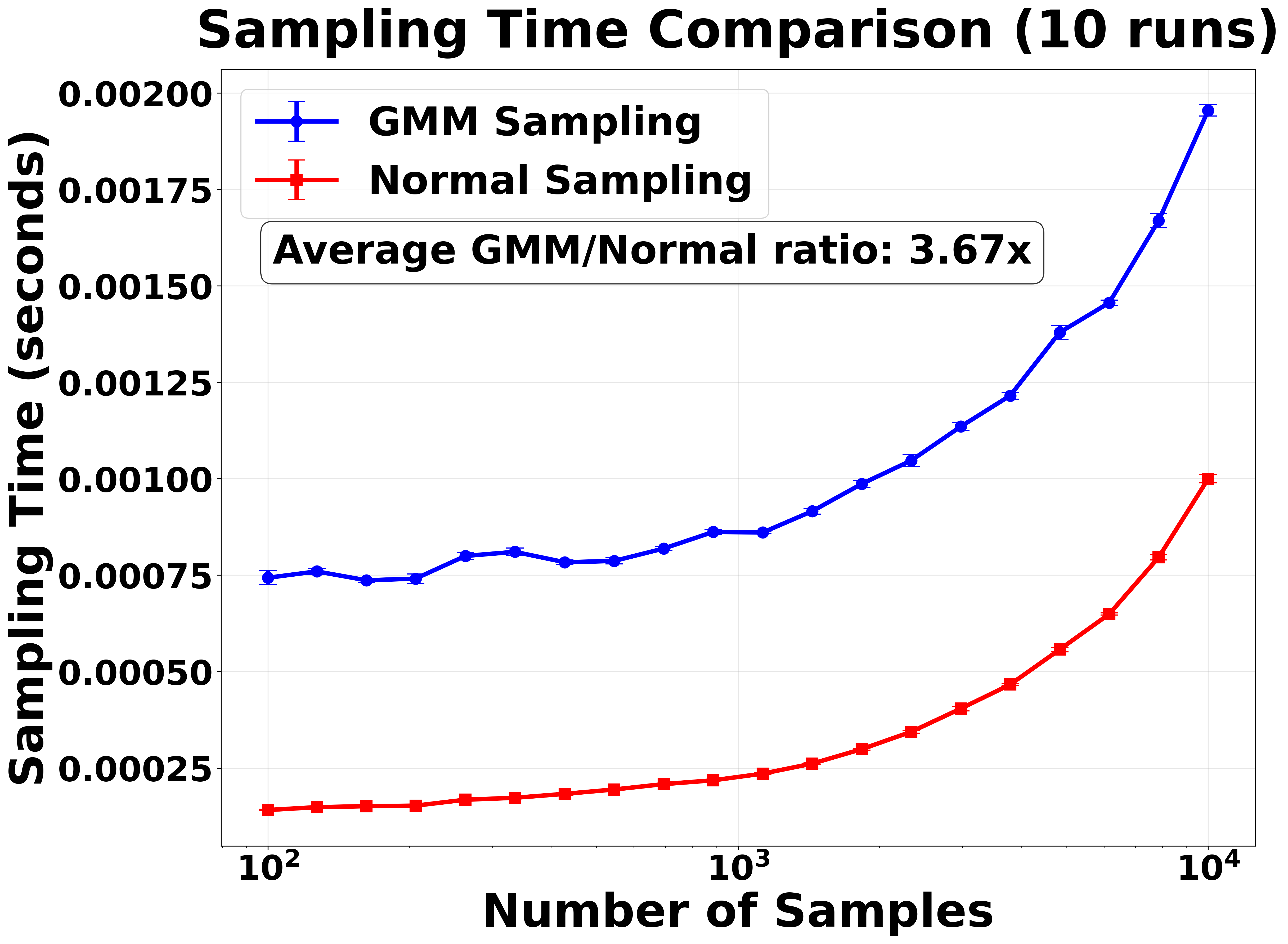}\label{fig:sampling_running_time}}
	\quad
	\subfloat[]{\includegraphics[width=0.36\textwidth]{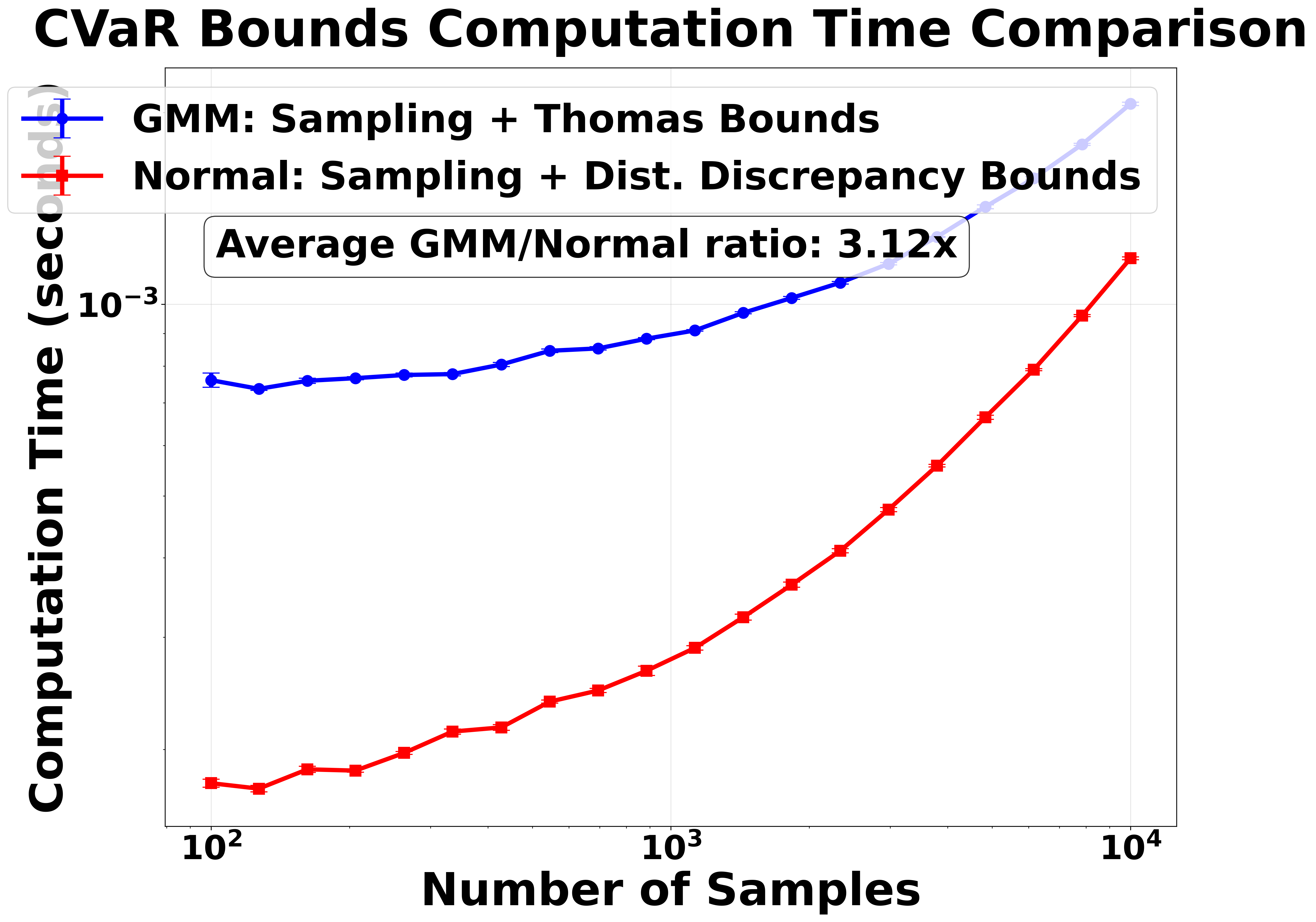}\label{fig:cvar_time_comparison_1}}
	\caption{Figure \ref{fig:cvar_bounds_convergence} shows the bounds established in Theorem \ref{thm:ecdf_y_bound_cvar_x}, demonstrating how a simple truncated Normal distribution can be employed to bound the CVaR of a truncated GMM. Figure \ref{fig:bounds_sensitivity_to_epsilon} illustrates the sensitivity of the CVaR bounds to the distributional discrepancy $\epsilon$, showing how the bound interval widens as $\epsilon$ increases, while the true CVaR remains contained within the bounds. When $\epsilon$ exceeds the risk level $\alpha$, the bounds become trivial, spanning the entire support; this is expected, as a large distributional discrepancy implies that the surrogate random variable $Y$ provides negligible information about the target random variable $X$. Figure \ref{fig:sampling_running_time} presents a comparison of the sampling times for each distribution, while Figure \ref{fig:cvar_time_comparison_1} reports the total computation time required to sample from the truncated GMM and estimate its CVaR, in contrast to sampling from the truncated Normal surrogate and computing the associated concentration bounds. Figures \ref{fig:cvar_bounds_convergence}, \ref{fig:sampling_running_time}, and \ref{fig:cvar_time_comparison_1} were configured with confidence level $\alpha=0.2$ and probability of error $\delta=0.05$. Figure \ref{fig:bounds_sensitivity_to_epsilon} uses $\alpha=0.5$. Figures \ref{fig:sampling_running_time} and \ref{fig:cvar_time_comparison_1} display 95\% confidence intervals around the empirical means.}
\end{figure*}

Specifically, we compare a truncated Gaussian Mixture Model (GMM) and its corresponding truncated Normal approximation for CVaR estimation under bounded support. Such a setting arises, for example, in POMDP planning, where a computationally expensive model used by the agent during online planning is replaced with a more tractable surrogate model, thereby improving the agent’s decision-making speed \cite{LevYehudi24aaai}. In these settings, the agent’s decisions may rely on bounds for the original value function that are derived from the tractable surrogate model \cite{ijcai2022p637}.

The GMM consists of five components with means $\mu_1 = 0.2$, $\mu_2 = -0.2$, $\mu_3 = -0.5$, $\mu_4 = 0.5$, $\mu_5 = 0.0$, variances $\sigma_1^2 = 0.5$, $\sigma_2^2 = 0.2$, $\sigma_3^2 = 0.1$, $\sigma_4^2 = 0.1$, $\sigma_5^2 = 0.3$, and weights $w_1 = 0.3$, $w_2 = 0.2$, $w_3 = 0.05$, $w_4 = 0.05$, $w_5 = 0.4$. The Normal approximation matches the GMM's mean and variance but cannot reproduce its multi-modal structure or outlier effects. All samples are truncated to the interval $[-1, 1]$, which reshapes the tails and slightly distorts boundary components. We compute CVaR at the 20\% quantile for sample sizes ranging from 100 to 10{,}000, with 100 independent repetitions per setting to ensure statistical reliability. The distributional discrepancy ($\epsilon$ in Theorem \ref{thm:ecdf_y_bound_cvar_x}) between the truncated GMM and the truncated Normal distribution is assessed via simulation, by estimating their respective cumulative distribution functions over a common set of bins and computing the maximum difference across all bins. This setup enables a direct assessment of the trade-off between computational efficiency and statistical accuracy when approximating a complex truncated mixture by a single truncated Normal distribution.

Figure \ref{fig:sampling_running_time} presents the sampling time comparison between the truncated GMM and the truncated Normal distribution, with an observed average time ratio of approximately 3.7 in favor of the Normal distribution. Figure \ref{fig:cvar_time_comparison_1} reports a total computational speedup of approximately 3.1 when comparing the process of sampling from the truncated GMM and estimating its CVaR to that of sampling from the truncated Normal and computing both upper and lower bounds as given in Theorem \ref{thm:ecdf_y_bound_cvar_x}. Figure \ref{fig:cvar_bounds_convergence} illustrates the convergence behavior of the CVaR bounds as a function of the number of samples, comparing estimates obtained from the truncated GMM with bounds derived from the truncated Normal surrogate. It is important to note that these bounds are obtained without requiring full knowledge of the underlying GMM distribution. Instead, they rely solely on the discrepancy between the corresponding CDFs. Figure \ref{fig:bounds_sensitivity_to_epsilon} depicts the sensitivity of the CVaR bounds to the distributional discrepancy $\epsilon$. In this experiment, the base distribution $F$ is a truncated Normal on $[-1,1]$ with mean $0$ and standard deviation $1$. For each value of $\epsilon \in [0, 0.5]$, a perturbed distribution $G$ is constructed such that its CDF satisfies $G(x) = \min(F(x) + \epsilon, 1)$, yielding a Kolmogorov--Smirnov distance of exactly $\epsilon$ between $F$ and $G$ by construction. Samples from $G$ are obtained via inverse transform sampling. The resulting bounds grow monotonically with $\epsilon$ while consistently enclosing the true CVaR of $F$. When $\epsilon$ exceeds the risk level $\alpha$, the bounds become trivial, spanning the full support; this is expected, since a large distributional discrepancy implies that $Y$ provides negligible information about $X$.

We evaluated the concentration bounds established by \cite{pmlr-v97-thomas19a} alongside our proposed bounds from Theorem \ref{thm:ecdf_cvar_bound} on a set of probability distributions: $\operatorname{Beta}(2, 2)$, $\operatorname{Beta}(0.5, 0.5)$, $\operatorname{Beta}(2, 5)$, $\operatorname{Beta}(5, 2)$, $\operatorname{Beta}(10, 2)$, $\operatorname{Beta}(2, 10)$, and the Laplace$(0, 1)$ distribution. These distributions were selected to match those used in the original study by \cite{pmlr-v97-thomas19a}, enabling a direct comparison under identical conditions. For each distribution, our bounds precisely coincide with those reported by \cite{pmlr-v97-thomas19a}, resulting in complete overlap between the two sets of bounds. The graphs exhibiting these results are available in Appendix~\ref{sec:thomas_comparison} (Figure~\ref{fig:cvar_bounds_comparison}).

\subsection{POMDP Environments}
We evaluate our approach on three POMDP domains: 2D Light-Dark Navigation, Laser Tag, and Push. Each environment contains dangerous areas where the likelihood of incurring high penalties is elevated, requiring the agent to balance risk-aware decision-making with task completion. In the 2D Light-Dark POMDP, an agent must navigate from a start position to a goal region while avoiding these dangerous areas. The observation noise is position-dependent, with lower uncertainty near designated beacon locations, requiring the agent to balance information gathering with goal-directed movement. In Laser Tag POMDP, the agent must locate and tag an opponent in a partially observable environment using noisy range-bearing measurements from multiple sensors, resulting in an 8-dimensional observation space, while navigating around dangerous regions. The Push POMDP involves manipulating an object to a target location while managing partial observability about the object's position through noisy 2-dimensional observations and avoiding dangerous areas. For each domain, we construct two environment variants: an original environment with a complex observation model represented by Gaussian Mixture Models (GMMs) capturing multi-modal observation distributions, and a simplified environment where this complex distribution is approximated using a single Gaussian distribution. This approximation introduces a controlled model mismatch, allowing us to evaluate the robustness of planning algorithms when the true observation dynamics differ from the assumed model. The GMM-based observation models enable realistic representation of sensor fusion, occlusions, and multi-hypothesis tracking, while the Gaussian approximation provides computational tractability at the cost of distributional fidelity. Full environment specifications are provided in Appendix~\ref{sec:env_configs}.

\subsection{Online Distributional Discrepancy Estimation}
To compute the distributional discrepancy \( \epsilon \) appearing in Corollary~\ref{thm:cvar_tv_dist_bound_simp_obs}, we define \( Q_0 \) as the uniform distribution over the axis-aligned square whose lower-left corner is \( x_{\text{start}} \) and upper-right corner is \( x_{\text{goal}} \). Samples \( x_n^\Delta \) are drawn from the mixture distribution \( (p_Z + q_Z)/2 \); that is, with probability \( 0.5 \), a sample is drawn from the original GMM observation model, and with probability \( 0.5 \), from the simplified observation model. The estimator \( \hat{\Delta}^s(x_n^\Delta) \) is then defined as
\begin{equation}\label{eq:delta_estimator}
	\hat{\Delta}^s(x_n^{\Delta}) = \sum_{j=1}^{N_z} 2 \cdot \frac{\left| p_z(z_j^n \mid x_n^{\Delta}) - q_z(z_j^n \mid x_n^{\Delta}) \right|}{p_z(z_j^n \mid x_n^{\Delta}) + q_z(z_j^n \mid x_n^{\Delta})},
\end{equation}
and is evaluated over $N^\Delta=100$ presampled states. Here, \( N_z = 2000 \) denotes the number of sampled observations used in the estimation procedure. Given equation~\eqref{eq:delta_estimator}, the belief-dependent one-step distributional discrepancy \( \hat{m}(\bar{b}_t, a) \) can be estimated accordingly. Since the state-transition model in \eqref{eq:m_est_state} assigns exponentially small probability to states that are far from the current state, we employ a K-Nearest Neighbors (KNN) approach with $K=10$ to select a subset of the samples \( \{x_n^\Delta\} \) that lie in the vicinity of a given state \( x \). The estimation is then performed over this localized subset, rather than the entire collection \( \{x_n^\Delta\} \).

\subsection{Empirical Static CVaR Bound Evaluation: Open-Loop Policies}
To evaluate the bounds we define for each environment two action-sequences - one that dictates a safe path and the other that dictates a dangerous path. Figure \ref{fig:agent_path} shows that bounds based on the simplified observation models distinguish between the two action sequences and would eliminate the dangerous path under action elimination. The estimated distributional discrepancies are $\epsilon = 0.38$ and $\epsilon = 0.36$ for the dangerous and safe paths in Light-Dark, $\epsilon = 0.52$ and $\epsilon = 0.53$ in Laser Tag, and $\epsilon = 0.24$ and $\epsilon = 0.20$ in Push. Despite high distributional discrepancies across paths within each environment, the bounds successfully discriminate between the two action sequences. Notably, in the Laser Tag environment, the estimated $\epsilon$ exceeds $0.5$, which is the value of the risk level $\alpha$; nevertheless, the bounds remain sufficiently tight to separate the safe and dangerous action sequences.

Comparing the bounds computation time using the original observation models using the bounds of \cite{pmlr-v97-thomas19a} to our bounds that use the simplified observation models, we gain approximately $20.89\times$ acceleration in Light-Dark POMDP, $8.87\times$ in Laser Tag POMDP, and $7.62\times$ in Push POMDP. Figure \ref{fig:acceleration_sensitivity} illustrates how the acceleration remains solid when the number of return samples and horizon changes.

Figure~\ref{fig:horizon_bounds} presents the evolution of the CVaR bounds as a function of the planning horizon for the same safe and dangerous action sequences depicted in Figure~\ref{fig:agent_path}, where each action sequence is decomposed into sub-sequences of increasing horizon length. In the Light-Dark environment, the bounds are separated from the earliest horizon, since the agent's start state is close to the dangerous area and the dangerous path immediately traverses the high-cost region. In Laser Tag and Push, the bounds of both paths are similar at early horizons, before the dangerous path enters the hazardous region, confirming that the bound separation is driven by the actual risk difference rather than by an artifact of the bounding mechanism. In Laser Tag, the separation is most pronounced: both paths have comparable bounds up to horizon~$4$, after which the dangerous path's bounds increase sharply as it enters the dangerous area, while the safe path remains at low values throughout. In Push, the two paths overlap closely up to horizon~$4$, and the bounds diverge around horizon~$6$ when the dangerous path approaches the hazardous region. These patterns are consistent across both $\alpha=0.5$ and $\alpha=0.1$, indicating that the bounds remain discriminative under a more risk-averse setting. The bottom row of Figure~\ref{fig:horizon_bounds} further examines the sensitivity of the bounds to the risk level~$\alpha$ at a fixed horizon. As $\alpha$ decreases, the CVaR focuses on worse-case outcomes, causing both the bound values and the bound intervals to increase; nevertheless, the bounds maintain a clear separation between the safe and dangerous paths across the full range of~$\alpha$.

\begin{figure}[H]
	\centering
	\subfloat[]{\includegraphics[height=0.24\textwidth]{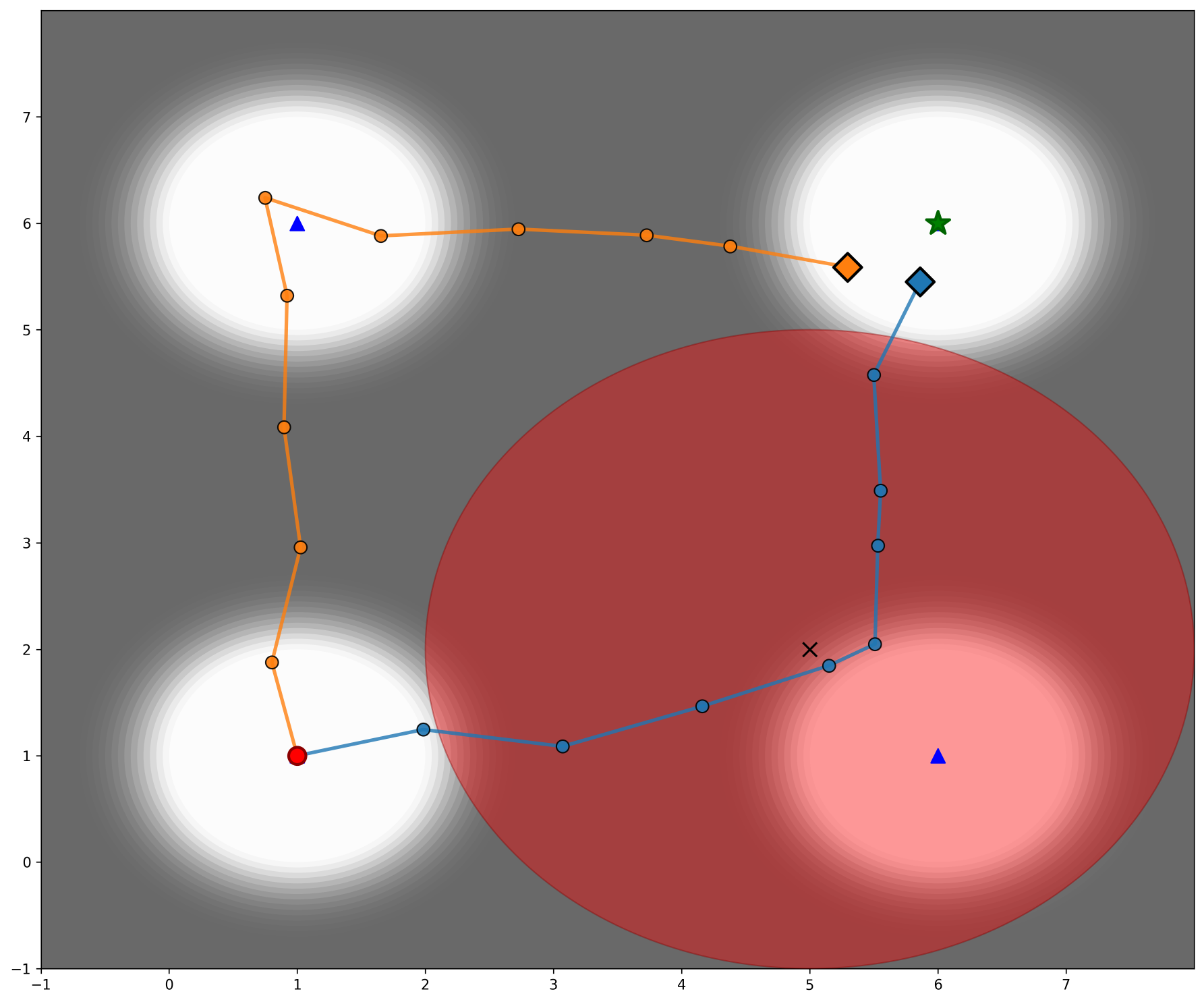}\label{fig:light_dark_env}}
	\subfloat[]{\includegraphics[height=0.24\textwidth]{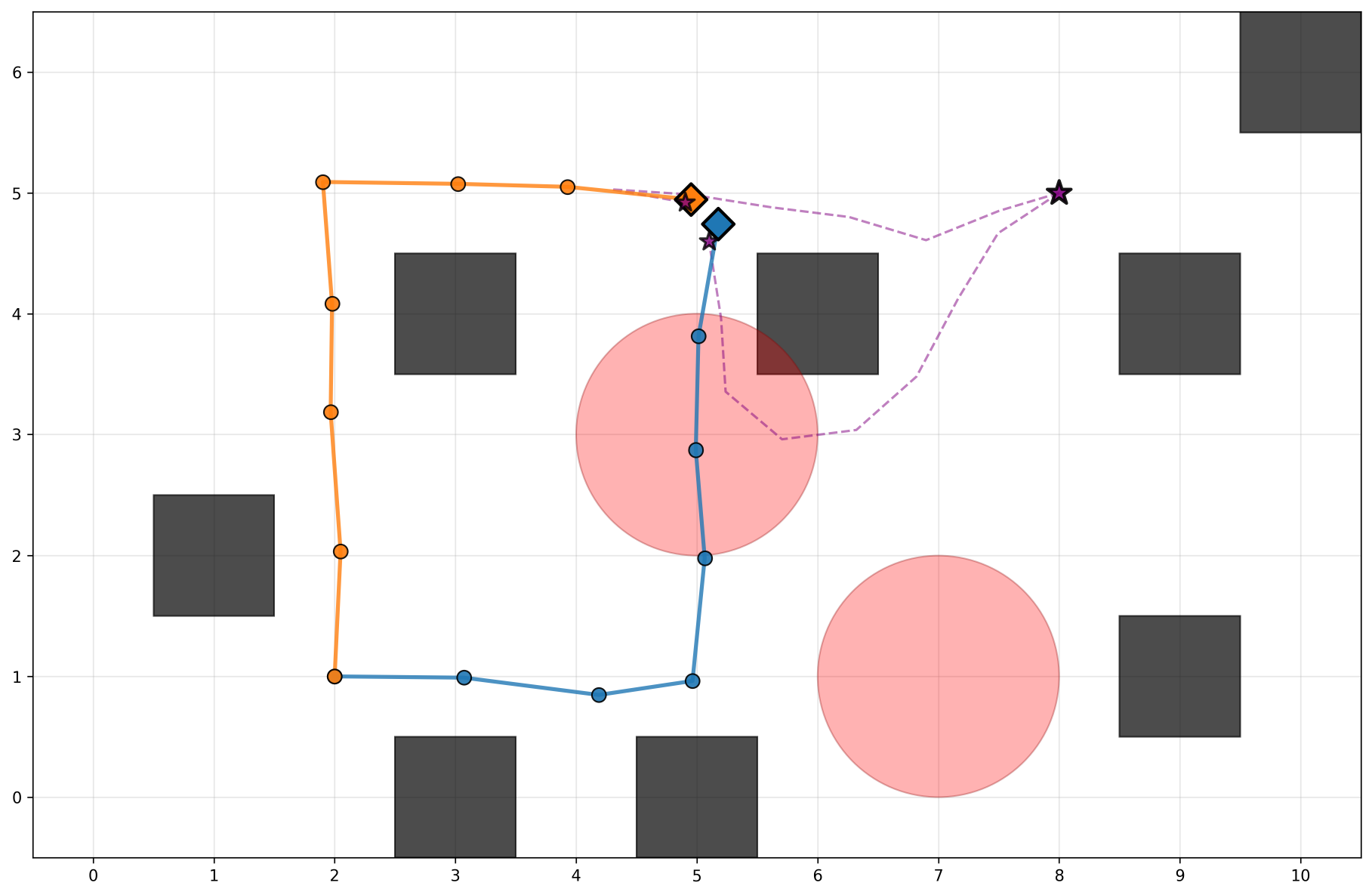}\label{fig:laser_tag_env}}
	\subfloat[]{\includegraphics[height=0.24\textwidth]{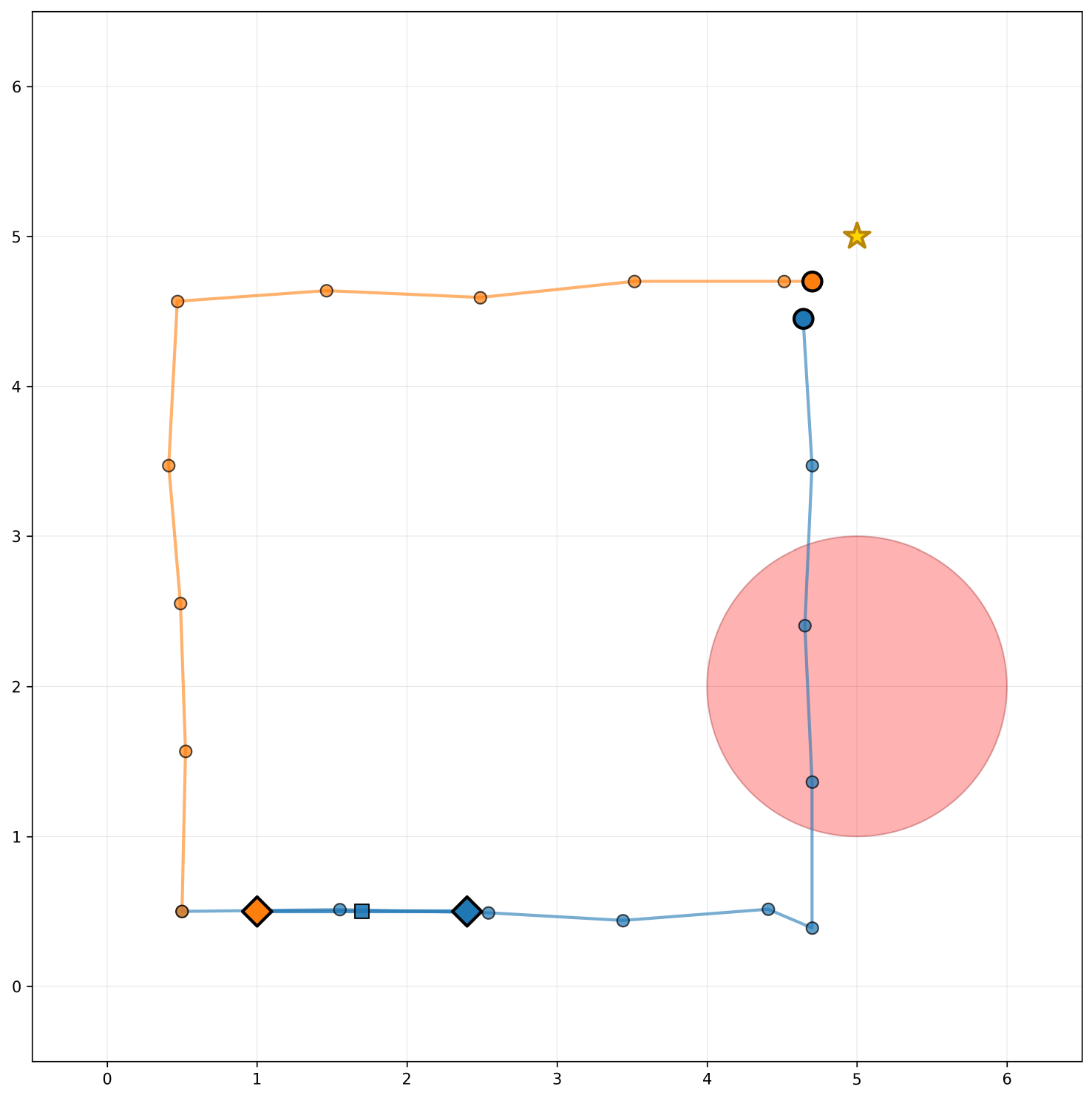}\label{fig:push_env}} \\[1em]
	\subfloat[]{\includegraphics[width=0.30\textwidth]{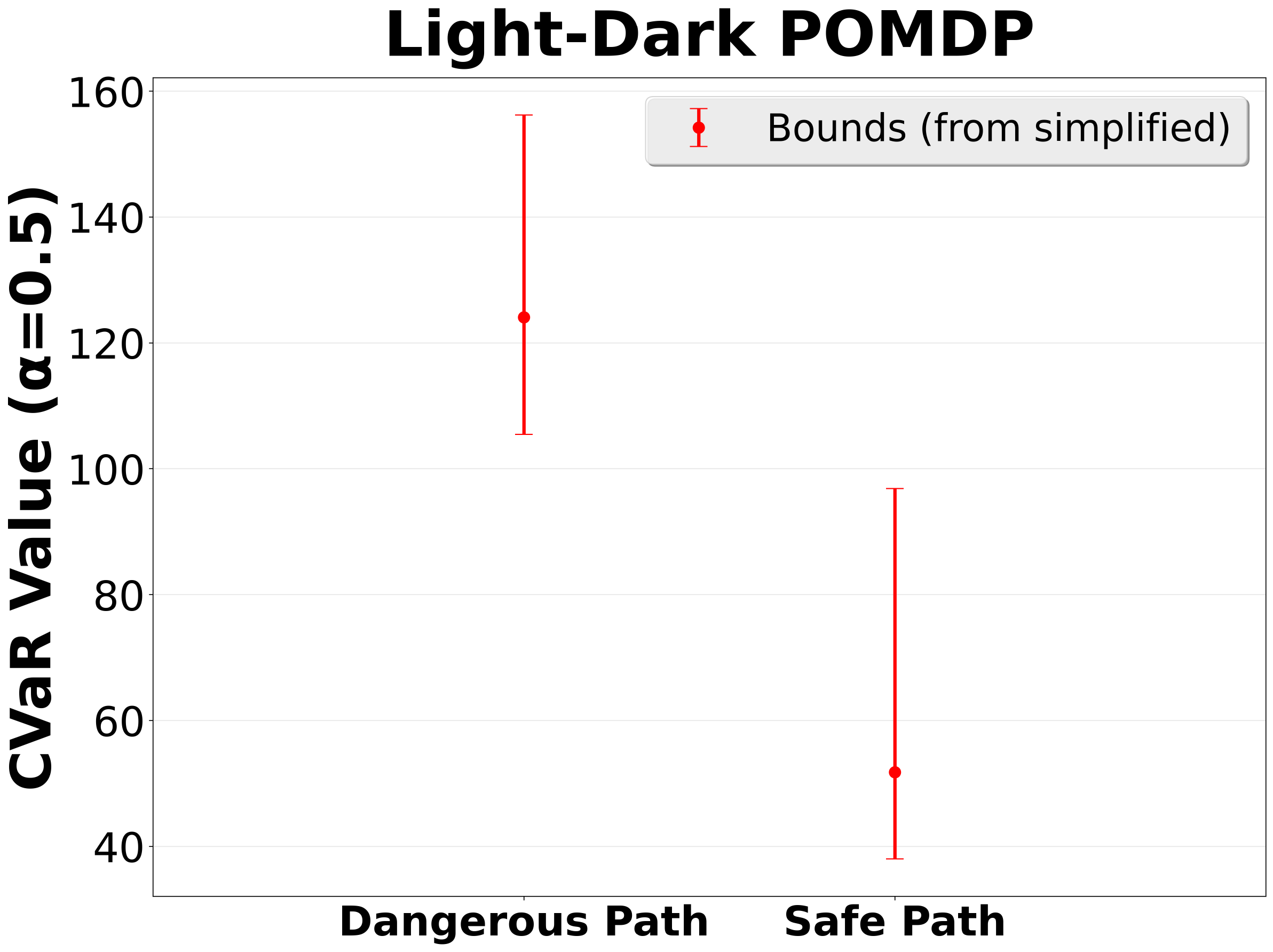}\label{fig:light_dark_action_sequence_bounds}}
	\subfloat[]{\includegraphics[width=0.30\textwidth]{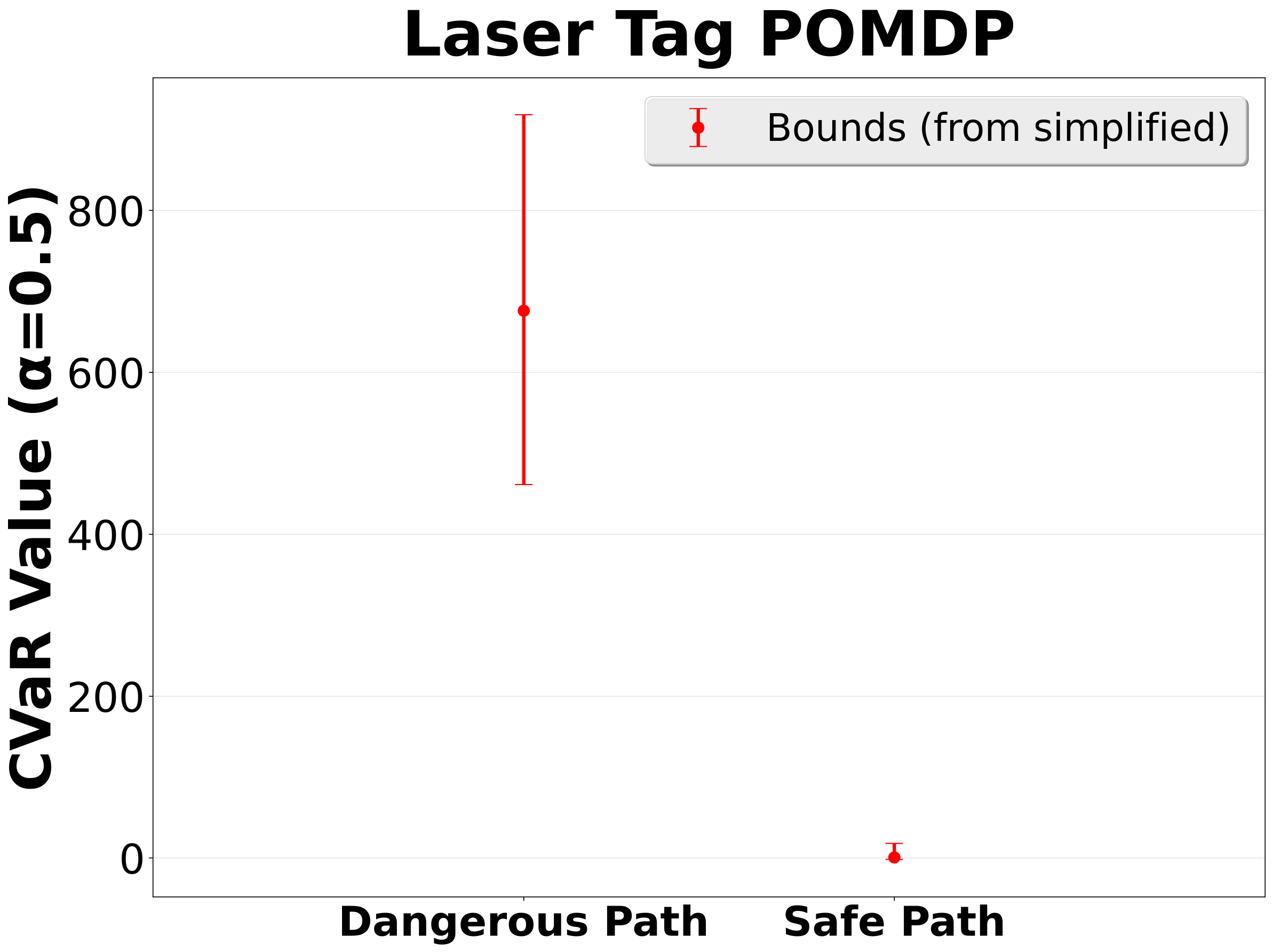}\label{fig:laser_tag_action_sequence_bounds}}
	\subfloat[]{\includegraphics[width=0.30\textwidth]{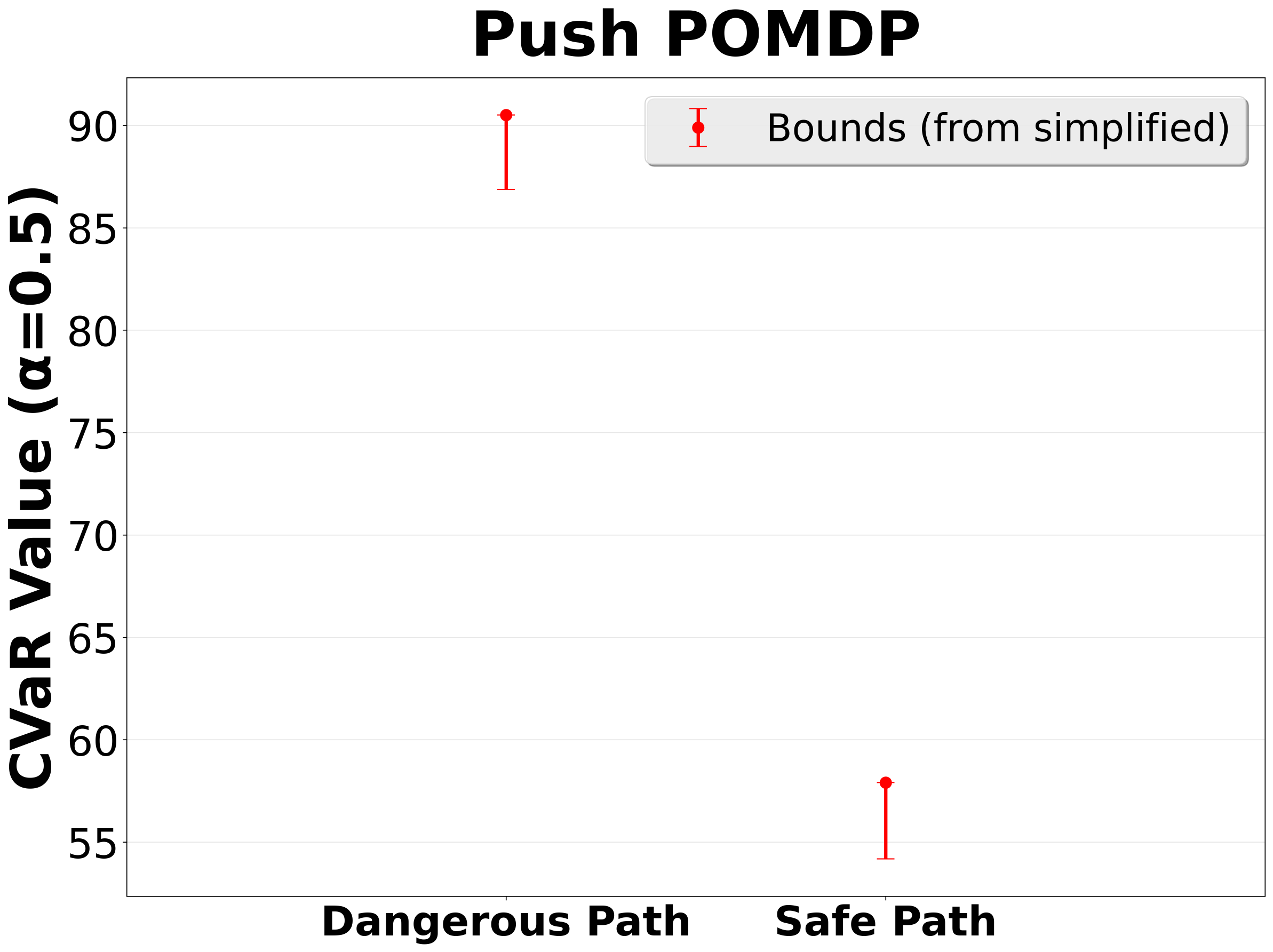}\label{fig:push_action_sequence_bounds}}
	\caption{
		Figures \ref{fig:light_dark_env}, \ref{fig:laser_tag_env}, and \ref{fig:push_env} illustrate pairs of agent trajectories, consisting of a safe (orange) and a dangerous (blue) path.
		In Figure \ref{fig:light_dark_env}, the agent moves from the lower-left to the upper-right goal location; the dangerous trajectory incurs higher cost.
		Figure \ref{fig:laser_tag_env} depicts trajectories in the Laser Tag environment, where the dangerous path traverses a high-cost region indicated in red.
		Figure \ref{fig:push_env} shows a similar pair of trajectories, with the dangerous path entering a hazardous region that leads to a terminal state.
		Across all three environments, the theoretical bounds on the values of the corresponding action sequences successfully discriminate between the safe and dangerous trajectories.
		Bounds are computed with probability of error $\delta = 0.05$.
	}
	\label{fig:agent_path}
\end{figure}

\begin{figure}[H]
	\centering
	\subfloat[]{\includegraphics[width=0.30\textwidth]{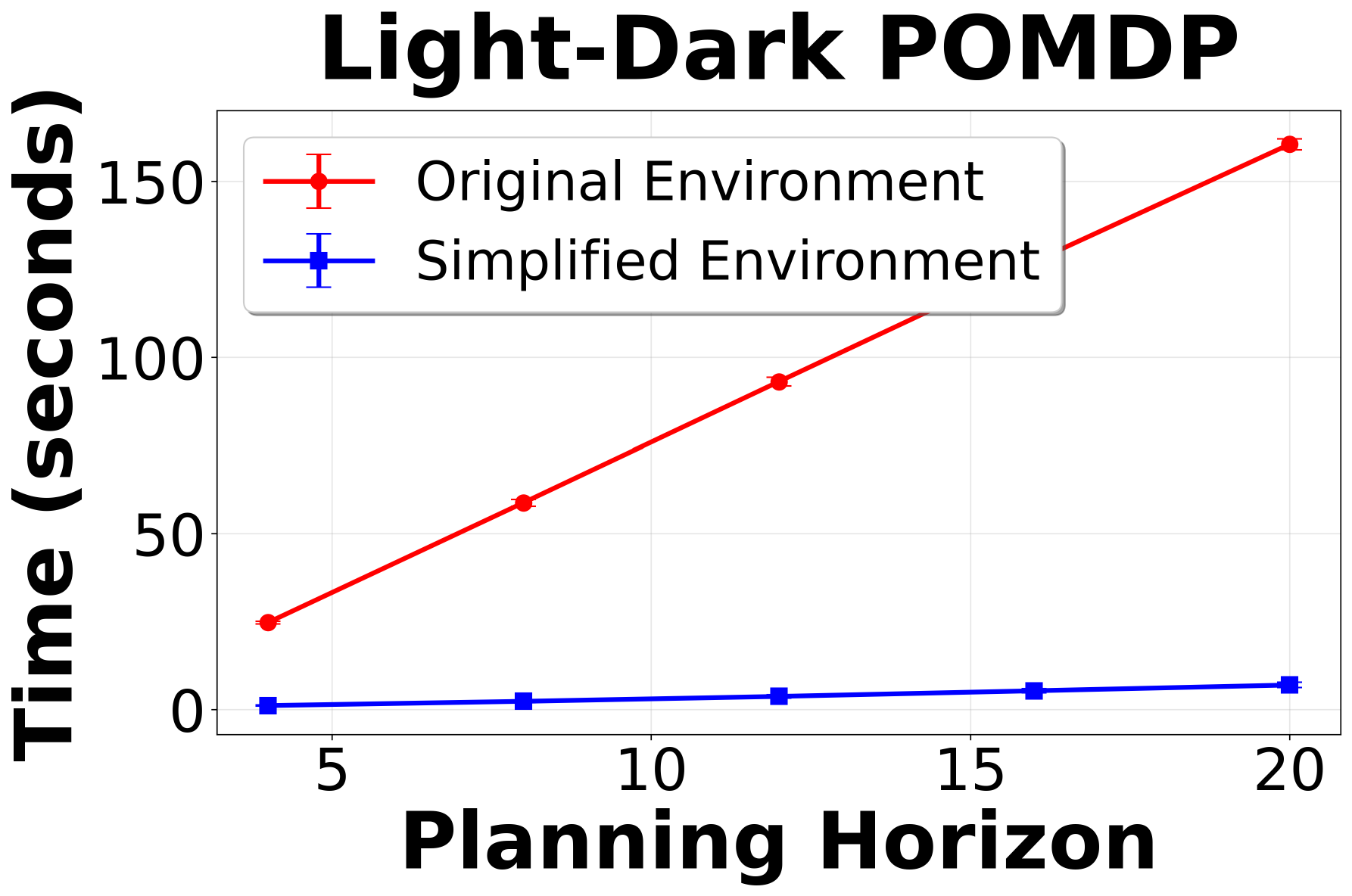}\label{fig:light_dark_horizon_time}}
	\subfloat[]{\includegraphics[width=0.30\textwidth]{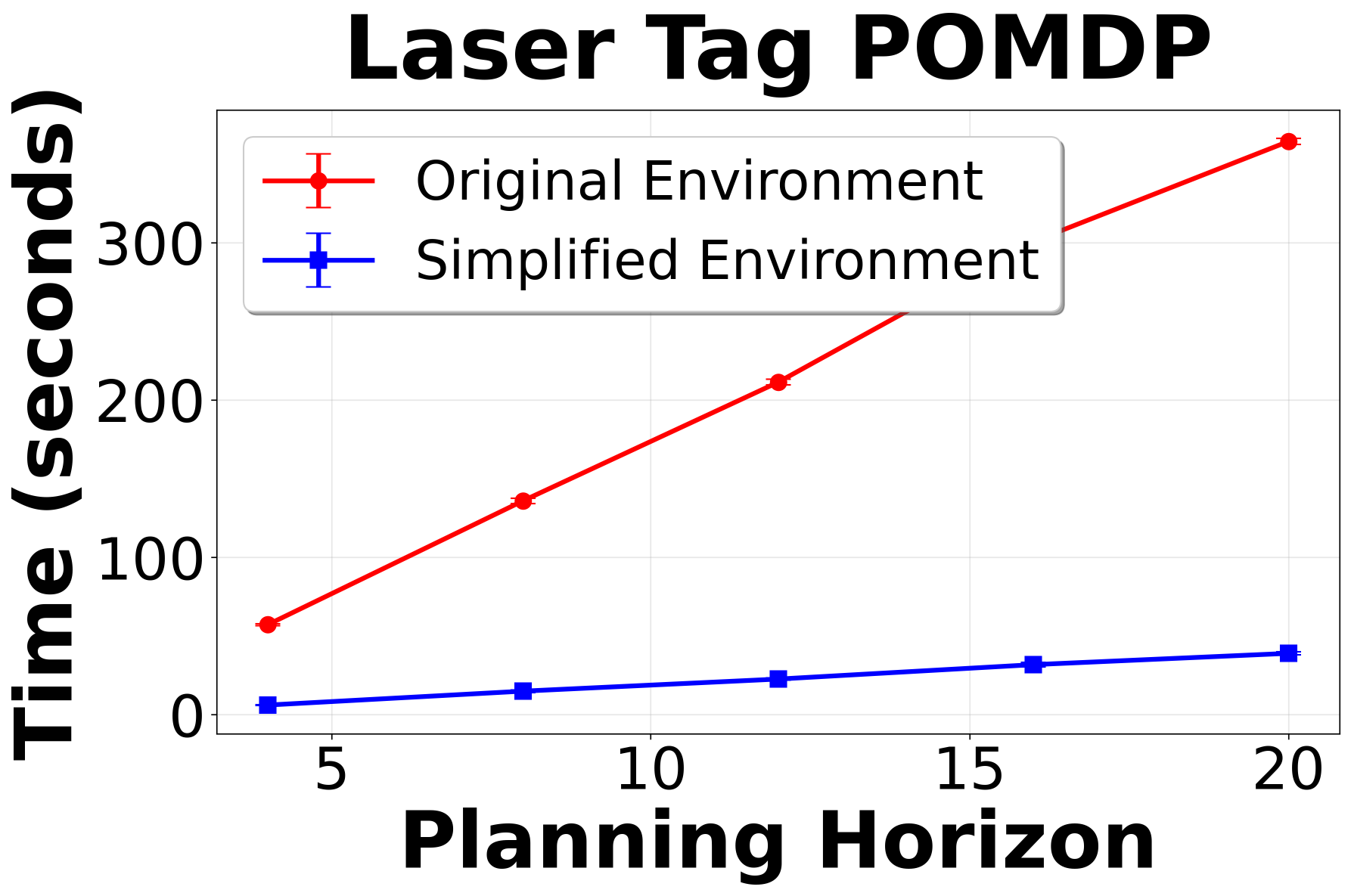}\label{fig:laser_tag_horizon_time}}
	\subfloat[]{\includegraphics[width=0.30\textwidth]{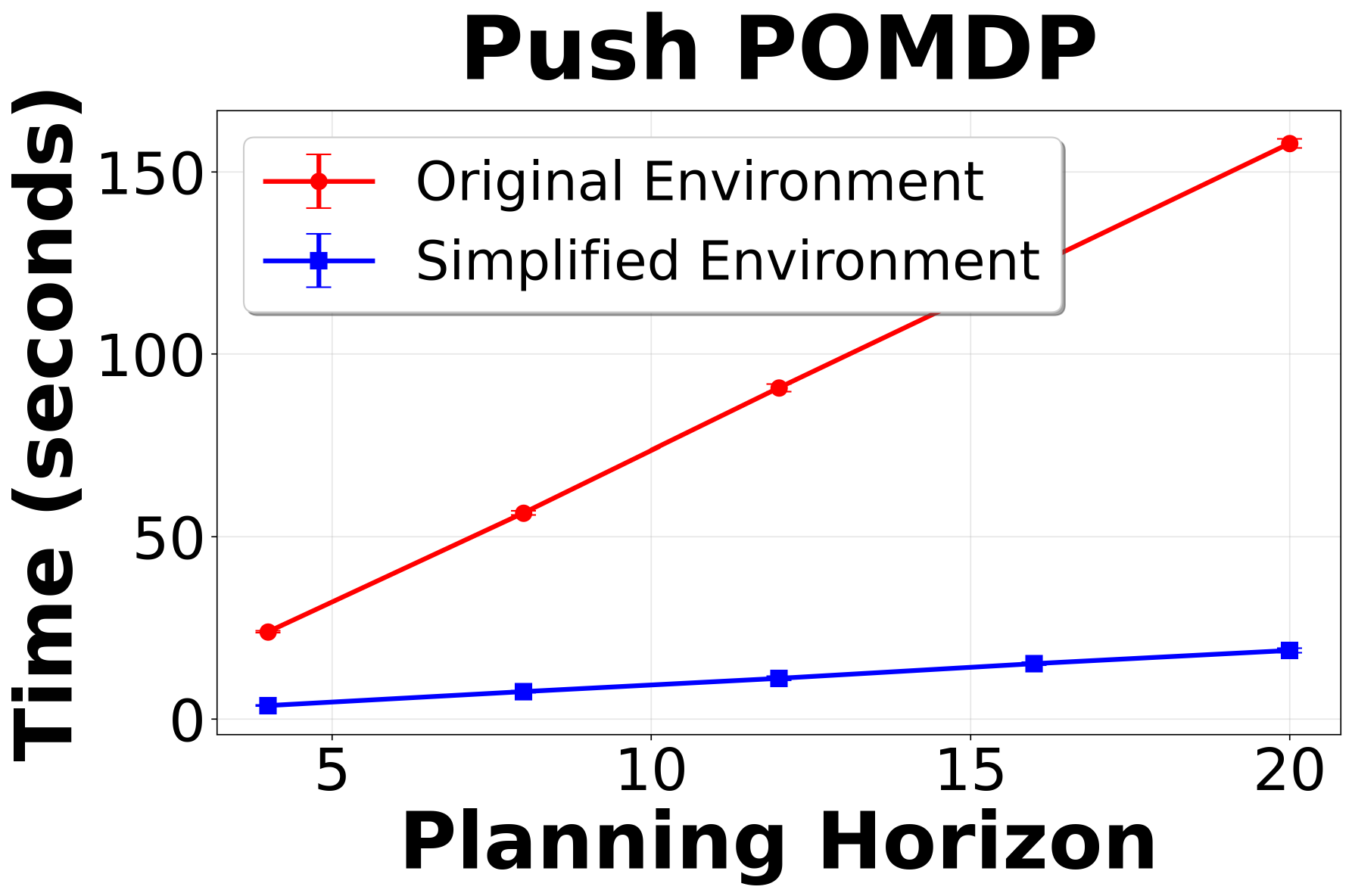}\label{fig:push_horizon_time}} \\[1em]
	\subfloat[]{\includegraphics[width=0.30\textwidth]{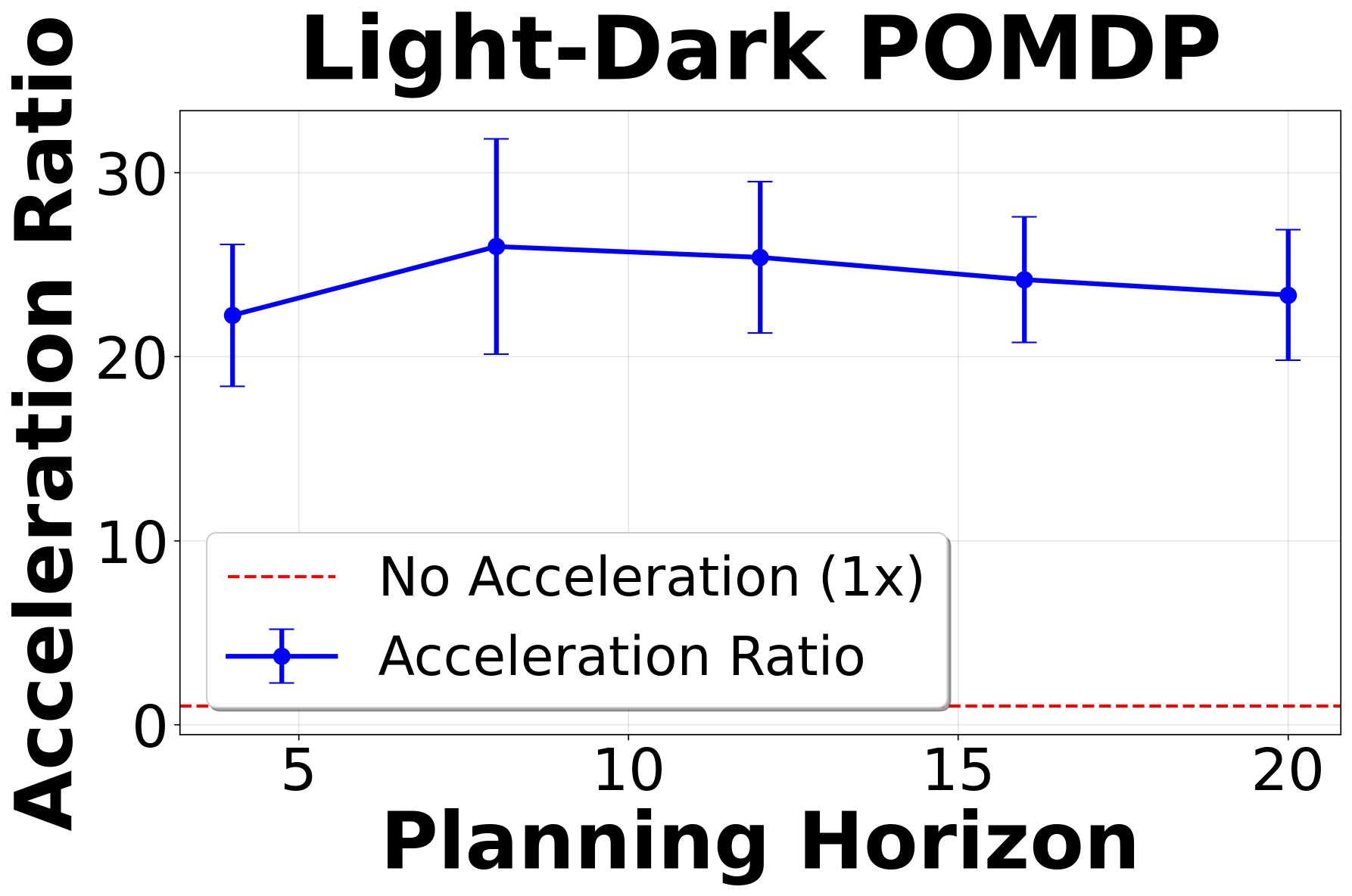}\label{fig:light_dark_horizon_acceleration}}
	\subfloat[]{\includegraphics[width=0.30\textwidth]{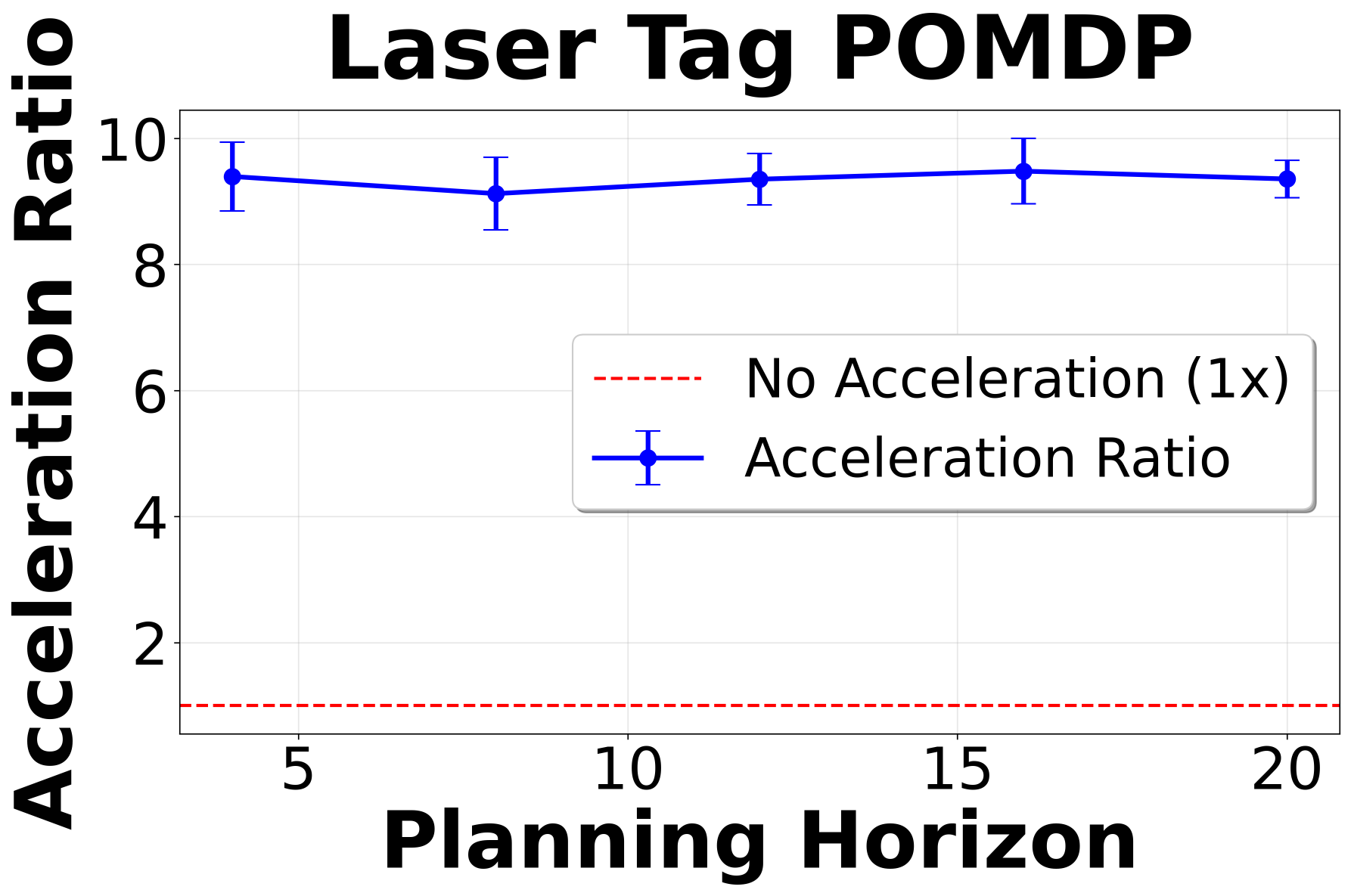}\label{fig:laser_tag_horizon_acceleration}}
	\subfloat[]{\includegraphics[width=0.30\textwidth]{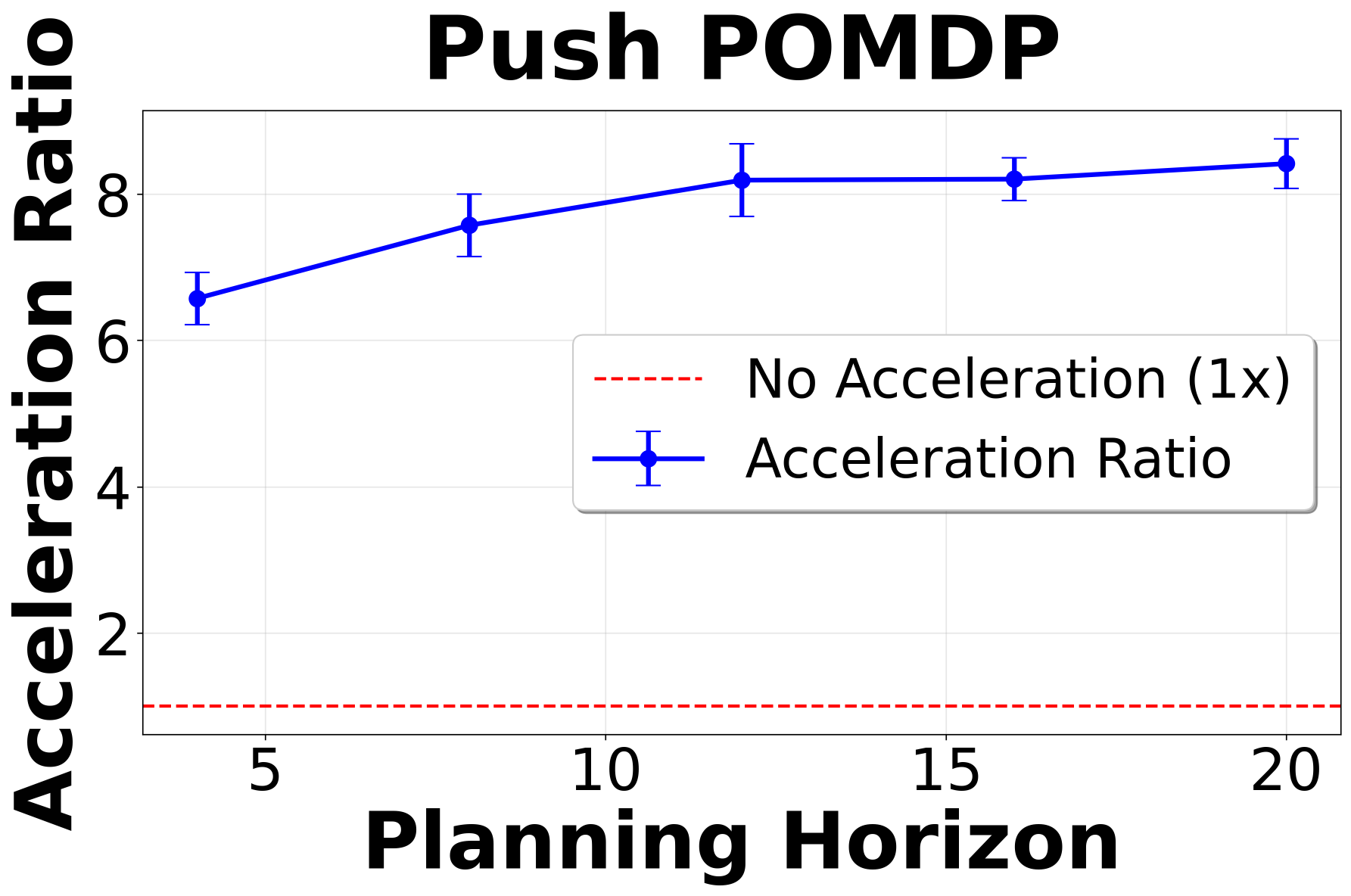}\label{fig:push_horizon_acceleration}} \\[1em]
	\subfloat[]{\includegraphics[width=0.30\textwidth]{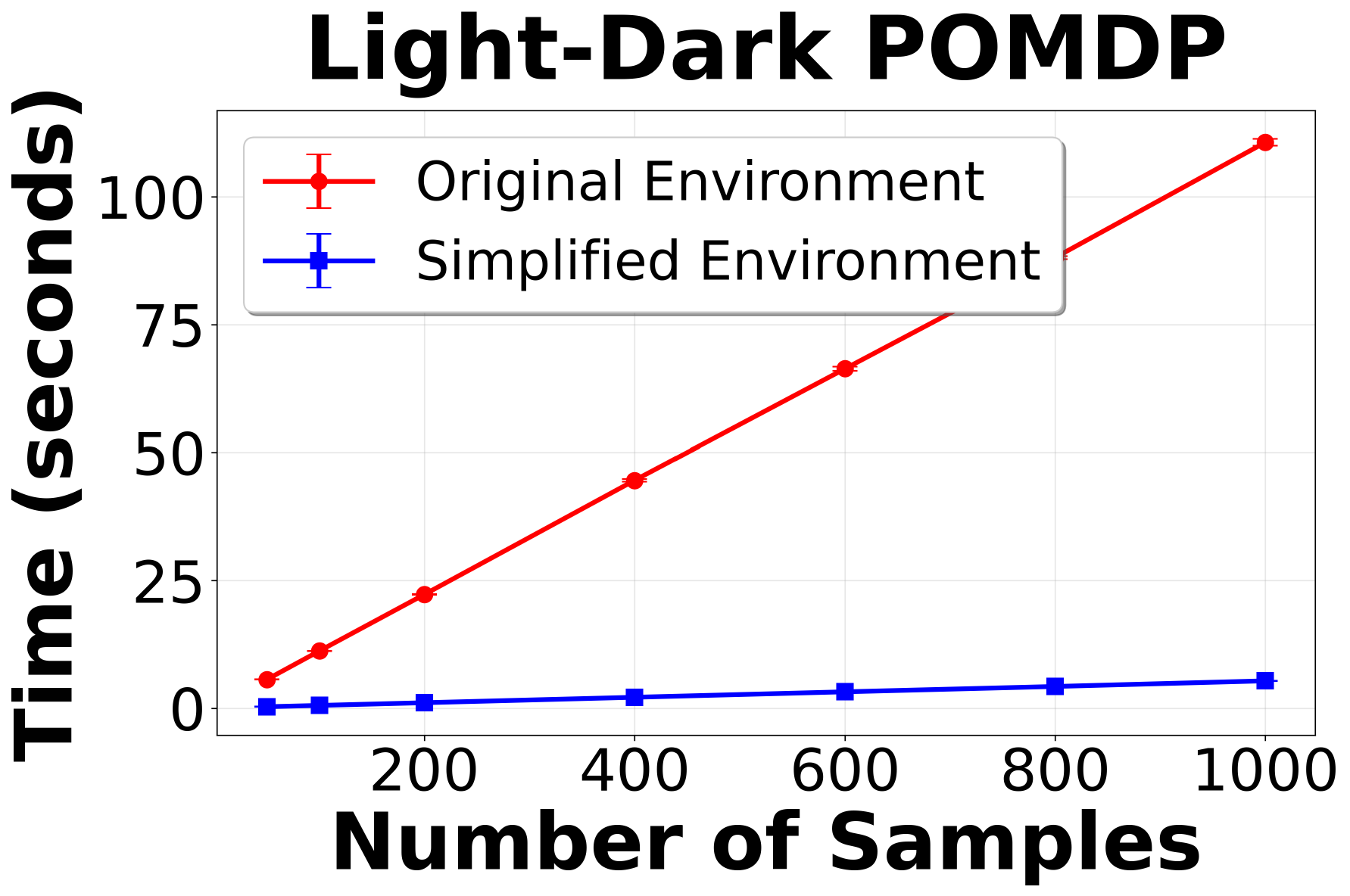}\label{fig:light_dark_n_return_samples_time}}
	\subfloat[]{\includegraphics[width=0.30\textwidth]{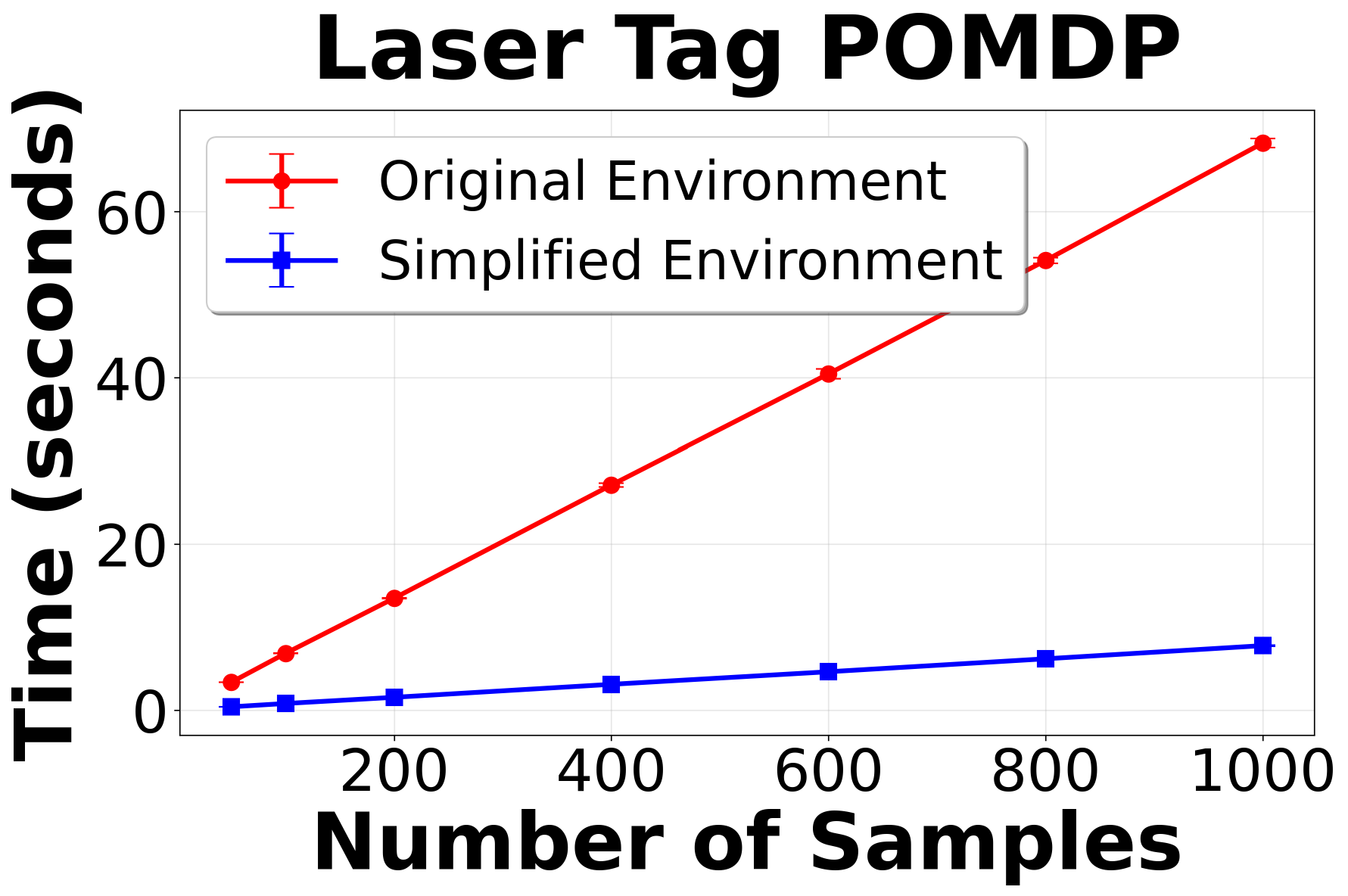}\label{fig:laser_tag_n_return_samples_time}}
	\subfloat[]{\includegraphics[width=0.30\textwidth]{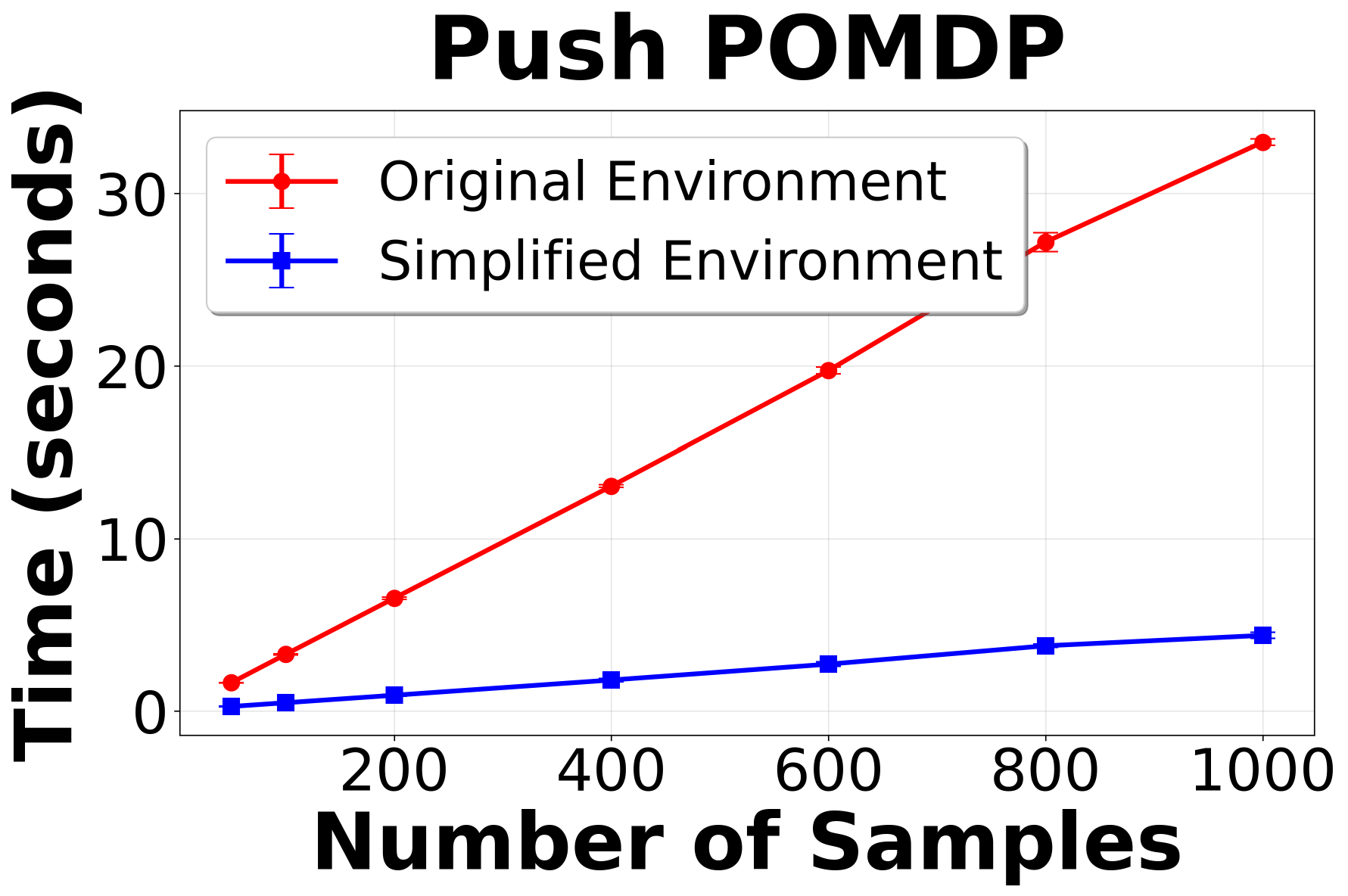}\label{fig:push_n_return_samples_time}} \\[1em]
	\subfloat[]{\includegraphics[width=0.30\textwidth]{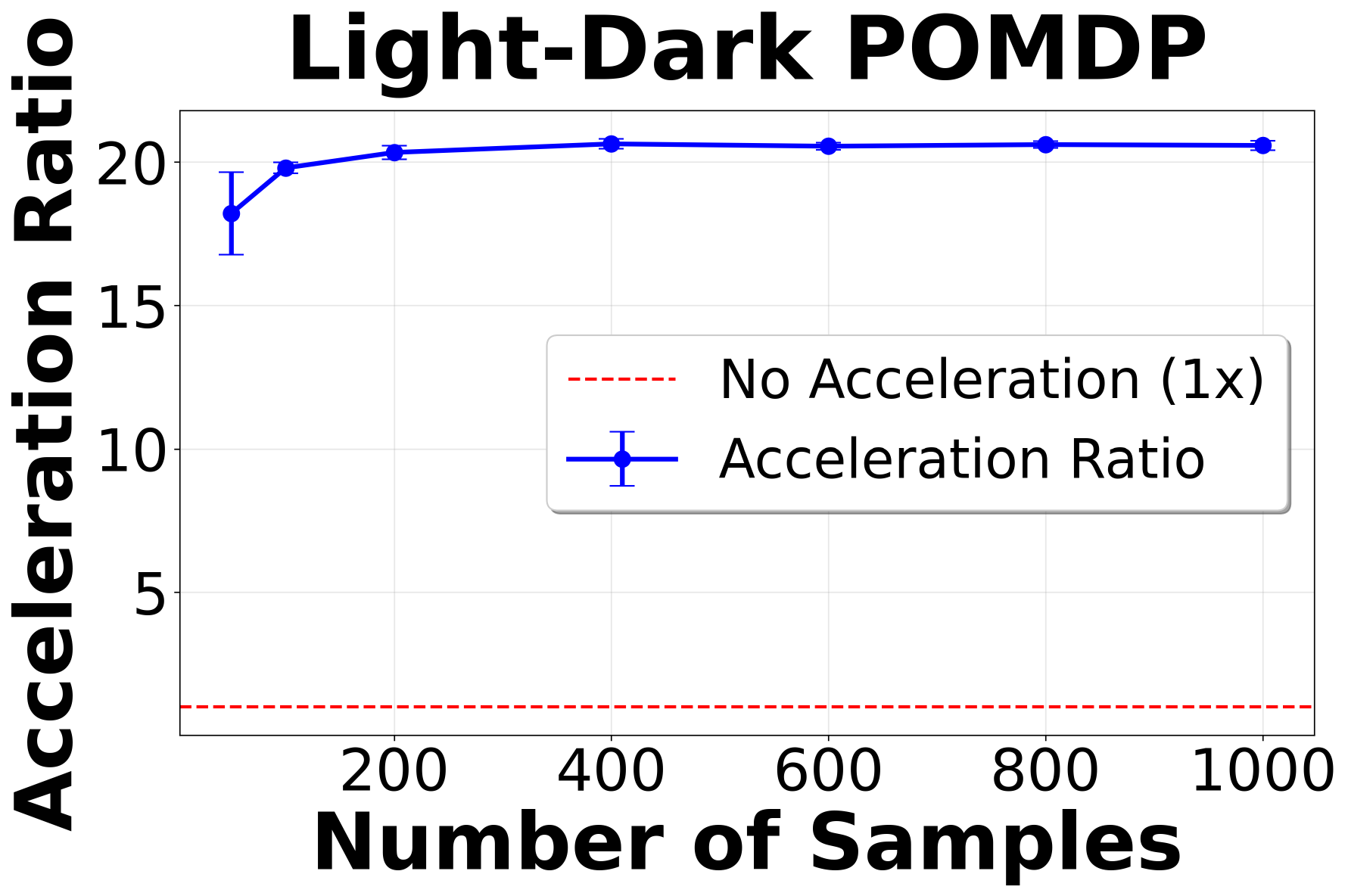}\label{fig:light_dark_n_return_samples_acceleration}}
	\subfloat[]{\includegraphics[width=0.30\textwidth]{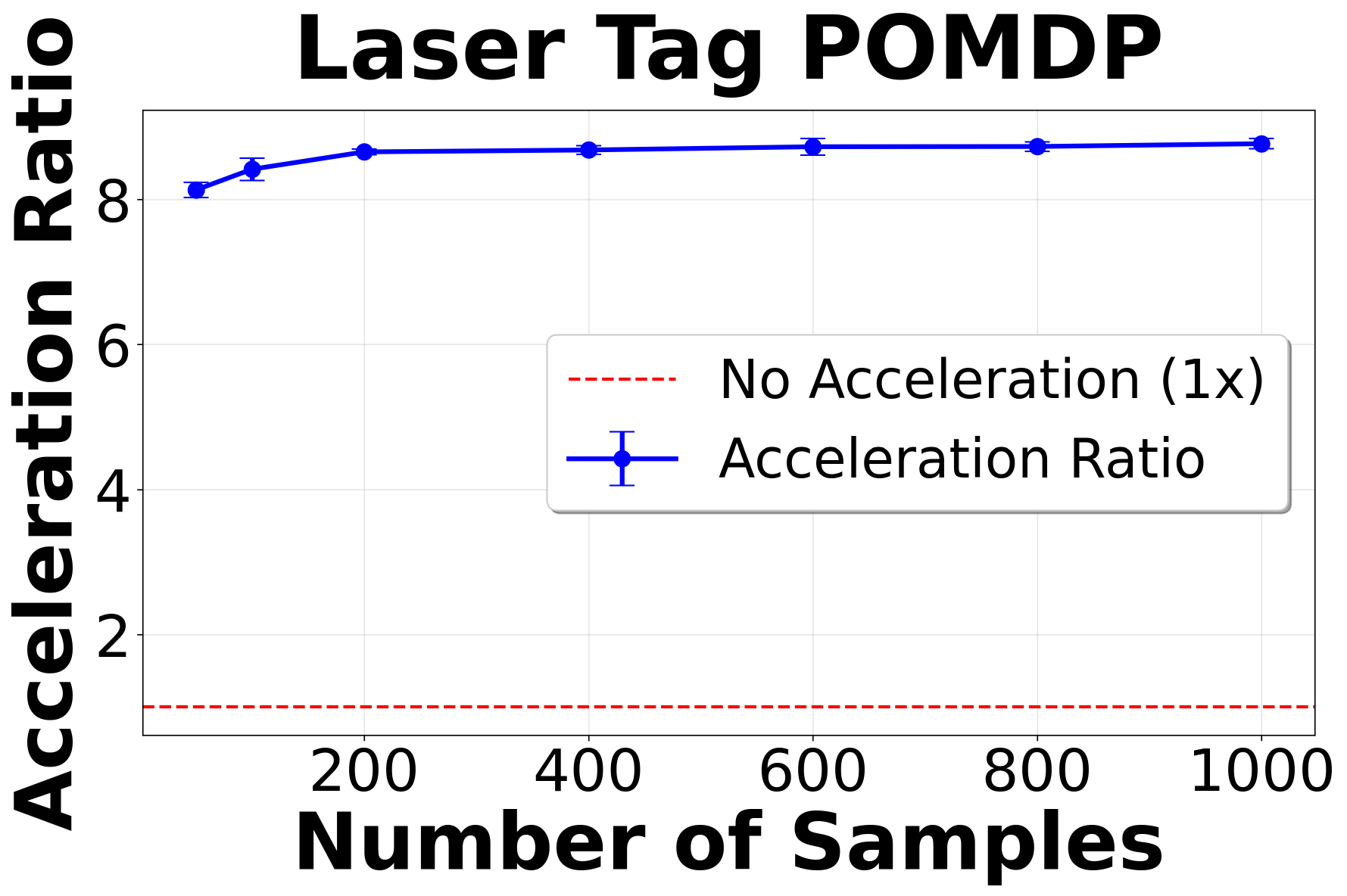}\label{fig:laser_tag_n_return_samples_acceleration}}
	\subfloat[]{\includegraphics[width=0.30\textwidth]{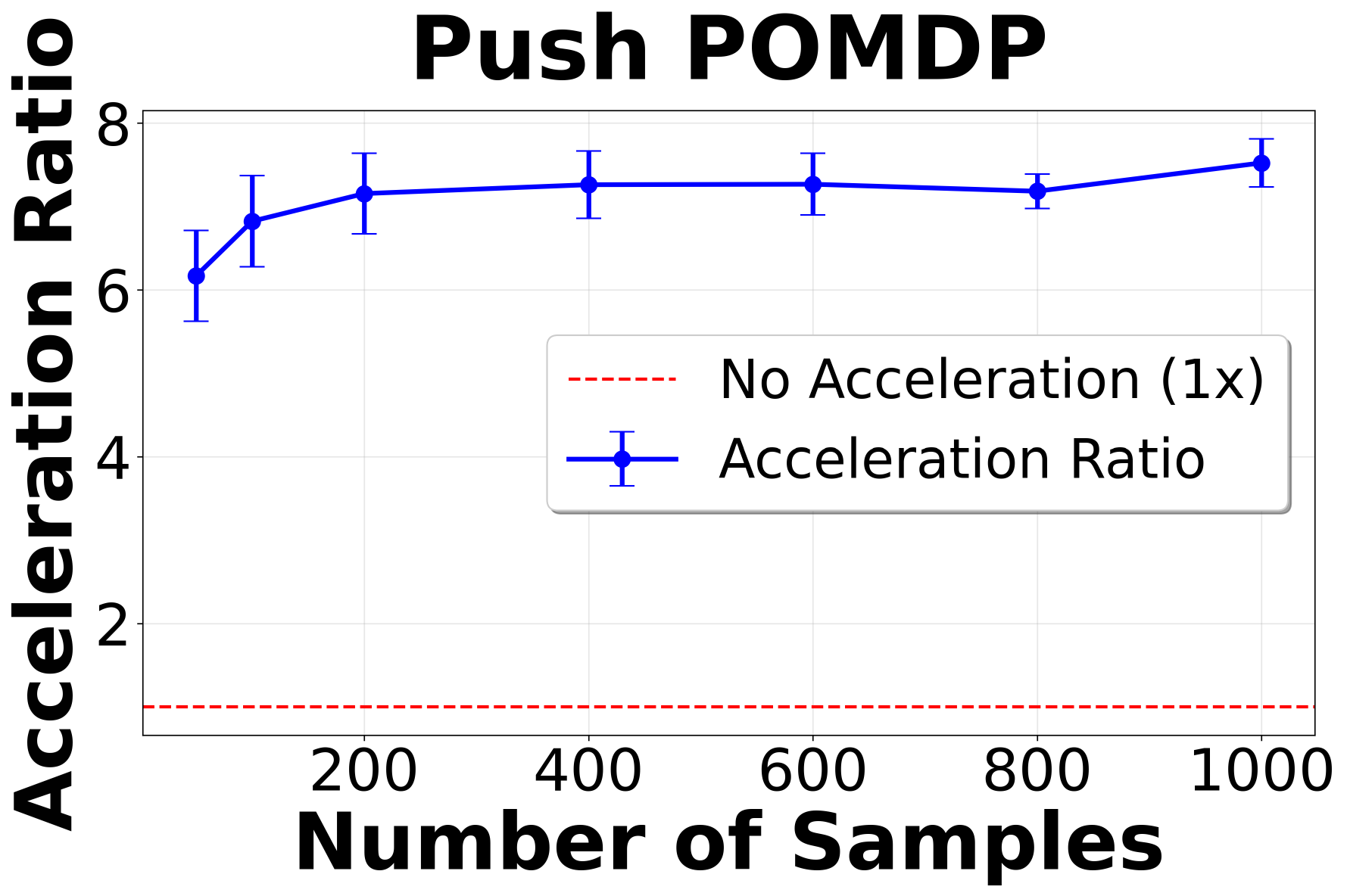}\label{fig:push_n_return_samples_acceleration}}
	\caption{Sensitivity analysis of the computational acceleration with respect to the planning horizon and number of return samples, for Light-Dark (left), Laser Tag (center), and Push (right), with $\alpha=0.5$. Rows~1--2: computation time and acceleration ratio as a function of the planning horizon. Rows~3--4: computation time and acceleration ratio as a function of the number of return samples. Error bars indicate 95\% confidence intervals.}
	\label{fig:acceleration_sensitivity}
\end{figure}

\begin{figure}[H]
	\centering
	\subfloat[]{\includegraphics[width=0.30\textwidth]{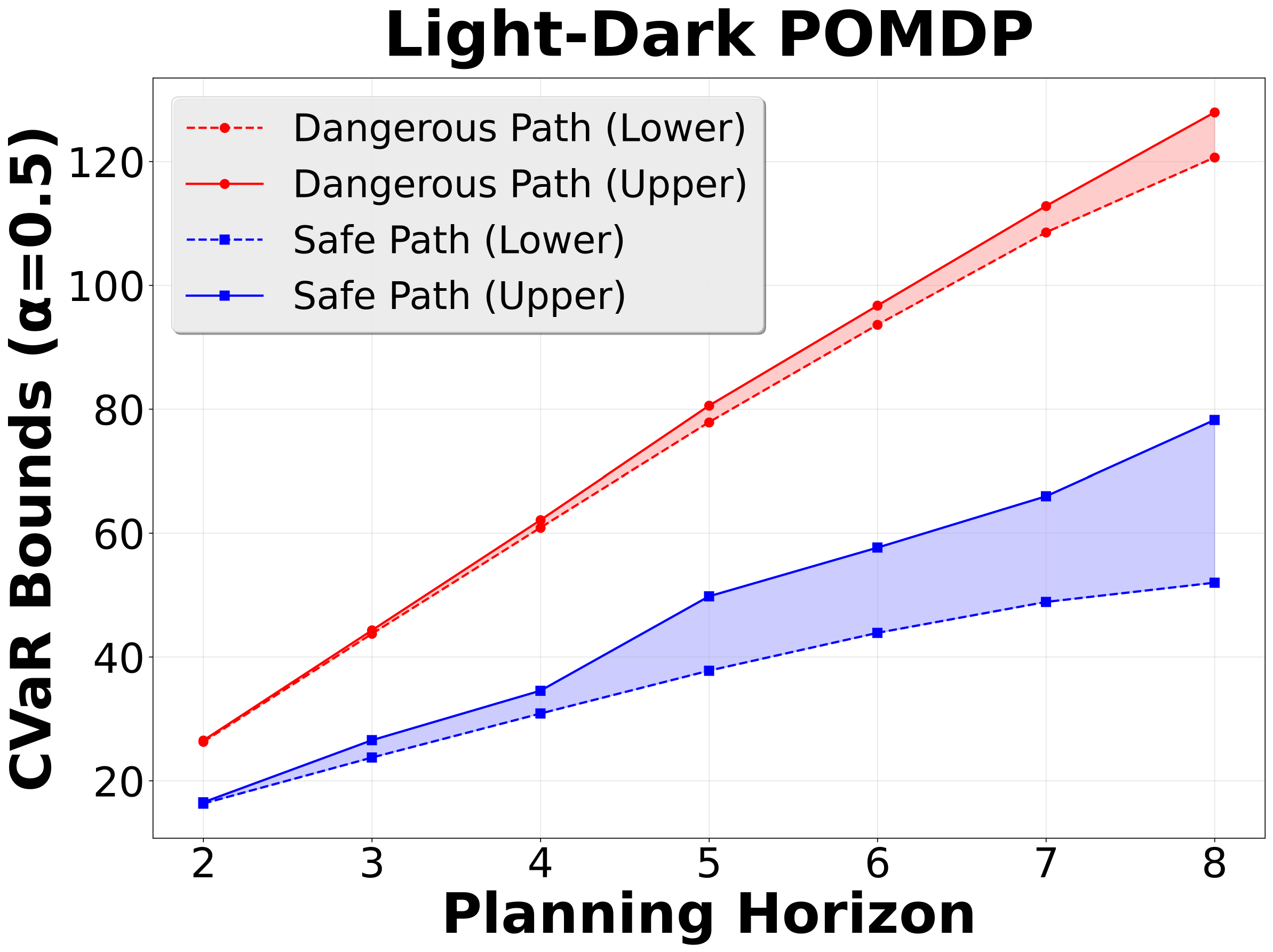}\label{fig:light_dark_horizon_bounds}}
	\subfloat[]{\includegraphics[width=0.30\textwidth]{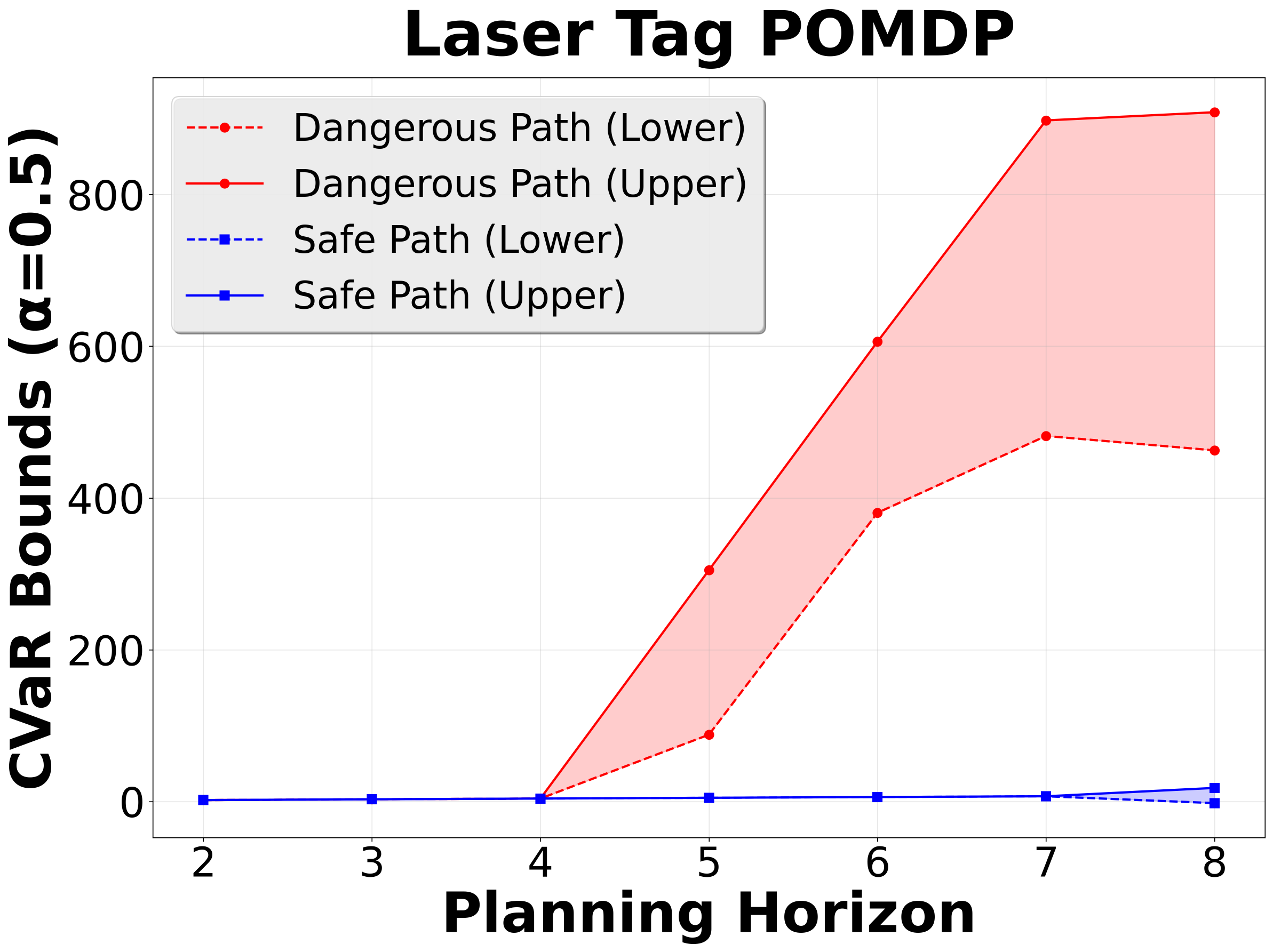}\label{fig:laser_tag_horizon_bounds}}
	\subfloat[]{\includegraphics[width=0.30\textwidth]{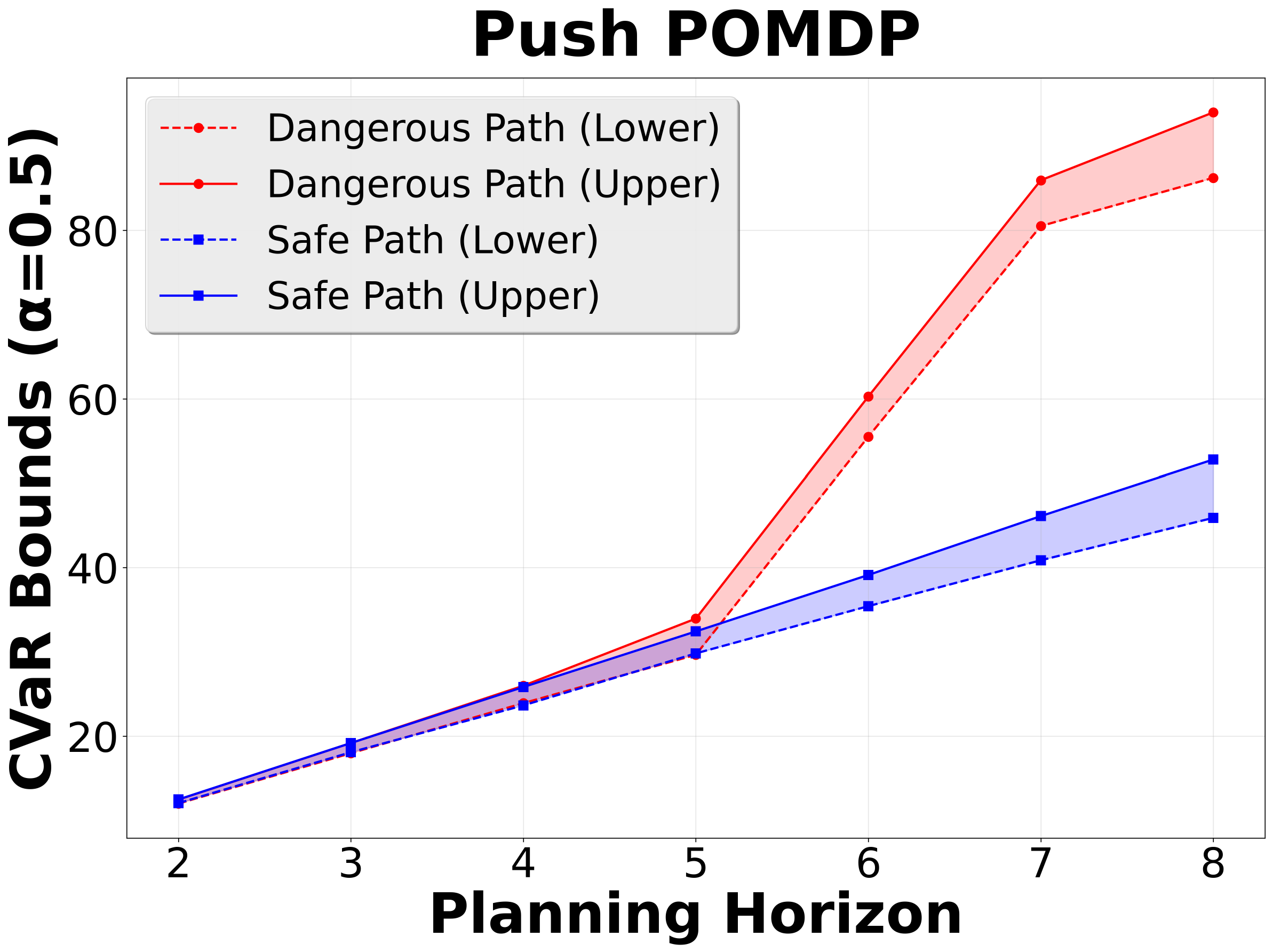}\label{fig:push_horizon_bounds}} \\[1em]
	\subfloat[]{\includegraphics[width=0.30\textwidth]{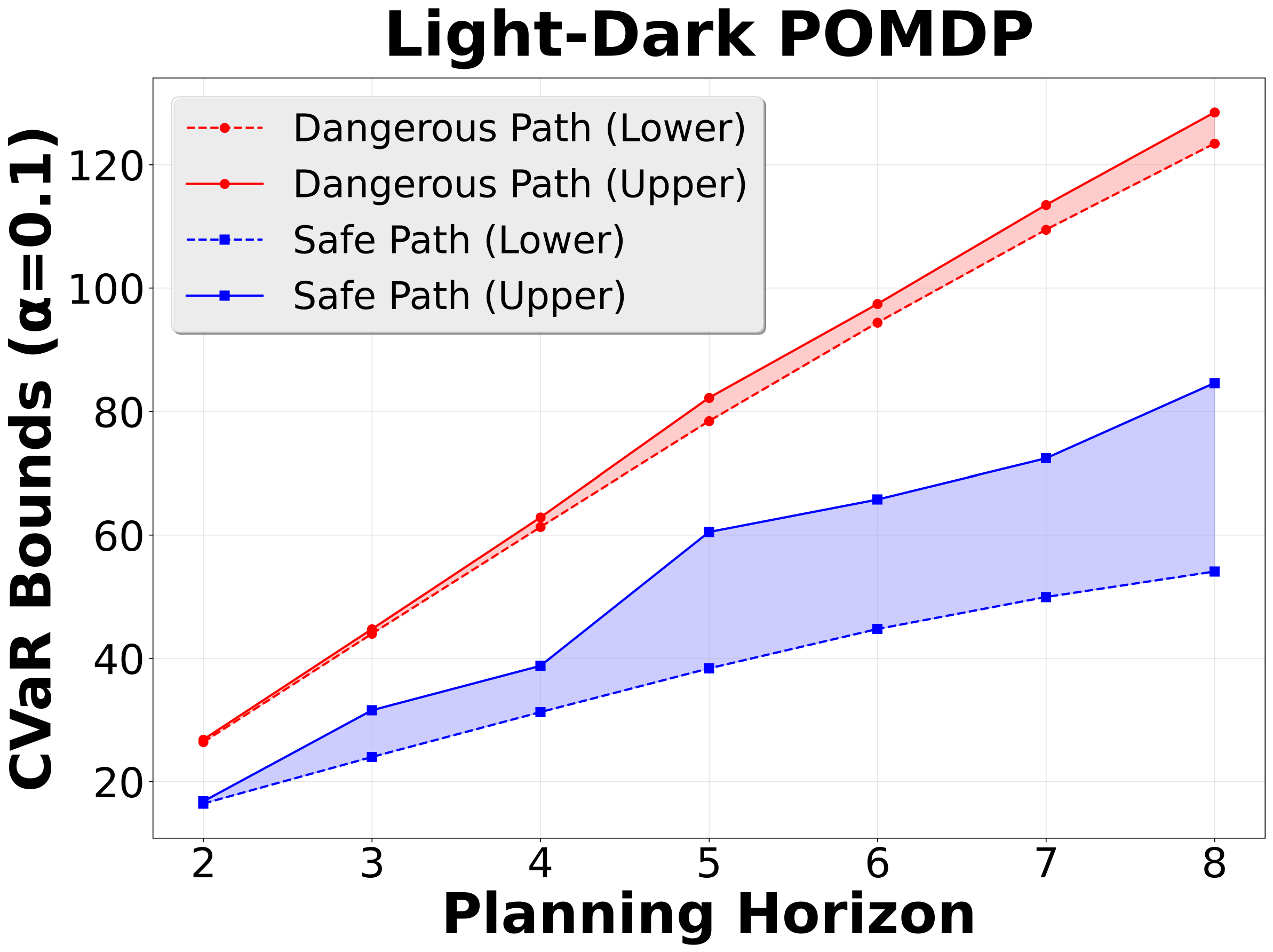}\label{fig:light_dark_horizon_bounds_alpha01}}
	\subfloat[]{\includegraphics[width=0.30\textwidth]{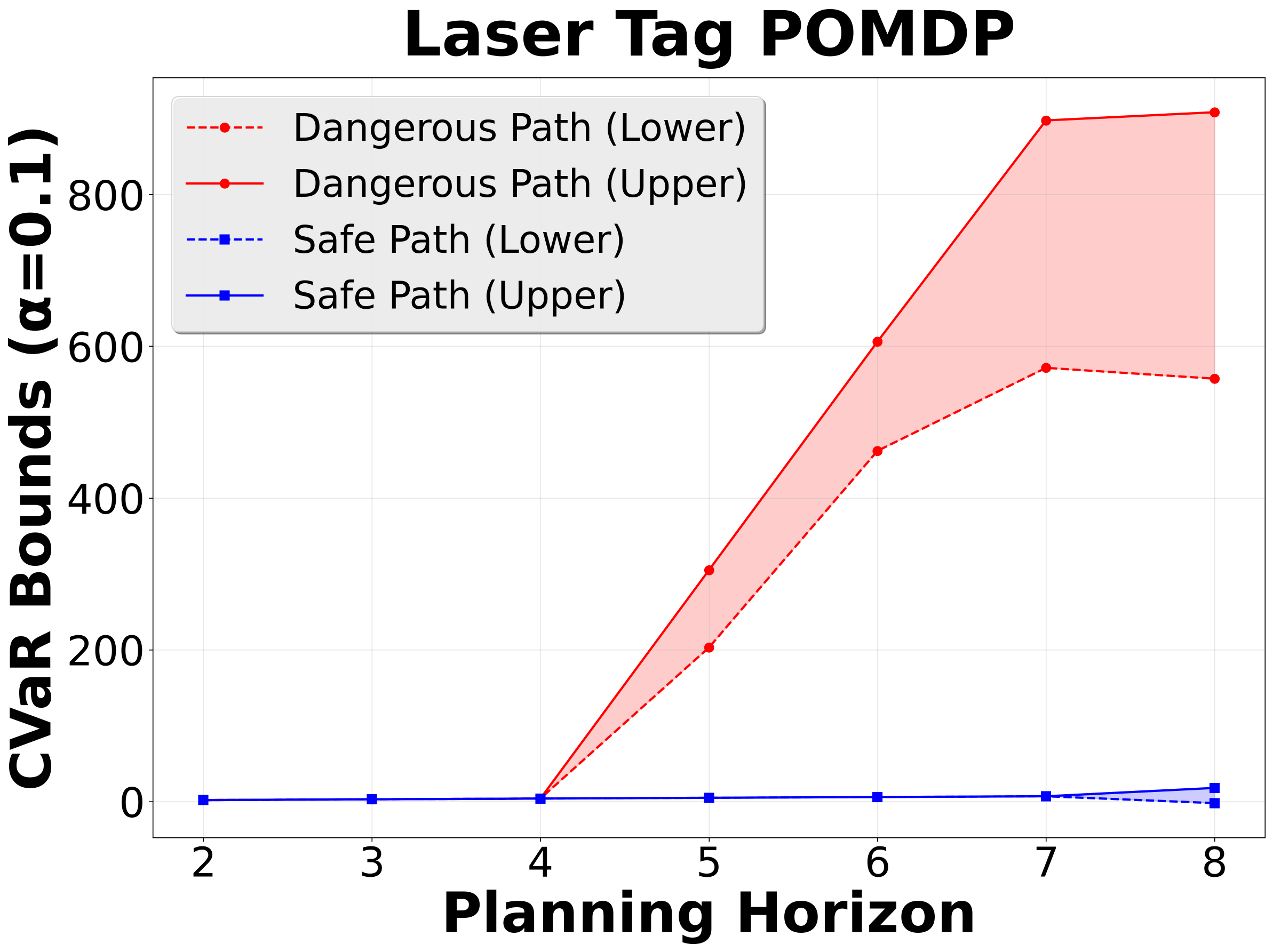}\label{fig:laser_tag_horizon_bounds_alpha01}}
	\subfloat[]{\includegraphics[width=0.30\textwidth]{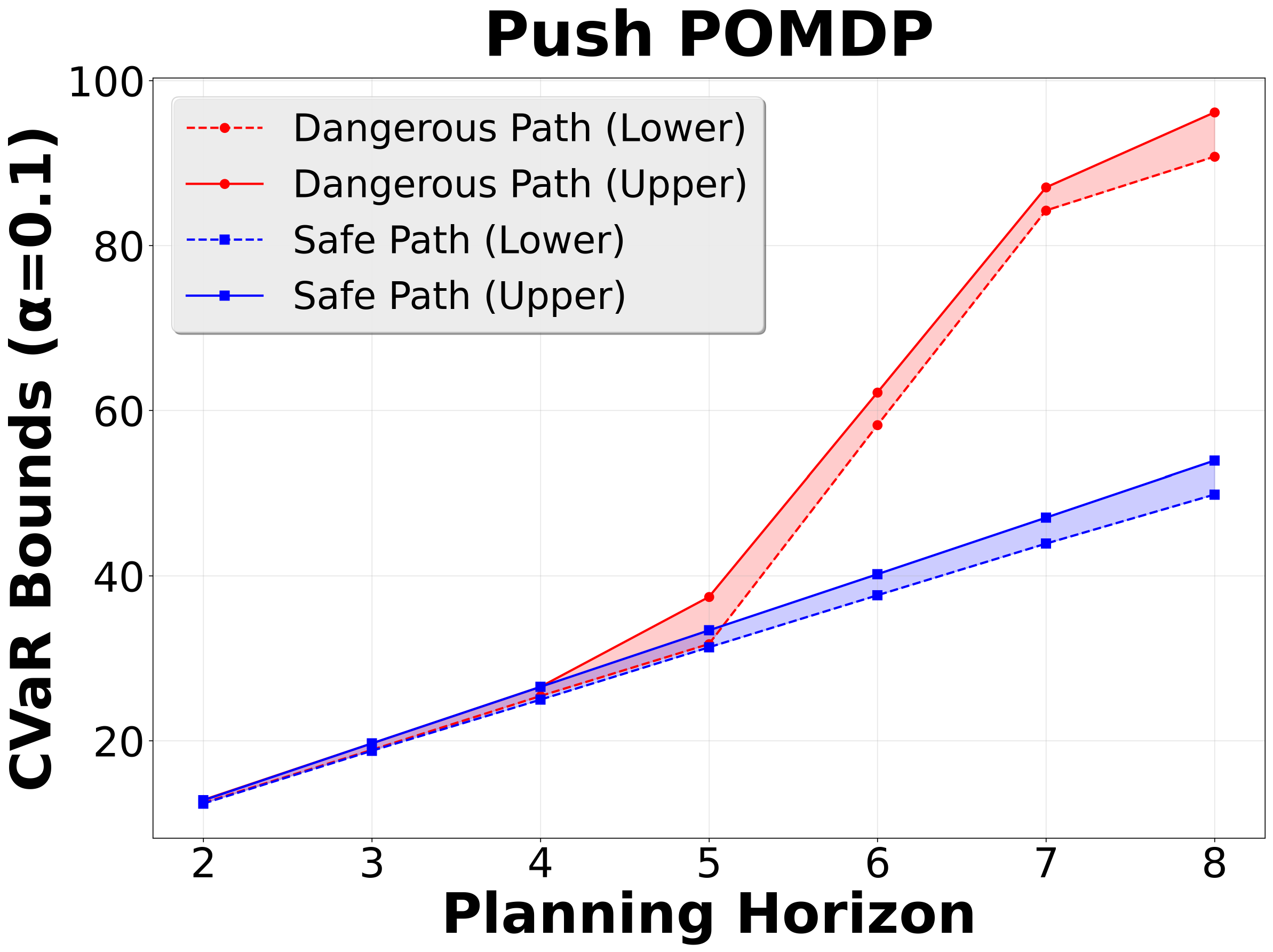}\label{fig:push_horizon_bounds_alpha01}} \\[1em]
	\subfloat[]{\includegraphics[width=0.30\textwidth]{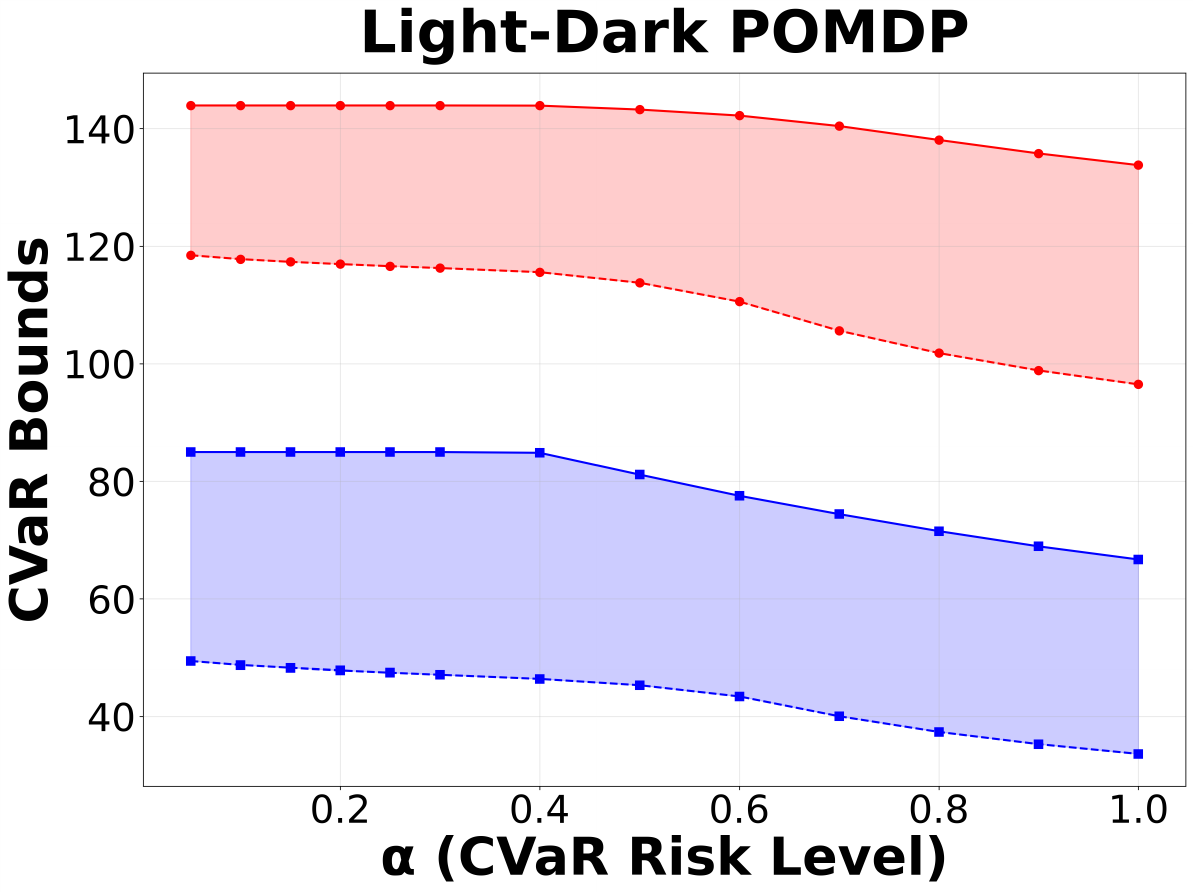}\label{fig:light_dark_alpha_bounds}}
	\subfloat[]{\includegraphics[width=0.30\textwidth]{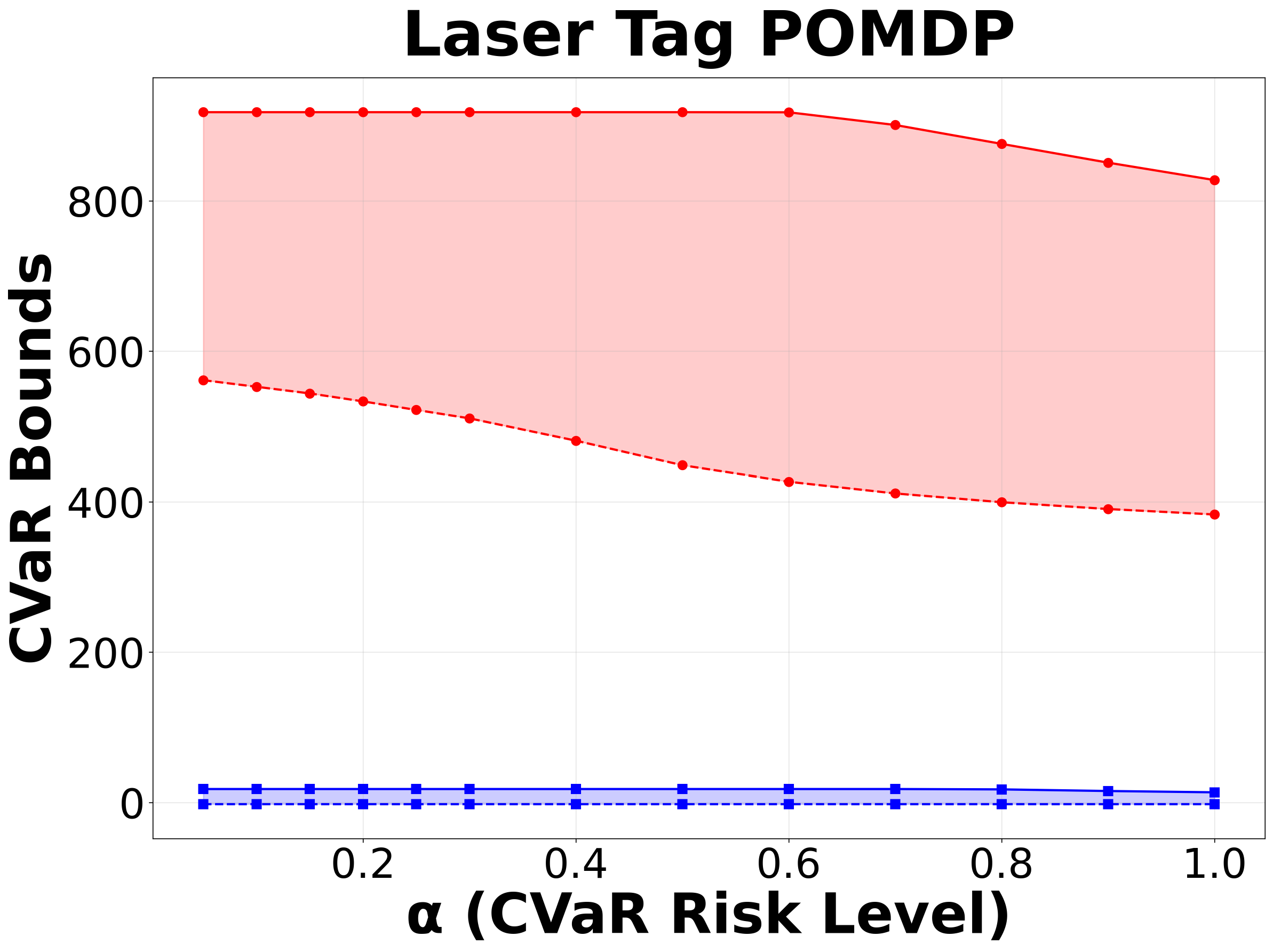}\label{fig:laser_tag_alpha_bounds}}
	\subfloat[]{\includegraphics[width=0.30\textwidth]{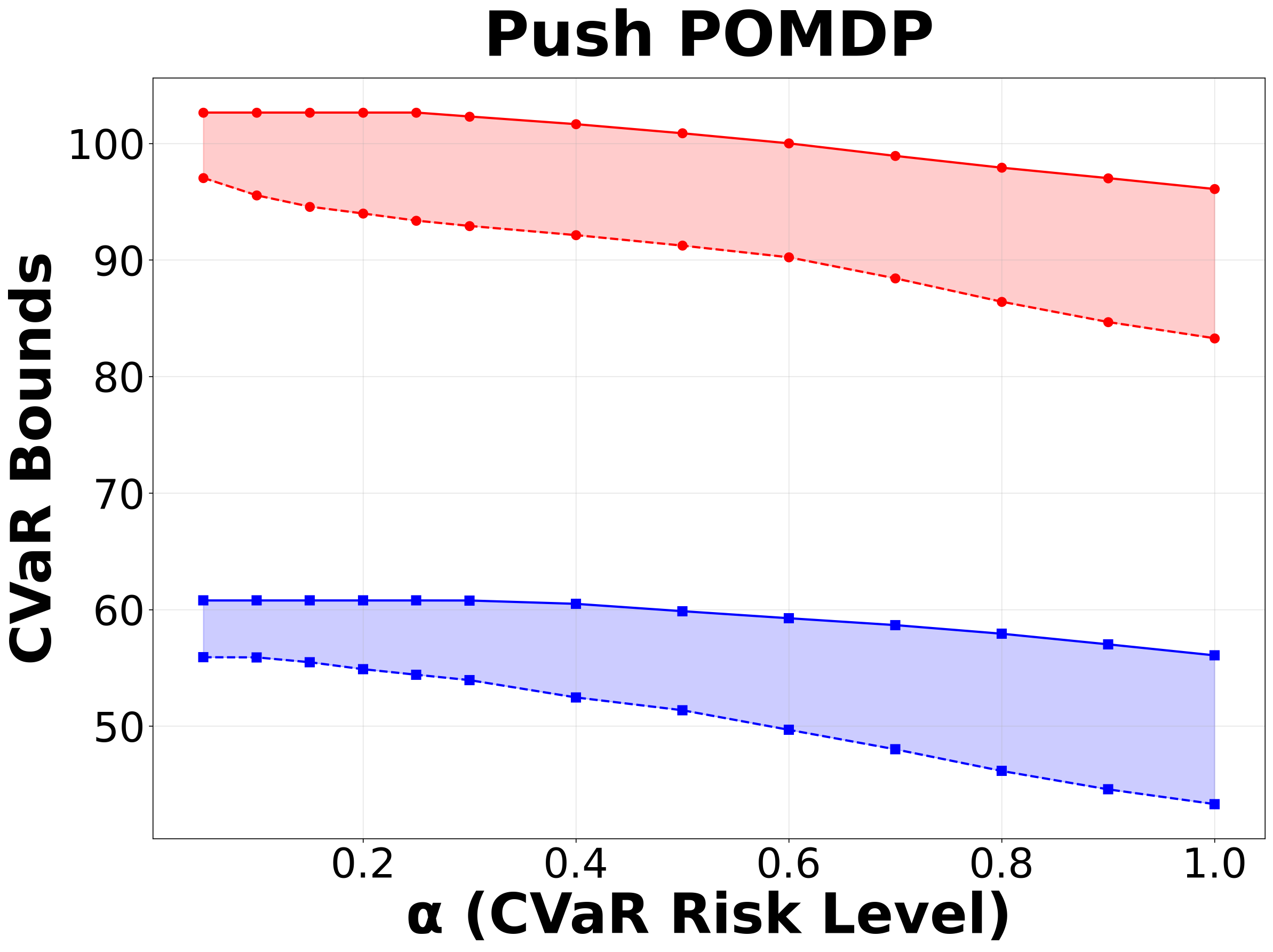}\label{fig:push_alpha_bounds}}
	\caption{Sensitivity of the CVaR bounds for the Light-Dark (left), Laser Tag (center), and Push (right) environments, using the complete safe (blue) and dangerous (red) action sequences depicted in Figure~\ref{fig:agent_path}; solid and dashed lines indicate the upper and lower bounds, respectively. Rows~1--2: bounds as a function of the planning horizon with $\alpha=0.5$ and $\alpha=0.1$, respectively. Row~3: bounds as a function of the risk level~$\alpha$ at a fixed horizon. As the horizon increases, the accumulated distributional discrepancy grows, causing the bounds to widen. As $\alpha$ decreases (more risk-averse), the bounds increase and the bound intervals widen, yet they remain well separated between the safe and dangerous paths across the full range of~$\alpha$. Bounds are computed with probability of error $\delta = 0.05$.}
	\label{fig:horizon_bounds}
\end{figure}

\subsection{Empirical Static CVaR Bound Evaluation: Closed-Loop Policy}

The preceding evaluation assessed bounds on fixed open-loop action sequences. We now demonstrate that the framework extends to closed-loop planners, using the Light-Dark POMDP as a test case. Specifically, we use the neural network trained by BetaZero~\cite{moss2024betazero} as a deterministic policy: BetaZero learns offline approximations of the optimal policy and value function for POMDPs and combines them with online Monte Carlo tree search at test time. Here, we extract a deterministic closed-loop policy from the trained network by selecting at each belief state the action with minimum predicted cost. The BetaZero configuration used in our experiments is detailed in Appendix~\ref{sec:betazero_config}.

We evaluate two objectives. First, we test whether the bounds computed under the simplified observation model successfully distinguish between policies of differing risk levels. We compare the BetaZero policy against the same safe and dangerous action sequences from the preceding subsection. As shown in Figures~\ref{fig:closed_loop_policy_and_trajectories} and~\ref{fig:closed_loop_sensitivity}, the bounds clearly separate the dangerous sequence from both the safe sequence and the BetaZero policy, with non-overlapping bound intervals. The trajectory plot confirms this ordering: the dangerous sequence traverses the high-cost region, while the BetaZero policy navigates through the illuminated safe zones. This separation persists across the full range of risk levels $\alpha$ and planning horizons.

Second, we show that evaluating the bounds under the simplified observation model is substantially faster than under the original model. Figure~\ref{fig:closed_loop_speedup} shows a consistent ${\approx}5\times$ speedup, stable across planning horizons from $4$ to $12$ and return sample counts from $100$ to $1000$.

\begin{figure}[H]
	\centering
	\subfloat[]{\includegraphics[width=0.30\textwidth]{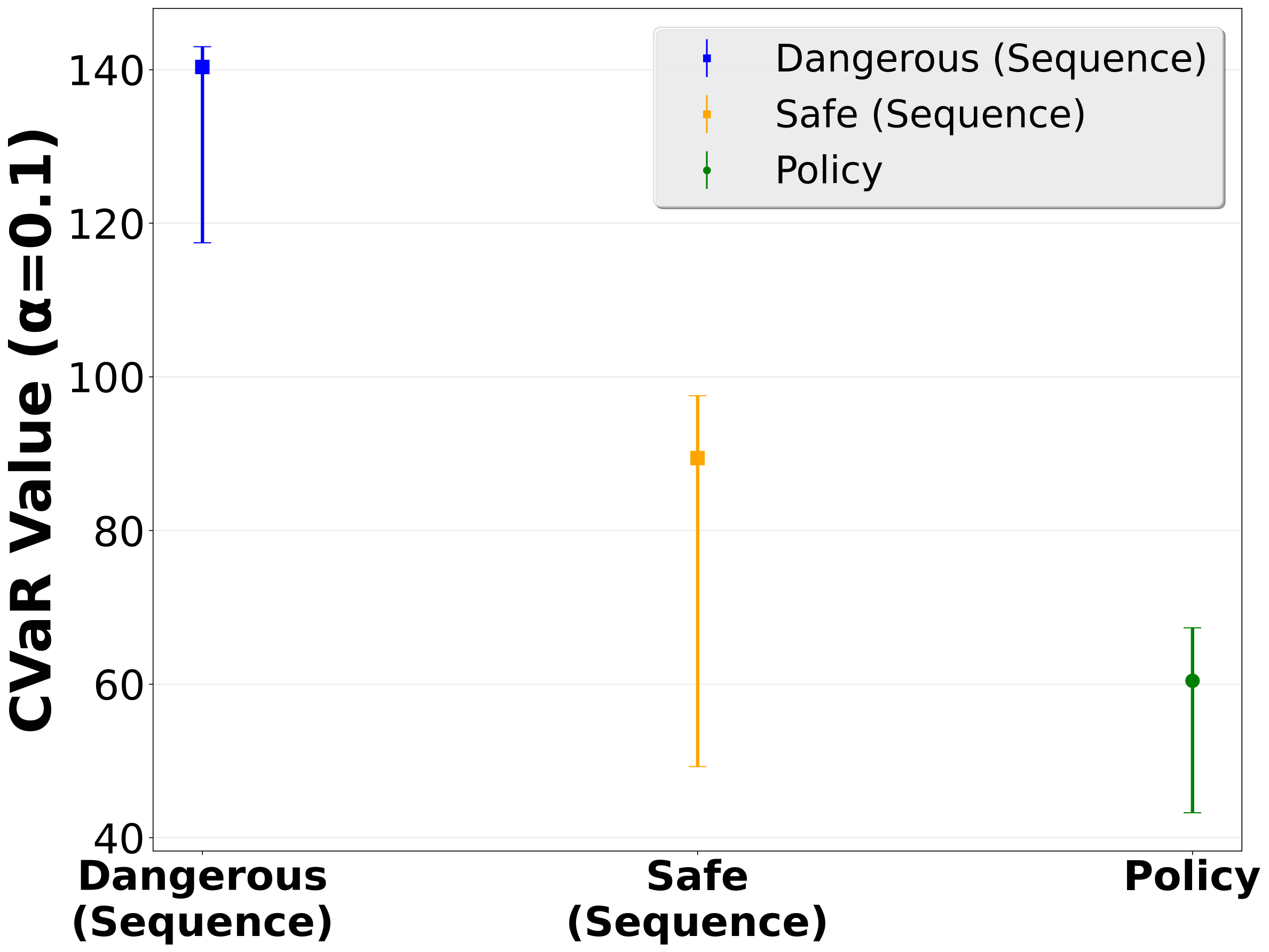}\label{fig:closed_loop_policy_eval}}
	\subfloat[]{\includegraphics[height=0.24\textwidth]{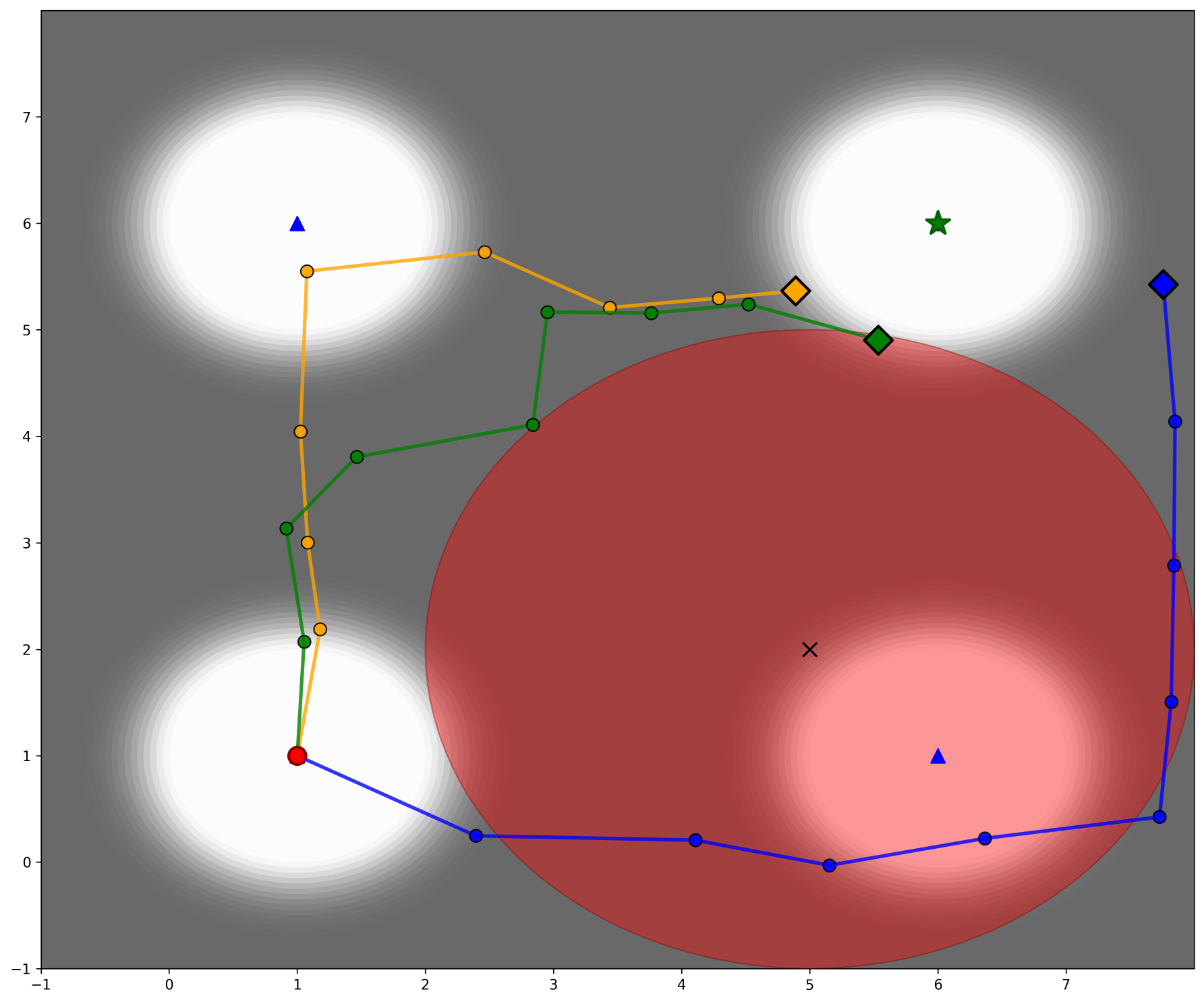}\label{fig:closed_loop_trajectories}}
	\caption{Evaluation of the BetaZero neural network policy against the safe and dangerous action sequences in the Light-Dark POMDP ($\alpha = 0.1$, $H = 9$), with bounds computed under the simplified observation model. Figure~\ref{fig:closed_loop_policy_eval}: CVaR bounds for the dangerous sequence (blue), safe sequence (orange), and BetaZero policy (green). The bound intervals of the policy and dangerous sequence do not overlap, demonstrating that the bounds correctly identify the policy as less risky. Bounds are computed with probability of error $\delta = 0.05$. Figure~\ref{fig:closed_loop_trajectories}: agent trajectories; the BetaZero policy (green) and safe sequence (orange) navigate through the illuminated safe zones, while the dangerous sequence (blue) traverses the high-cost region.}
	\label{fig:closed_loop_policy_and_trajectories}
\end{figure}

\begin{figure}[t]\begin{flushright}
		
	\end{flushright}
	\centering
	\subfloat[]{\includegraphics[width=0.38\textwidth]{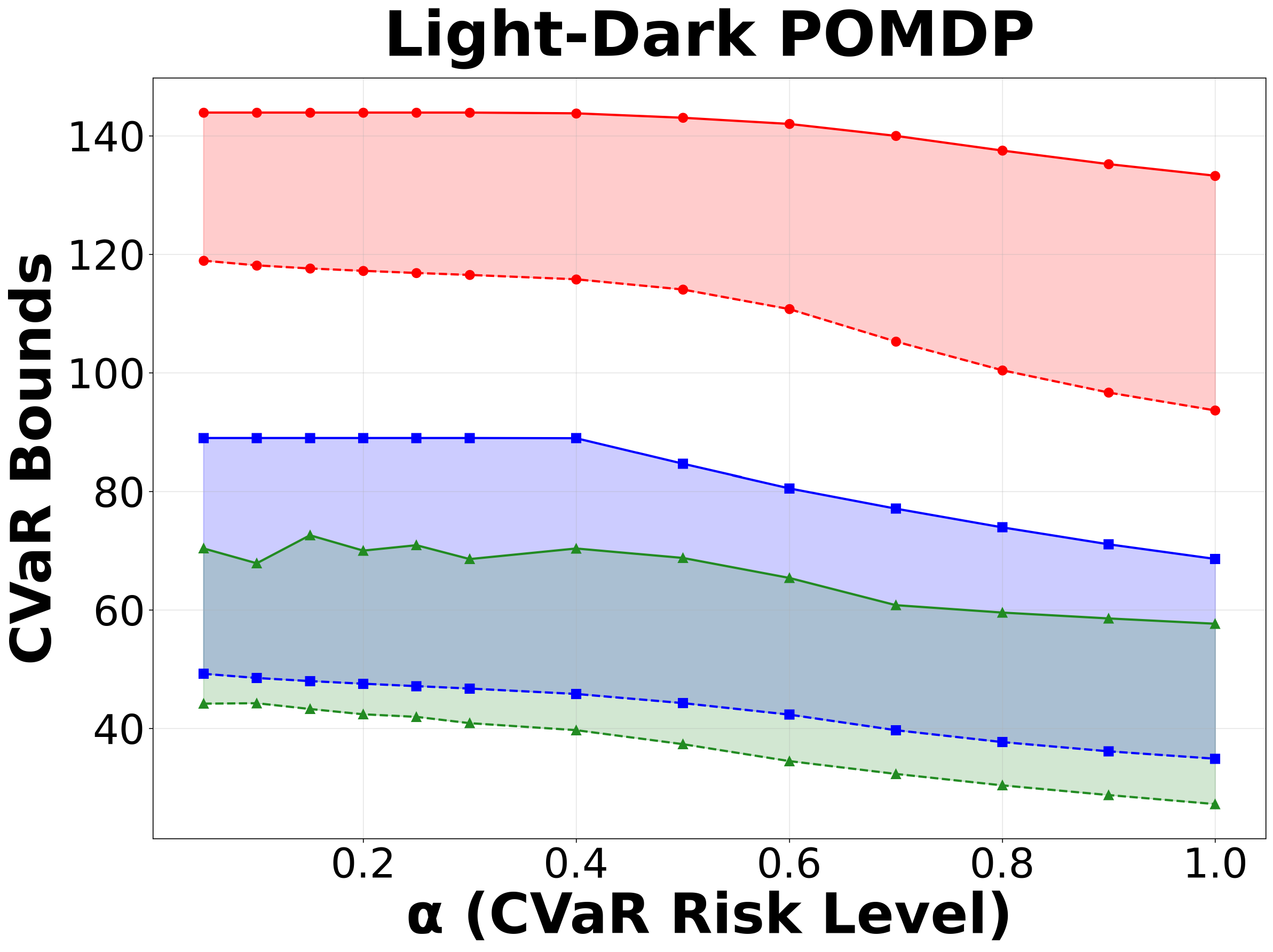}\label{fig:closed_loop_alpha_sensitivity}}
	\subfloat[]{\includegraphics[width=0.38\textwidth]{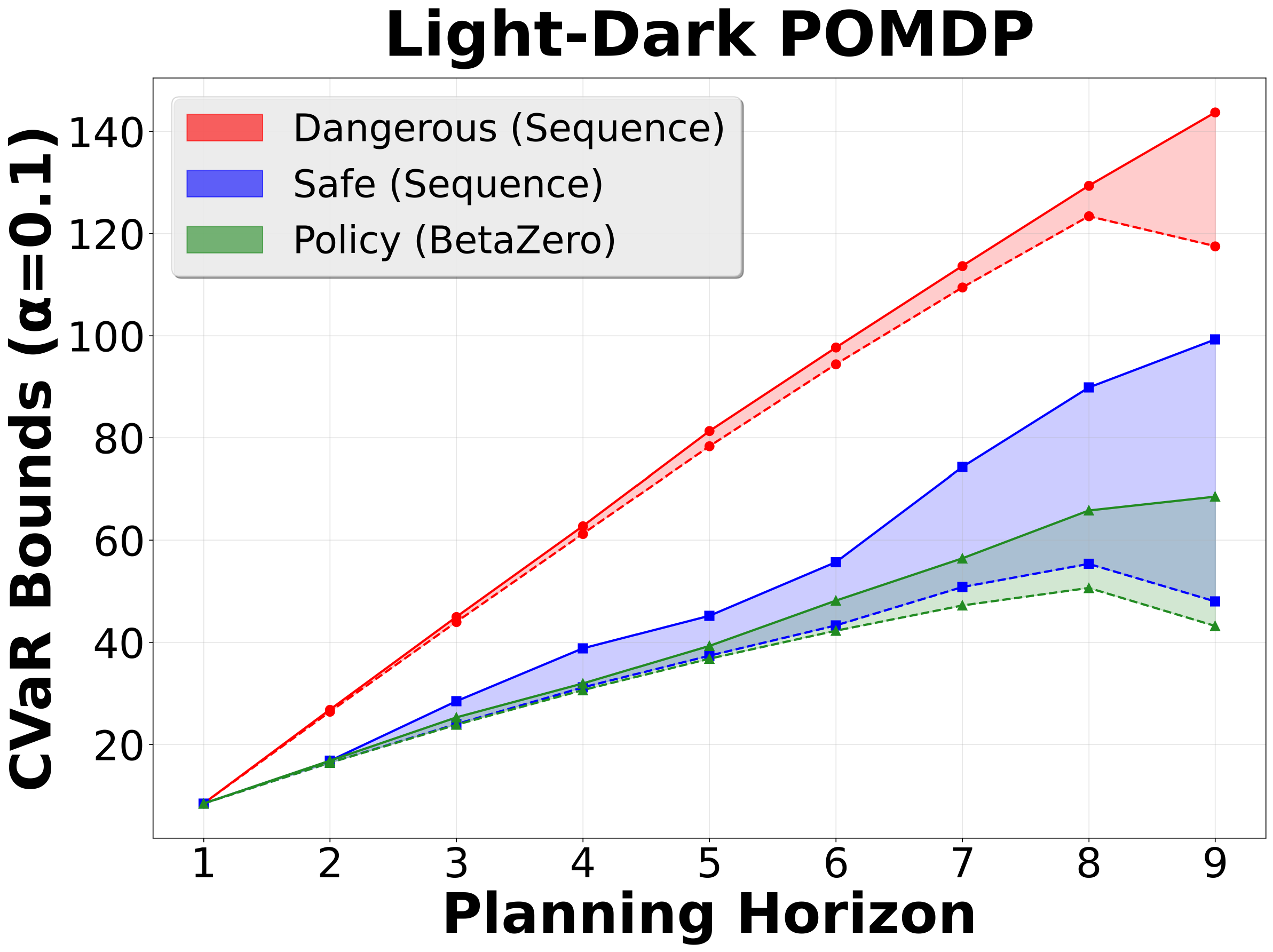}\label{fig:closed_loop_horizon_sensitivity}}
	\caption{Sensitivity of the CVaR bounds (computed under the simplified observation model) for the BetaZero policy, safe sequence, and dangerous sequence in the Light-Dark POMDP; solid and dashed lines denote upper and lower bounds, respectively. Figure~\ref{fig:closed_loop_alpha_sensitivity}: bounds as a function of the risk level $\alpha$ at horizon $H = 9$. Figure~\ref{fig:closed_loop_horizon_sensitivity}: bounds as a function of the planning horizon at $\alpha = 0.1$. The dangerous sequence remains clearly separated from the policy and safe sequence across all conditions. Bounds are computed with probability of error $\delta = 0.05$.}
	\label{fig:closed_loop_sensitivity}
\end{figure}

\begin{figure}[t]
	\centering
	\subfloat[]{\includegraphics[width=0.38\textwidth]{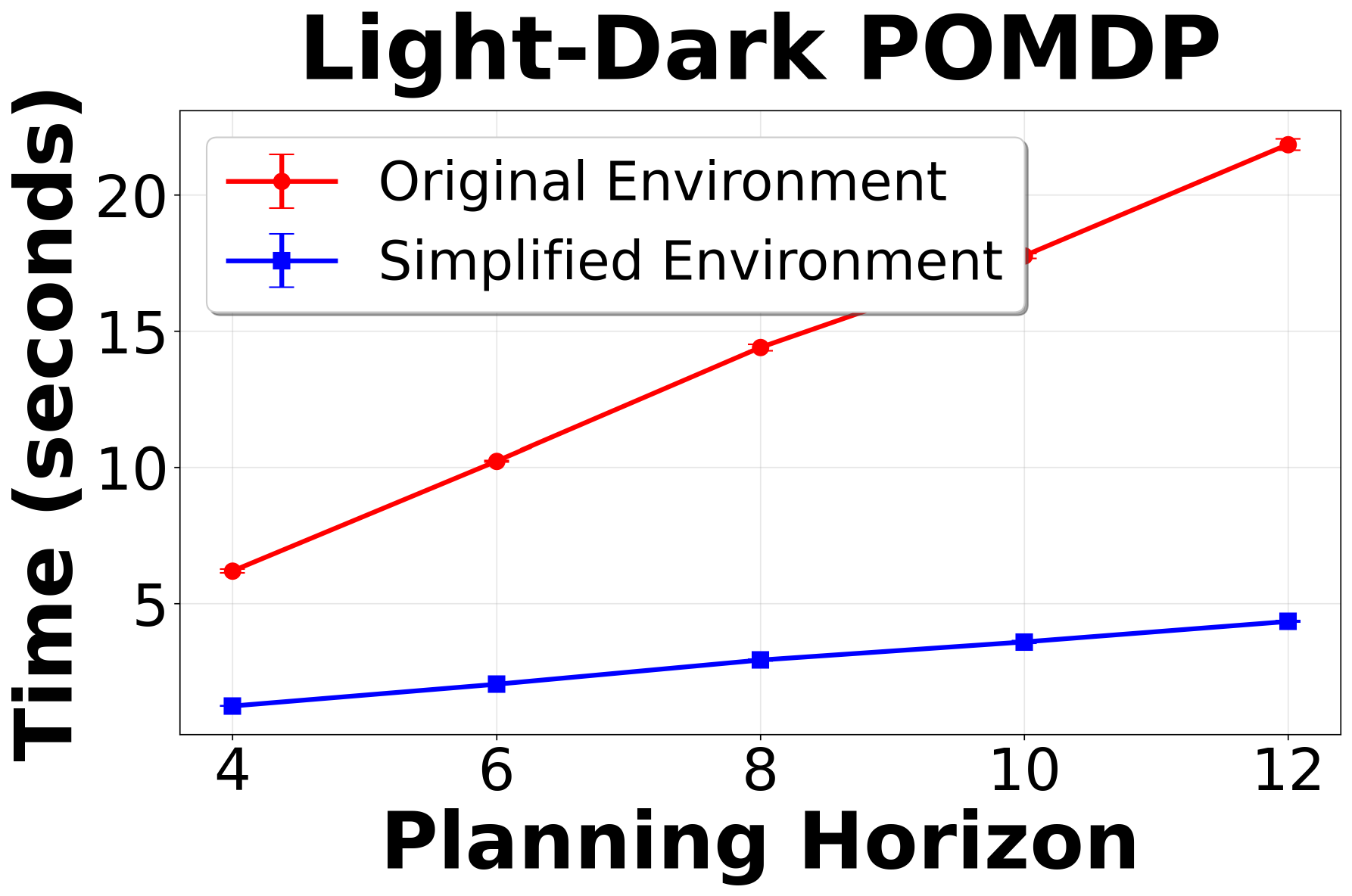}\label{fig:closed_loop_horizon_time}}
	\subfloat[]{\includegraphics[width=0.38\textwidth]{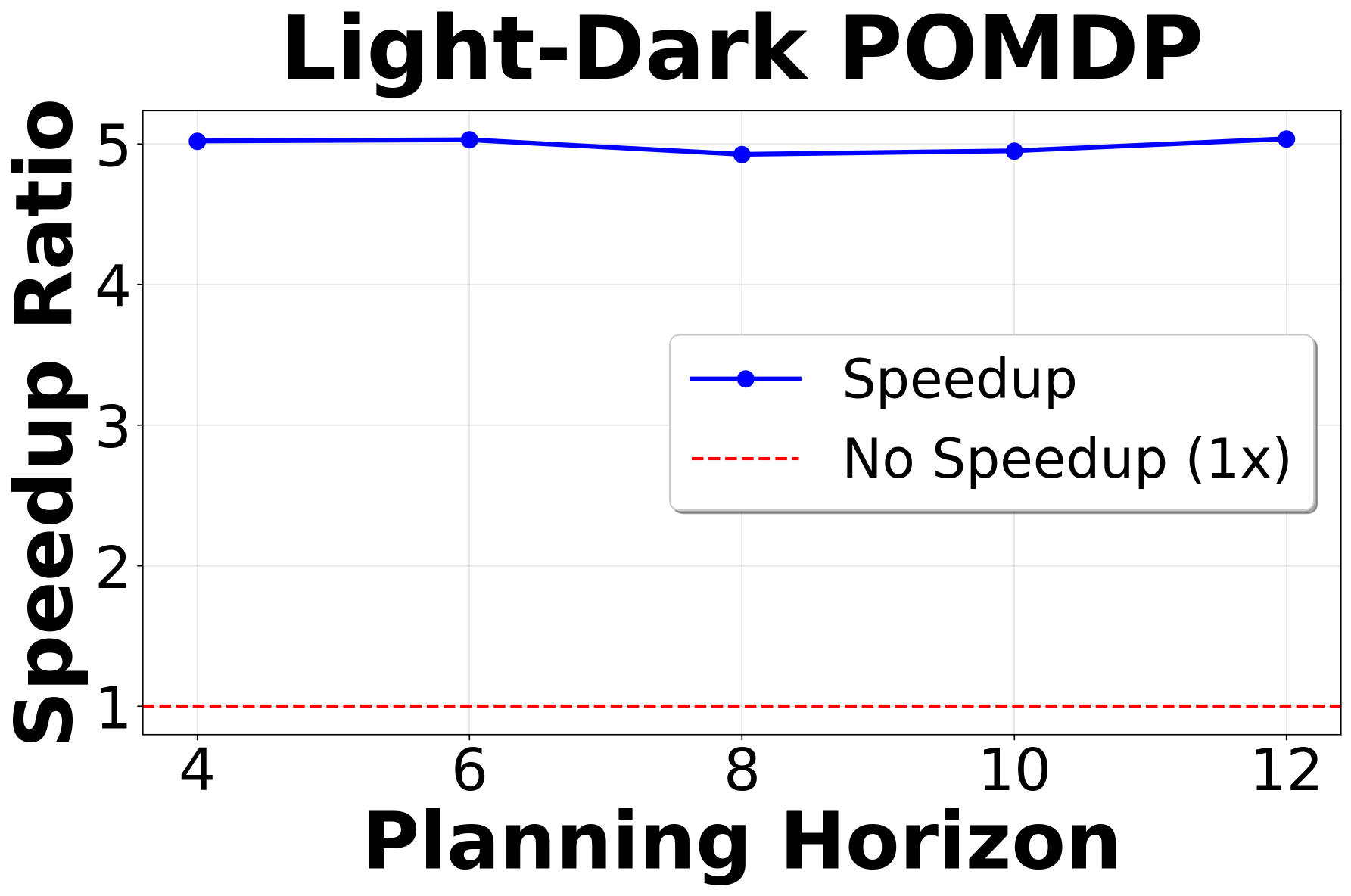}\label{fig:closed_loop_horizon_speedup}} \\[1em]
	\subfloat[]{\includegraphics[width=0.38\textwidth]{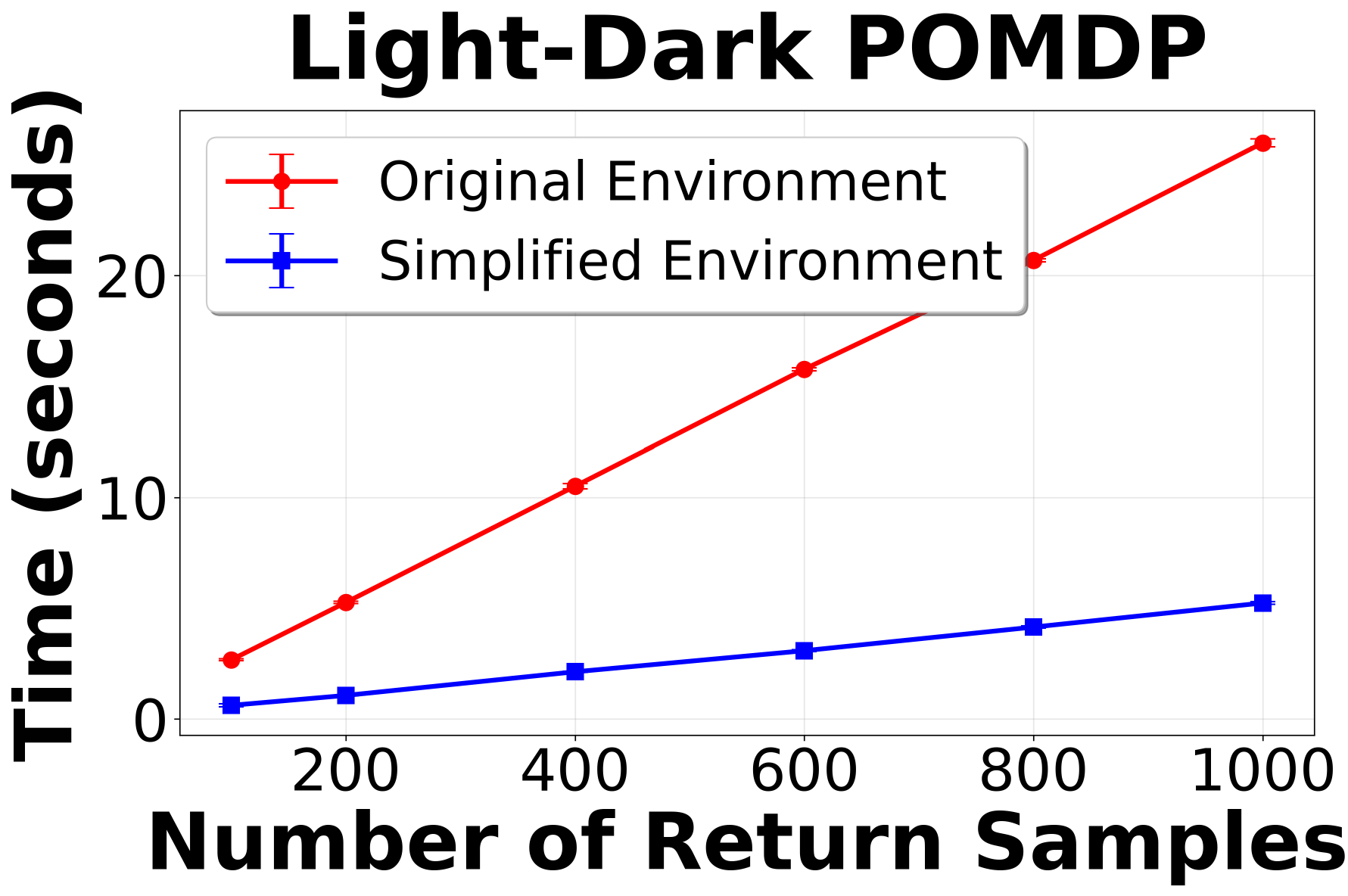}\label{fig:closed_loop_nsamples_time}}
	\subfloat[]{\includegraphics[width=0.38\textwidth]{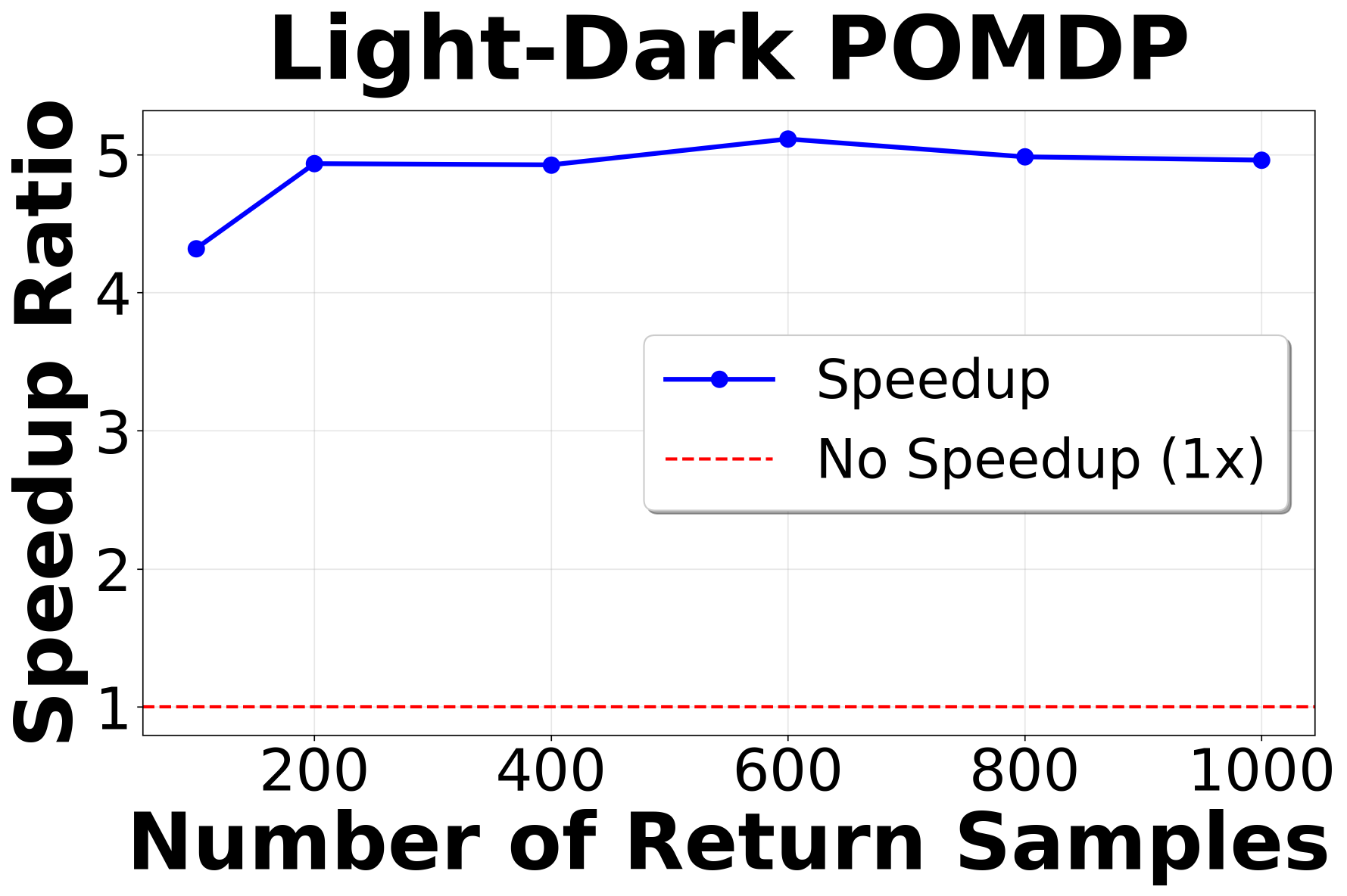}\label{fig:closed_loop_nsamples_speedup}}
	\caption{Computational speedup from evaluating CVaR bounds under the simplified observation model relative to the original model for the BetaZero policy in the Light-Dark POMDP ($\alpha = 0.1$). Row~1: computation time and speedup ratio as a function of the planning horizon. Row~2: computation time and speedup ratio as a function of the number of return samples. The simplified model achieves a consistent ${\approx}5\times$ speedup across all conditions. Error bars indicate $95\%$ confidence intervals.}
	\label{fig:closed_loop_speedup}
\end{figure}

\section{Conclusions}
In this work we introduced a belief-simplification framework for risk-averse policy evaluation in POMDPs under a static CVaR objective, with guaranteed bounds on the resulting action-value function estimators, enabling accelerated online policy evaluation. Our approach establishes bounds on the true action-value function by leveraging an approximate action-value function computed under a simplified belief-transition model. We further derived specialized bounds for the case of a simplified observation model and showed how these bounds can be employed to accelerate online planning. To support these results, we developed the requisite mathematical foundations for bounding the CVaR of a random variable $X$ through an auxiliary random variable $Y$, under assumptions relating their respective cumulative and density functions. These theoretical results are independent of the POMDP-planning setting and provide a general analytical framework that can facilitate future research on risk-sensitive decision making.

Our empirical evaluation, focused on observation model simplification, demonstrates the practical utility of the framework in both open-loop and closed-loop settings. In the open-loop case, the bounds correctly distinguish action sequences of differing risk levels across multiple environments and risk parameters. In the closed-loop case, using a BetaZero neural-network policy, the bounds successfully rank policies by risk level with non-overlapping bound intervals, while achieving a consistent ${\approx}5\times$ computational speedup over evaluation under the original model. The theoretical framework accommodates general belief-transition model simplifications, including state-transition models, whose empirical investigation we leave for future work.

\printbibliography

\newpage
\appendix

\section{CVaR Bounds Proofs}
\begin{theorem}\label{prf:cvar_bound_v2}
	Let $X$ and $Y$ be random variables and $\epsilon \in [0, 1]$. 
	\begin{enumerate}
		\item \textbf{Upper Bound:} assume that $P(X\leq b_X)=1, P(Y\leq b_Y)=1$ and $\forall z\in \mathbb{R},F_Y(z)-F_X(z)\leq \epsilon$, then 
		\begin{enumerate}
			\item If $\alpha > \epsilon$, \begin{equation}
				\CVaR_\alpha(X) \leq \frac{\epsilon}{\alpha}\max(b_X,b_Y) + (1 - \frac{\epsilon}{\alpha})\CVaR_{\alpha-\epsilon}(Y).
			\end{equation}
			
			\item If $\alpha \leq \epsilon$, then $\CVaR_\alpha(X)\leq \max(b_X,b_Y)$.
		\end{enumerate}
		
		\item \textbf{Lower Bound:} If $P(X\geq a_X)=1, P(Y\geq a_Y)=1$ and $\forall z\in \mathbb{R},F_X(z)-F_Y(z)\leq \epsilon$, then
		\begin{enumerate}
			\item If $\alpha+\epsilon \leq 1$, \begin{equation}
				\CVaR_{\alpha}(X)\geq (1+\frac{\epsilon}{\alpha})\CVaR_{\alpha+\epsilon}(Y)-\frac{\epsilon}{\alpha} \CVaR_{\epsilon}(Y)
			\end{equation}
			
			\item If $\alpha+\epsilon > 1$ and $a_{min} = \min(a_X, a_Y)$, \begin{equation}
				\CVaR_\alpha(X) \geq \frac{1}{\alpha}\big(\mathbb{E}[Y] - \epsilon \CVaR_\epsilon(Y) + (\alpha+\epsilon-1)a_{min}\big)
			\end{equation}
		\end{enumerate}
	\end{enumerate}
\end{theorem}
\begin{proof}
	The strategy of the proof is to construct two distributions derived from $F_Y$, denoted $F_Y^L$ and $F_Y^U$, such that $F_X$ is stochastically bounded between them; that is, $F_Y^L \leq F_X \leq F_Y^U$. Consequently, since CVaR is a coherent risk measure, it follows that
	\begin{equation}
		\CVaR_\alpha^{F_Y^L} \leq \CVaR_\alpha^{F_X} \leq \CVaR_\alpha^{F_Y^U},
	\end{equation}
	where $\CVaR_\alpha^{F_Y^L}$, $\CVaR_\alpha^{F_X}$, and $\CVaR_\alpha^{F_Y^U}$ denote the CVaR at level $\alpha$ corresponding to the distributions $F_Y^L$, $F_X$, and $F_Y^U$, respectively.
	
	Let $a_{\min} = \min(a_X, a_Y)$ and $b_{\min} = \min(b_X, b_Y)$. Define the interval $[b_{\min}, b_X)$ to be $\emptyset$ if $b_{\min} = b_X$, and equal to $[b_Y, b_X)$ otherwise. Analogously, define $[a_{\min}, a_X)$ in the same manner. We then define upper and lower bounds for $F_Y$ as follows (see Figure \ref{fig:CDF_bounds}): \begin{equation}
		F_Y^U(y)=\begin{cases}
			0 & y< \max(a_X,q_{1-\epsilon}^Y) \\
			\min(F_Y(y)-\epsilon, 1-\epsilon) & y\in [\max(a_X,q_{1-\epsilon}^Y), b_{min}) \\
			1-\epsilon & y\in [b_{min}, b_X) \\
			1 & y\geq \max(b_X, b_Y),
		\end{cases}
	\end{equation}
	\begin{equation}
		F_Y^L(y)=\begin{cases}
			0 & y<a_{min}  \\
			\epsilon & y\in [a_{min}, a_Y) \\
			\min(F_Y(y)+\epsilon, 1) & y\in [a_Y, \min(q_{\epsilon}^Y, b_X)) \\
			1 & y\geq \min(q_\epsilon^Y, b_X).
		\end{cases}
	\end{equation}

	\begin{figure}[htbp]
		\centering
		\includegraphics[width=0.45\textwidth]{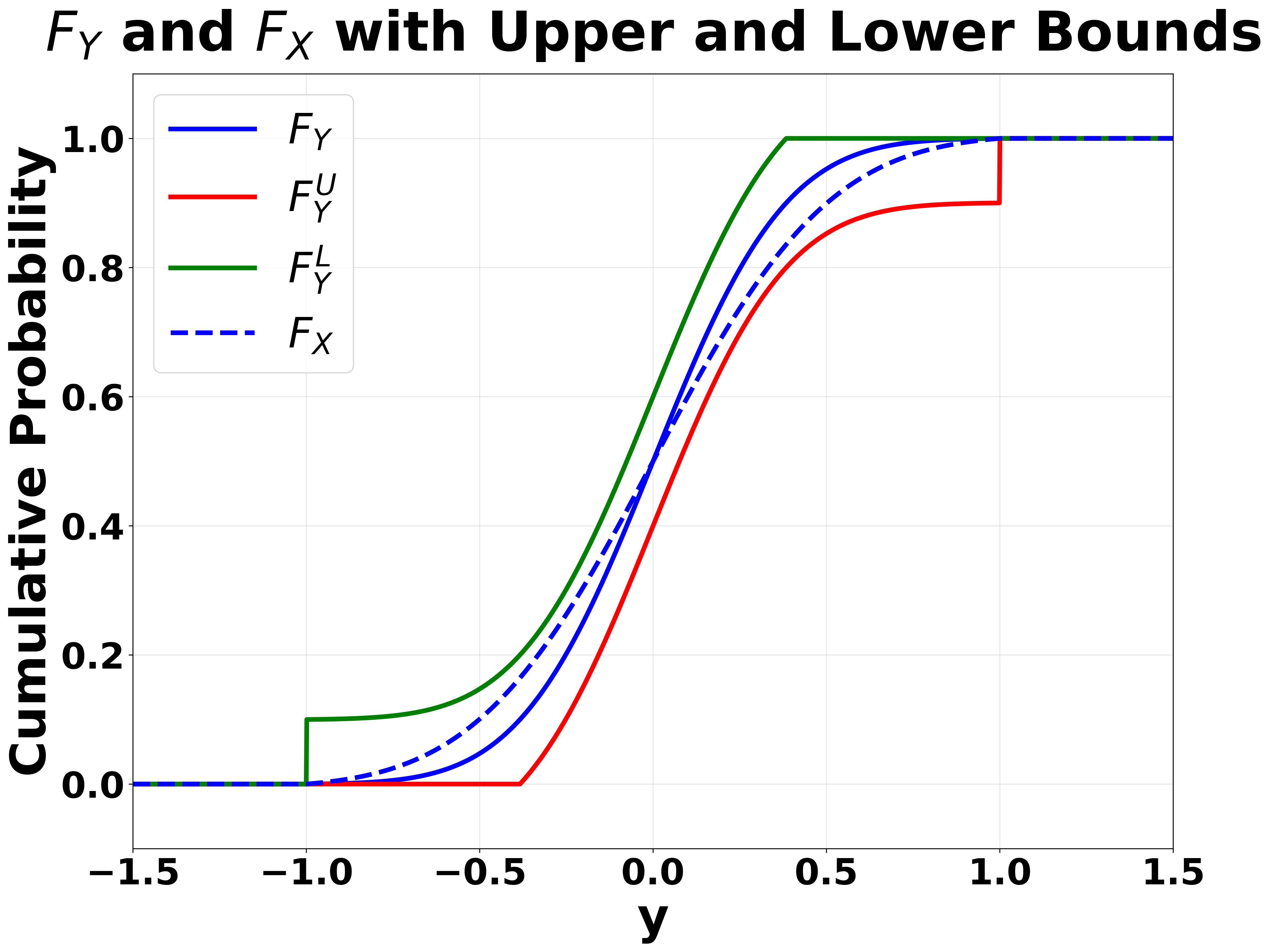}
		\caption{Illustration of the upper and lower bounds $F_Y^U$ and $F_Y^L$ on the cumulative distribution function $F_X$.}
		\label{fig:CDF_bounds}
	\end{figure}
	As a first step, we verify that $F_Y^L$ and $F_Y^U$ are valid CDFs and that they satisfy $F_Y^L \leq F_X \leq F_Y^U$. To establish that $F$ is a CDF, it suffices to verify the following properties:
	\begin{enumerate}
		\item $F$ is non-decreasing;
		\item $F : \mathbb{R} \to [0,1]$ with $\lim_{x \to \infty} F(x) = 1$ and $\lim_{x \to -\infty} F(x) = 0$;
		\item $F$ is right-continuous.
	\end{enumerate}
	\textbf{Proof that $F_Y^L$ is a CDF:} \begin{enumerate}
		\item On the interval $[a_Y, \min(q_{1-\epsilon}^Y, b_X))$, the function $F_Y^L$ is monotone increasing, since $F_Y$ is monotone increasing by virtue of being a CDF. Outside this interval, $F_Y^L$ is constant and consequently preserves monotonicity.
		
		\item By definition, $\lim_{y\rightarrow \infty}F_Y^L(y)=1$ and $\lim_{y\rightarrow -\infty}F_Y^L(y)=0$. We need to show that $F_Y^L$ is bounded between 0 and 1. By its definition, $F_Y^L$ is bounded between 0 and 1.
		
		\item Within the interval $[a_Y, \min(q_{\epsilon}^Y, b_X))$, $F_Y^L$ is right-continuous since $F_Y$ is right-continuous as a CDF. Outside this interval, $F_Y^L$ is constant and hence also right-continuous.
	\end{enumerate}
	\textbf{Proof that $F_Y^U$ is a CDF:} \begin{enumerate}
		\item On the interval $[\max(a_X,q_{1-\epsilon}^Y), b_{min})$, the function $F_Y^U$ is monotone increasing, since $F_Y$ is monotonically increasing by virtue of being a CDF. Outside this interval, $F_Y^U$ is constant and consequently preserves monotonicity.
		\item By definition, $\lim_{y\rightarrow \infty}F_Y^U(y)=1$ and $\lim_{y\rightarrow -\infty}F_Y^U(y)=0$. By its definition, $F_Y^U$ is bounded between 0 and 1.
		\item Within the interval $[\max(a_X,q_{1-\epsilon}^Y), b_{min})$, $F_Y^U$ is right-continuous since $F_Y$ is right-continuous as a CDF. Outside this interval, $F_Y^U$ is constant and hence also right-continuous.
	\end{enumerate}
	\textbf{Proof that $F_Y^L \leq F_X \leq F_Y^U$:} To establish that $F_Y^L \leq F_X \leq F_Y^U$, we need to show that for all $z \in \mathbb{R}$, $F_Y^L(z) \geq F_X(z) \geq F_Y^U(z)$. Let $z\in \mathbb{R}$. \begin{itemize}
		\item If $z<a_{min}$ then $F_Y^L(z)=0 = F_X(z)$.
		\item If $z\in [a_{min}, a_Y)$ then $F_X(y)=F_X(y)-F_Y(y)+F_Y(y)\leq |F_X(y)-F_Y(y)| + F_Y(y)\leq \epsilon + F_Y(y)=\epsilon=F_Y^L(y)$.
		\item If $z\in [a_Y, \min(q_\epsilon^Y, b_X))$ we assume that $F_Y(z)+\epsilon\leq 1$ because otherwise $F_Y^L(z)=1$ and the inequality holds. $F_Y^L(z)=F_Y(z) + \epsilon = F_Y(z)-F_X(z)+F_X(z)+\epsilon \geq F_X(z) - |F_Y(z)-F_X(z)| + \epsilon\geq F_X(z)$.
		\item If $z\geq \min(q_\epsilon^Y, b_X)$ then $F_Y^L(z)=1\geq F_X(z)$.
	\end{itemize} 
	and therefore $F_Y^L\leq F_X$. Note that the last proof holds when $b_X=\infty$, making $F_Y^L$ a valid CDF when the support of $X$ is not bounded from above.
	\begin{itemize}
		\item If $z<\max(a_X,q_{1-\epsilon}^Y)$ then $F_Y^U(z)=0\leq F_X(z)$.
		\item If $z\in [\max(a_X,q_{1-\epsilon}^Y), b_{min})$ then $F_Y^U(z)\leq F_Y(z)-\epsilon=F_Y(z)-F_X(z)+F_X(z)-\epsilon \leq |F_Y(z)-F_X(z)|+F_X(z)-\epsilon \leq \epsilon+F_X(z)-\epsilon= F_X(z)$.
		\item If $z \in [b_{min}, b_X)$ then $F_Y^U(z)=1-\epsilon=F_Y(z)-\epsilon=F_Y(z)-\epsilon+F_X(z)-F_X(z)\leq F_X(z)-\epsilon+|F_Y(z)-F_X(z)|\leq F_X(z)$
	\end{itemize}
	and therefore $F_X\leq F_Y^U$. Note that the last proof holds when $a_X=-\infty$, making $F_Y^U$ a valid CDF when the support of $X$ is not bounded from below.
	
	Next, we derive bounds for the CVaR associated with $F_Y^U$ and $F_Y^L$.
	
	\textbf{Upper bound for $\CVaR_\alpha^{F_Y^U}$:} From \cite{acerbi2002coherence}, CVaR is equal to an integral of the VaR
	\begin{equation}
		\begin{aligned}
			&\CVaR_\alpha^{F_Y^U}=\frac{1}{\alpha}\int_{1-\alpha}^1 \inf \{y\in \mathbb{R}:F_Y^U(y)\geq \tau\} d\tau 
			=\frac{1}{\alpha}\int_{1-\alpha+\epsilon}^{1+\epsilon} \inf \{y\in \mathbb{R}:F_Y^U(y)\geq \tau -\epsilon\} d\tau
		\end{aligned}
	\end{equation}  
	If $\alpha \leq \epsilon$, then $1 - \alpha + \epsilon \geq 1$, rendering the bound in the preceding equation trivial, as it attains the maximum value of the support of both $X$ and $Y$.
	\begin{equation}\label{eq:helper1}
		\begin{aligned}
			&\frac{1}{\alpha}\int_{1-\alpha+\epsilon}^{1+\epsilon} \inf \{y\in \mathbb{R}:F_Y^U(y)\geq \tau -\epsilon\} 
			\leq \frac{1}{\alpha}\int_{1-\alpha+\epsilon}^{1+\epsilon} \max(b_X, b_Y)d\tau = \max(b_X, b_Y).
		\end{aligned}
	\end{equation}
	That is, $\CVaR_\alpha(X)\leq \max(b_X, b_Y)$.
	
	If $\alpha > \epsilon$, then $1 - \alpha + \epsilon < 1$, and the integral may be decomposed into one term that is trivially bounded by the maximum of the support and another term that can be computed explicitly
	\begin{equation}
		\begin{aligned}
			\CVaR_\alpha^{F_Y^U}&=\frac{1}{\alpha}\int_{1-\alpha+\epsilon}^{1+\epsilon} \inf \{y\in \mathbb{R}:F_Y^U(y)\geq \tau -\epsilon\} d\tau 
			=\frac{1}{\alpha}[\underbrace{\int_{1}^{1+\epsilon} \inf \{y\in \mathbb{R}:F_Y^U(y)\geq \tau -\epsilon\} d\tau}_{\triangleq A_1} \\
			&+\underbrace{\int_{1-\alpha+\epsilon}^{1} \inf \{y\in \mathbb{R}:F_Y^U(y)\geq \tau -\epsilon\} d\tau}_{\triangleq A_2}]
		\end{aligned}
	\end{equation}
	The term $A_1$ is bounded, in a manner analogous to \eqref{eq:helper1}, by $\epsilon \max(b_X, b_Y)$, which is essentially the tightest bound attainable given the definition of $F_Y^U$. The term $A_2$ can be expressed in terms of the CVaR of $Y$, evaluated at a shifted confidence level. Specifically, the variable $\tau - \epsilon$ in $A_2$ lies within the interval $[1 - \alpha, 1 - \epsilon]$. Over this range, for all $y \in [\max(a_X, q_{1-\epsilon}^Y), b_{\min}]$, the inequality $F_Y^U(y) \leq F_Y(y) - \epsilon$ holds. This follows because if $q_{1-\epsilon}^Y \geq a_X$, then by definition $F_Y^U(y) = \min(F_Y(y) - \epsilon, 1-\epsilon)$, and if $a_X > q_{1-\epsilon}^Y$, then $F_Y^U(y)=0 \leq F_Y(y) - \epsilon$ for $y\in [q_{1-\epsilon}^Y, a_X]$ and again $F_Y^U(y) = \min(F_Y(y) - \epsilon, 1-\epsilon)$ for $y\in [a_X, b_{min}]$.
	\begin{equation}
		\begin{aligned}
			&A_2 \leq \int_{1-\alpha+\epsilon}^1 \inf \{y\in \mathbb{R}:F_Y(y) - \epsilon \geq \tau -\epsilon\}d\tau 
			=\int_{1-(\alpha-\epsilon)}^1 \inf \{y\in \mathbb{R}:F_Y(y) \geq \tau\}d\tau 
			=(\alpha-\epsilon)\CVaR_{\alpha-\epsilon}(Y).
		\end{aligned}
	\end{equation}
	Note that the confidence level is valid, as $\alpha, \epsilon \in (0,1)$ and $\alpha > \epsilon$ imply $\alpha-\epsilon \in (0, 1)$. By combining the previous two expressions, we obtain a single bound:
	\begin{equation}
		\CVaR_\alpha^{F_Y^U} \leq \epsilon\frac{\max(b_X,b_Y)}{\alpha} + (1 - \frac{\epsilon}{\alpha})\CVaR_{\alpha-\epsilon}(Y).
	\end{equation}
	In the case where $a_X = -\infty$, we have $\max(a_X, q_{1-\epsilon}^Y) = q_{1-\epsilon}^Y$, and the definition of $F_Y^U$ no longer involves $a_X$. Therefore, the inequality remains valid in the case of $a_X = -\infty$.
	
	\textbf{Lower bound for $\CVaR_\alpha^{F_Y^L}$:} From \cite{acerbi2002coherence}, CVaR is equal to an integral of the VaR
	\begin{equation}
		\begin{aligned}
			&\CVaR_\alpha^{F_Y^L}=\frac{1}{\alpha}\int_{1-\alpha}^1 \inf \{y\in \mathbb{R}:F_Y^L(y)\geq \tau\} d\tau 
			=\frac{1}{\alpha}\int_{1-(\epsilon+\alpha)}^{1-\epsilon} \inf \{y\in \mathbb{R}:F_Y^L(y)\geq \tau +\epsilon\} d\tau.
		\end{aligned}
	\end{equation}
	It holds that $F_Y^L(y)\leq F_Y(y)+\epsilon$ and therefore
	\begin{equation}
		\begin{aligned}
			&\CVaR_\alpha^{F_Y^L} \geq \frac{1}{\alpha}\int_{1-(\epsilon+\alpha)}^{1-\epsilon} \inf \{y\in \mathbb{R}:F_Y(y)+\epsilon \geq \tau +\epsilon\} d\tau 
			=\frac{1}{\alpha}\int_{1-(\epsilon+\alpha)}^{1-\epsilon} \inf \{y\in \mathbb{R}:F_Y(y)\geq \tau\} d\tau \\
			&=\frac{1}{\alpha}[\int_{1-(\epsilon+\alpha)}^{1} \inf \{y\in \mathbb{R}:F_Y(y)\geq \tau\} d\tau 
			-\int_{1-\epsilon}^{1} \inf \{y\in \mathbb{R}:F_Y(y)\geq \tau\} d\tau] \\
			&=\frac{1}{\alpha}[(\alpha+\epsilon)\CVaR_{\alpha+\epsilon}(Y)-\epsilon \CVaR_{\epsilon}(Y)].
		\end{aligned}
	\end{equation}
	In the case where $\epsilon+\alpha > 1$,
	\begin{equation}
		\begin{aligned}
			\CVaR_\alpha^{F_Y^L}&=\frac{1}{\alpha}\int_{1-(\epsilon+\alpha)}^{1-\epsilon} \inf \{y\in \mathbb{R}:F_Y^L(y)\geq \tau +\epsilon\} d\tau \\
			&=\frac{1}{\alpha}[\underbrace{\int_{0}^{1-\epsilon} \inf \{y\in \mathbb{R}:F_Y^L(y)\geq \tau +\epsilon\} d\tau}_{\triangleq A_3} 
			+\underbrace{\int_{1-(\epsilon+\alpha)}^{0} \inf \{y\in \mathbb{R}:F_Y^L(y)\geq \tau +\epsilon\} d\tau}_{\triangleq A_4}].
		\end{aligned}
	\end{equation}
	
	\begin{equation}
		\begin{aligned}
			A_3 &\geq \int_{0}^{1-\epsilon} \inf \{y\in \mathbb{R}:F_Y(y)\geq \tau\} d\tau 
			=\int_{0}^{1} \inf \{y\in \mathbb{R}:F_Y(y)\geq \tau\} d\tau 
			-\int_{1-\epsilon}^{1} \inf \{y\in \mathbb{R}:F_Y(y)\geq \tau\} d\tau \\
			&=\mathbb{E}[Y] - \epsilon \CVaR_\epsilon(Y)
		\end{aligned}
	\end{equation}
	In the case of $A_4$, after the change of variable, $\tau$ ranges from $1-\alpha$ to $\epsilon$, and therefore 
	\begin{equation}
		\begin{aligned}
			&A_4=\int_{1-\alpha}^{\epsilon} \inf \{y\in \mathbb{R}:F_Y^L(y)\geq \tau\} d\tau \geq (\alpha+\epsilon-1)a_{min}
		\end{aligned}
	\end{equation}
	By combining the last equations to one bound we get \begin{equation}
		\begin{aligned}
			\CVaR_\alpha(X)&\geq \CVaR_\alpha^{F_Y^L}
			\geq \frac{1}{\alpha}\big(\mathbb{E}[Y] - \epsilon \CVaR_\epsilon(Y) + (\alpha+\epsilon-1)a_{min}\big).
		\end{aligned}
	\end{equation}
\end{proof}

\begin{theorem}\label{prf:cvar_bound_convergence}
	Given the definition of $X,Y,\epsilon$ and $\alpha$ as in Theorem \ref{thm:cvar_bound_v2}, the lower and upper bounds of Theorem \ref{thm:cvar_bound_v2} converge to $\CVaR_\alpha(X)$ as $\epsilon \rightarrow 0$.
\end{theorem}
\begin{proof}
	Let $f$ and $g$ be continuous functions such that $f(0)$ and $g(0)$ are finite. Then
	$\lim_{x \to 0} f(x) g(x) = \big( \lim_{x \to 0} f(x) \big) \big( \lim_{x \to 0} g(x) \big)$.
	This property will be used throughout the proof.
	
	Assume that $F_Y(z) - F_X(z) \leq \epsilon$ for all $z \in \mathbb{R}$. Since we consider the limit as $\epsilon \to 0$, we further assume that $\alpha > \epsilon$ when computing the bound. Noting that CVaR is continuous with respect to the confidence level, it follows that $\lim_{\epsilon \to 0} \CVaR_{\alpha - \epsilon}(Y) = \CVaR_{\alpha}(Y)$.
	\begin{equation}
		\begin{aligned}
			&\lim_{\epsilon\rightarrow 0} \frac{\epsilon}{\alpha}\max(b_X,b_Y) + (1 - \frac{\epsilon}{\alpha})\CVaR_{\alpha-\epsilon}(Y) 
			=\lim_{\epsilon \rightarrow 0} (1 - \frac{\epsilon}{\alpha})\CVaR_{\alpha-\epsilon}(Y) 
			=\lim_{\epsilon \rightarrow 0} 1 - \frac{\epsilon}{\alpha} \lim_{\epsilon\rightarrow 0} \CVaR_{\alpha-\epsilon}(Y)=\CVaR_\alpha(Y)
		\end{aligned}
	\end{equation}
	
	If $F_X(z) - F_Y(z) \leq \epsilon$ for all $z \in \mathbb{R}$, then, since we consider the limit as $\epsilon \to 0$, we assume in the derivation of the bound that $\alpha + \epsilon \leq 1$.
	\begin{equation}
		\begin{aligned}
			&\lim_{\epsilon\rightarrow 0} (1+\frac{\epsilon}{\alpha})\CVaR_{\alpha+\epsilon}(Y)-\frac{\epsilon}{\alpha} \CVaR_{\epsilon}(Y)
			=\lim_{\epsilon\rightarrow 0} 1+\frac{\epsilon}{\alpha} \lim_{\epsilon\rightarrow 0} \CVaR_{\alpha+\epsilon}(Y)
			-\lim_{\epsilon\rightarrow 0} \frac{\epsilon}{\alpha} \lim_{\epsilon \rightarrow 0} \CVaR_{\epsilon}(Y)
			=\CVaR_\alpha(Y).
		\end{aligned}
	\end{equation}
\end{proof}

\begin{theorem}
	\label{prf:ecdf_cvar_bound}
	Let $X$ be a random variable, $\alpha\in (0, 1],\delta \in (0, 0.5), \epsilon=\sqrt{\ln(1/\delta)/(2n)}$. Let $X_1, \dots, X_n \overset{iid}{\sim} F_X$ be random variables that define the ECDF $\hat{F}_X$.
	\begin{enumerate}
		\item \textbf{Upper Bound:} If $P(X\leq b)=1$, then
		\begin{enumerate}
			\item If $\alpha > \epsilon$ then $P(\CVaR_\alpha(X) \leq (1-\frac{\epsilon}{\alpha})C_{\alpha-\epsilon}^{\hat{F}_X} + \frac{\epsilon}{\alpha}b)> 1-\delta$.
			
			\item If $\alpha\leq \epsilon$ then $\CVaR_\alpha(X)\leq b$
		\end{enumerate}
		
		\item \textbf{Lower Bound:} If $P(X\geq a)=1$, then 
		\begin{enumerate}
			\item If $\alpha+\epsilon < 1$, then $P(\CVaR_\alpha(X)\geq (1+\frac{\epsilon}{\alpha})C_{\alpha+\epsilon}^{\hat{F}_X}-\frac{\epsilon}{\alpha}C_{\epsilon}^{\hat{F}_X})> 1-\delta$.
			
			\item If $\alpha+\epsilon \geq 1$, then $P(\CVaR_\alpha(X) \geq \frac{1}{\alpha}[(\alpha+\epsilon - 1)a + E_{\hat{F}_X}[X] - \epsilon C_\epsilon^{\hat{F}_X}]) > 1-\delta$.
		\end{enumerate}
	\end{enumerate}
\end{theorem}
\begin{proof}
	From DKW inequality the following inequalities can be derived \cite{pmlr-v97-thomas19a}
	\begin{equation}
		\Pr \left( \sup_{x \in \mathbb{R}} \big( \hat{F}(x) - F(x) \big) 
		\leq \sqrt{\frac{\ln(1/\delta)}{2n}} \right) \geq 1 - \delta, \Pr \left( \sup_{x \in \mathbb{R}} \big( \hat{F}(x) - F(x) \big) 
		\geq \sqrt{\frac{\ln(1/\delta)}{2n}} \right) \geq 1 - \delta.
	\end{equation}
	Let $\epsilon=\sqrt{\frac{\ln(1/\delta)}{2n}}$. By using Theorem \ref{thm:cvar_bound_v2} we get the following. 
	
	If $\alpha > \epsilon$ then \begin{equation}
		\begin{aligned}
			P(\CVaR_\alpha(X) \leq \frac{\alpha-\epsilon}{\alpha}C_{\alpha-\epsilon}^{\hat{F}_X} + \frac{\epsilon}{\alpha}b) 
			&= \underbrace{P(C_\alpha(X) \leq \frac{\alpha-\epsilon}{\alpha}C_{\alpha-\epsilon}^{\hat{F}_X} + \frac{\epsilon}{\alpha}b\Bigr{|} \sup_{z\in\mathbb{R}}(\hat{F}_X-F_X)\leq \epsilon)}_{=1} 
			\underbrace{P(\sup_{z\in\mathbb{R}}(\hat{F}_X-F_X)\leq \epsilon)}_{>1-\delta} \\
			&+ \underbrace{P(C_\alpha(X) \leq \frac{\alpha-\epsilon}{\alpha}C_{\alpha-\epsilon}^{\hat{F}_X} + \frac{\epsilon}{\alpha}b\Bigr{|} \sup_{z\in\mathbb{R}}(\hat{F}_X-F_X)> \epsilon)}_{\geq 0} 
			\underbrace{P(\sup_{z\in\mathbb{R}}(\hat{F}_X-F_X)> \epsilon)}_{\geq 0} \\
			&> 1-\delta.
		\end{aligned}
	\end{equation}
	Observe that, conditional on the event 
	\[
	\sup_{z \in \mathbb{R}} \big( \hat{F}_X(z) - F_X(z) \big) \leq \epsilon,
	\]
	the bound in Theorem~\ref{thm:cvar_bound_v2} holds deterministically. Consequently, the probability that
	\[
	C_\alpha(X) \leq \frac{\alpha - \epsilon}{\alpha} C_{\alpha - \epsilon}^{\hat{F}_X} + \frac{\epsilon}{\alpha} b
	\]
	holds, given 
	\[
	\sup_{z \in \mathbb{R}} \big( \hat{F}_X(z) - F_X(z) \big) \leq \epsilon,
	\]
	is equal to one. The same observation is necessary through the rest of the proof in a similar manner.
	If $\alpha+\epsilon < 1$, then \begin{equation}
		\begin{aligned}
			P(C_\alpha(X)\geq (1+\frac{\epsilon}{\alpha})C_{\alpha+\epsilon}^{\hat{F}_X}-\frac{\epsilon}{\alpha}C_{\epsilon}^{\hat{F}_X}) 
			&= \underbrace{P(C_\alpha(X)\geq (1+\frac{\epsilon}{\alpha})C_{\alpha+\epsilon}^{\hat{F}_X}-\frac{\epsilon}{\alpha}C_{\epsilon}^{\hat{F}_X}\Bigr{|} \sup_{z\in\mathbb{R}}(F_X-\hat{F}_X)\leq \epsilon)}_{=1} \underbrace{P(\sup_{z\in\mathbb{R}}(F_X-\hat{F}_X)\leq \epsilon)}_{>1-\delta} \\
			&+\underbrace{P(C_\alpha(X)\geq (1+\frac{\epsilon}{\alpha})C_{\alpha+\epsilon}^{\hat{F}_X}-\frac{\epsilon}{\alpha}C_{\epsilon}^{\hat{F}_X}\Bigr{|} \sup_{z\in\mathbb{R}}(F_X-\hat{F}_X) > \epsilon)}_{\geq 0} 
			\underbrace{P(\sup_{z\in\mathbb{R}}(F_X-\hat{F}_X) > \epsilon)}_{\geq 0} \\
			&> 1-\delta
		\end{aligned}
	\end{equation}
	If $\alpha+\epsilon \geq 1$, then \begin{equation}
		\begin{aligned}
			&P(C_\alpha(X) \geq \frac{1}{\alpha}[(\alpha+\epsilon - 1)a + E_{\hat{F}_X}[X] - \epsilon C_\epsilon^{\hat{F}_X}]) \\
			&=\underbrace{P(C_\alpha(X) \geq \frac{1}{\alpha}[(\alpha+\epsilon - 1)a + E_{\hat{F}_X}[X] - \epsilon C_\epsilon^{\hat{F}_X}]\Bigr{|} \sup_{z\in\mathbb{R}}(F_X-\hat{F}_X)\leq \epsilon)}_{=1} 
			\underbrace{P(\sup_{z\in\mathbb{R}}(F_X-\hat{F}_X) \leq \epsilon)}_{>1-\delta} \\
			&+\underbrace{P(C_\alpha(X) \geq \frac{1}{\alpha}[(\alpha+\epsilon - 1)a + E_{\hat{F}_X}[X] - \epsilon C_\epsilon^{\hat{F}_X}]\Bigr{|} \sup_{z\in\mathbb{R}}(F_X-\hat{F}_X) > \epsilon)}_{\geq 0} 
			\underbrace{P(\sup_{z\in\mathbb{R}}(F_X-\hat{F}_X) > \epsilon)}_{\geq 0} \\
			&> 1-\delta
		\end{aligned}
	\end{equation}
\end{proof}

\begin{theorem}\label{prf:tight_cvar_lower_bound}
	(Tighter CVaR Lower Bound) Let $\alpha \in (0,1)$, $X$ and $Y$ be random variables. Define a random variable $Y^L$ such that $F_{Y^L}(y) \triangleq \min(1, F_Y(y) + g(y))$ for $g:\mathbb{R}\rightarrow [0, \infty)$. Assume 
	$\lim_{x \rightarrow -\infty}g(x)=0$, $g$ is continuous from the right and monotonic increasing.
	If $\forall x\in \mathbb{R}, F_X(x)\leq F_Y(x)+g(x)$, then $F_{Y^L}$ is a CDF and $\CVaR_\alpha(Y^L)\leq \CVaR_\alpha(X).$
\end{theorem}
\begin{proof}
	In order to prove that $F_{Y^L}$ is a CDF we need to prove that $F_{Y^L}$ is:
	\begin{enumerate}
		\item Monotonic increasing
		\item $F:\mathbb{R}\rightarrow [0,1]$, $\lim_{x\rightarrow \infty}F_{Y^L}(x)=1$, $\lim_{x\rightarrow -\infty}F_{Y^L}(x)=0$
		\item Continuous from the right
	\end{enumerate}
	
	\textbf{Monotonic increasing:}
	Note that for every $f_i:\mathbb{R}\rightarrow \mathbb{R}, i=1,2$ that are monotonic increasing, $f_1(f_2(x)))$ is also monotonic increasing in x. Denote $f(x):=\min(x,1)$ and $f_2(x):=F_Y(x)+g(x)$. $F_Y(x)$ is a CDF and therefore monotonic increasing, so $f_2(x)$ is monotonic increasing as a sum of monotonic increasing functions. $f_1$ is also monotonic increasing, and $F_{Y^L}(x)=f_1(f_2(x))$. Therefore $F_{Y^L}(x)$ is monotonic increasing. \\
	
	\textbf{Limits:} \begin{equation}
		\begin{aligned}
			1&\geq\lim_{x\rightarrow \infty} F_{Y^L}(x)=\lim_{x\rightarrow \infty} \min(1,F_Y(x)+g(x)) \geq \lim_{x\rightarrow \infty}\min(1,F_Y(x))=\lim_{x\rightarrow \infty} F_Y(x)=1
		\end{aligned}
	\end{equation}
	and therefore $\lim_{x\rightarrow \infty} F_{Y^L}(x)=1$.
	$$\begin{aligned}
		0&\leq \lim_{x\rightarrow -\infty} F_{Y^L}(x)=\lim_{x\rightarrow -\infty} \min(1,F_Y(x)+g(x)) 
		\leq \lim_{x\rightarrow -\infty} F_Y(x)+g(x) 
		=\lim_{x\rightarrow -\infty} F_Y(x)+\lim_{x\rightarrow -\infty} g(x) 
		=0	
	\end{aligned}$$
	and therefore $\lim_{x\rightarrow -\infty} F_{Y^L}(x)=0$. By definition $\forall x\in \mathbb{R},F_{Y^L}(x)\leq 1$, and $\forall x\in \mathbb{R},F_{Y^L}(x)\geq 0$ because both g and $F_{Y^L}$ are non negative functions. \\
	
	\textbf{Continuity from the right:}
	$F_Y$ is continuous from the right because it is a CDF, and $g$ is continuous from the right by assumption. Their sum $F_Y + g$ is therefore continuous from the right, and since the constant function $1$ is continuous, $F_{Y^L} = \min(1, F_Y + g)$ is continuous from the right as the minimum of two right-continuous functions. \\
	Thus, $F_{Y^L}$ is a CDF. \\
	
	\textbf{Bound proof:}
	If $Y^L\leq X$, then $\CVaR_\alpha(Y^L)\leq \CVaR_\alpha(X)$ because CVaR is a coherent risk measure. Note that 
	if $F_Y(x) + g(x)<1$ then
	$$F_X(x)\leq F_Y(x) + g(x)=F_{Y^L}(x),$$
	and if $F_Y(x) + g(x)\geq 1$, $1=F_{Y^L}(x)\geq F_X(x)$. Therefore $Y^L\leq X$.
\end{proof}

\begin{theorem}\label{prf:lower_cvar_bound_using_density_bound}
	Let $\alpha \in (0,1)$, $X$ and $Y$ be random variables. Define $h:\mathbb{R}\rightarrow [0,\infty)$ to be a continuous function, $g(z):=\int_{-\infty}^z h(x)dx$ and $Y^L$ to be a random variable such that $F_{Y^L}(y):=\min(1, F_{Y}(y) + g(y))$. If $\lim_{z\rightarrow -\infty} g(z)=0$ and $\forall x\in \mathbb{R},f_x(x)\leq f_y(x) + h(x)$, then $F_{Y^L}$ is a CDF and $\CVaR_\alpha(Y^L)\leq \CVaR_\alpha(X)$.
\end{theorem}
\begin{proof}
	We will show the g satisfies the properties of Theorem \ref{thm:tight_cvar_lower_bound}, and therefore this theorem holds. We need to prove that \begin{enumerate}
		\item $\lim_{z\rightarrow -\infty} g(z)=0$
		\item g is continuous from the right.
		\item g is monotonic increasing.
		\item $F_X(y)\leq F_Y(y) + g(y)$
	\end{enumerate}
	It is given in the theorem's assumptions that $\lim_{z\rightarrow -\infty} g(z)=0$, so (1) holds. h is non negative and therefore $g$ is monotonic increasing, so (3) holds. $g$ is continuous if its derivative exists for all $z\in \mathbb{R}$. Let $z\in \mathbb{R}$ and $a<z$.
	\begin{equation}
		\begin{aligned}
			\frac{d}{dz}g(z) &= \frac{d}{dz}\int_{-\infty}^z h(x)dx = \frac{d}{dz}[\int_{-\infty}^a h(x)dx + \int_{a}^z h(x)dx] 
			=\frac{d}{dz}\int_{a}^z h(x)dx = h(z)
		\end{aligned}
	\end{equation}
	where the third equality holds because $\int_{-\infty}^a h(x)dx = g(a)$ is a constant that does not depend on z. The last equality holds from the fundamental theorem of calculus because $h$ is continuous. Finally, (4) holds because \begin{equation}
		\begin{aligned}
			F_X(y) &\triangleq \int_{-\infty}^y f_x(x)dx\leq \int_{-\infty}^y f_y(x) + h(x)dx 
			\triangleq F_Y(y) + g(y).
		\end{aligned}
	\end{equation}
\end{proof}

\begin{corollary}\label{prf:x_concentration_bound_asymptotic_convergence}
	Let $X$ be a random variable, $\alpha\in (0, 1],\delta \in (0, 1), \epsilon=\sqrt{\ln(1/\delta)/(2n)}, a\in \mathbb{R}, b\in \mathbb{R}$. Let $X_1, \dots, X_n \overset{iid}{\sim} F_X$ be random variables that define the ECDF $\hat{F}_X$. Denote by $U(n)$ and $L(n)$ the upper and lower bounds respectively from Theorem \ref{thm:ecdf_cvar_bound}, where $n$ is the number of samples, then
	\begin{enumerate}
		\item If $P(X\leq b)=1$, then $\lim_{n\rightarrow \infty} U(n)\overset{a.s.}{=}\CVaR_\alpha(X)$.
		\item If $P(X\geq a)=1$, then $\lim_{n\rightarrow \infty} L(n)\overset{a.s.}{=}\CVaR_\alpha(X)$.
	\end{enumerate}
	where a.s.~denotes almost sure convergence.
\end{corollary}
\begin{proof}
	\begin{equation}
		\begin{aligned}
			&\lim_{n \to \infty} U(n)\overset{a.s}{=} \lim_{n \to \infty} (1-\frac{\epsilon}{\alpha})C_{\alpha-\epsilon}^{\hat{F}_X} + \frac{\epsilon}{\alpha}b 
			\overset{a.s}{=}\lim_{n \to \infty} C_{\alpha-\epsilon}^{\hat{F}_X}=\lim_{n \to \infty} C_{\alpha}^{\hat{F}_X} = C_{\alpha}^{F_X}
		\end{aligned}
	\end{equation}
	where the first and second equalities follow from the fact that $\epsilon \to 0$ as $n \to \infty$; the third equality holds by the continuity of CVaR with respect to $\alpha$; and the fourth equality holds by the almost sure convergence of the empirical CVaR of i.i.d. samples to the true CVaR. For the same reason,
	\begin{equation}
		\begin{aligned}
			&\lim_{n\rightarrow \infty} L(n)\overset{a.s}{=} \lim_{n\rightarrow \infty} (1+\frac{\epsilon}{\alpha})C_{\alpha+\epsilon}^{\hat{F}_X}-\frac{\epsilon}{\alpha}C_{\epsilon}^{\hat{F}_X} 
			\overset{a.s}{=} \lim_{n\rightarrow \infty} C_{\alpha+\epsilon}^{\hat{F}_X}
			\overset{a.s}{=} \lim_{n\rightarrow \infty} C_{\alpha}^{\hat{F}_X}
			\overset{a.s}{=} \lim_{n\rightarrow \infty} C_{\alpha}^{F_X}.
		\end{aligned}
	\end{equation}
\end{proof}

\begin{theorem}
	\label{prf:ecdf_y_bound_cvar_x}
	Let $X$ and $Y$ be random variables, $\epsilon \in [0, 1]$, and $\eta = \sqrt{\ln(1/\delta)/(2n)}, \epsilon'=\min(\epsilon+\eta, 1)$. Let $Y_1, \dots, Y_n$ be independent and identically distributed samples from $F_Y$, and denote by $\hat{F}_Y$ the associated empirical cumulative distribution function.
	\begin{enumerate}
		\item \textbf{Upper Bound:} If $\forall z\in \mathbb{R},F_Y(z)-F_X(z)\leq \epsilon$ and $P(X\leq b_X)=1, P(Y\leq b_Y)=1$, then
		
		\begin{enumerate}
			\item If $\alpha > \epsilon'$ then \begin{equation}
				\begin{aligned}
					P\Big(\CVaR_\alpha(X) &\leq \frac{\epsilon'}{\alpha}\max(b_X,b_Y) 
					+ (1 - \frac{\epsilon'}{\alpha})\CVaR_{\alpha-\epsilon'}^{\hat{F}_Y}\Big)>1-\delta.
				\end{aligned}
			\end{equation}
			
			\item If $\alpha \leq \epsilon'$, then $\CVaR_\alpha(X)\leq \max(b_X,b_Y)$
		\end{enumerate}
		
		\item \textbf{Lower Bound:} If $\forall z\in \mathbb{R},F_X(z)-F_Y(z)\leq \epsilon$ and $P(X\geq a_X)=1,P(Y\geq a_Y)=1$, then \begin{enumerate}
			\item If $\alpha+\epsilon' \leq 1$, then \begin{equation}
				\begin{aligned}
					P\Big(\CVaR_{\alpha}(X)&\geq (1+\frac{\epsilon'}{\alpha})\CVaR_{\alpha+\epsilon'}^{\hat{F}_Y} 
					-\frac{\epsilon'}{\alpha} \CVaR_{\epsilon'}^{\hat{F}_Y}\Big)>1-\delta.
				\end{aligned}
			\end{equation}
			
			\item If $\alpha+\epsilon' > 1$, then \begin{equation}
				\begin{aligned}
					P\Big(\CVaR_\alpha(X) &\geq \frac{1}{\alpha}\big(\mathbb{E}_{\hat{F}_Y}[Y] - \epsilon' \CVaR_{\epsilon'}^{\hat{F}_Y}
					+ (\alpha+\epsilon'-1)a_{min}\big) \Big) > 1-\delta.
				\end{aligned}
			\end{equation}
		\end{enumerate}
	\end{enumerate}
\end{theorem}
\begin{proof}
	We begin by establishing that $\epsilon'$ bounds the distributional discrepancy between $\hat{F}_Y$ and $F_X$, under the assumption that $\sup_x \big( \hat{F}_Y(x) - F_Y(x) \big) \leq \eta$. Let $x\in \mathbb{R}$,
	\begin{equation}
		\begin{aligned}
			&\hat{F}_Y(x)-F_X(x)=\hat{F}_Y(x)-F_Y(x) + F_Y(x)-F_X(x) \\
			&\leq |\hat{F}_Y(x)-F_Y(x)| + |F_Y(x)-F_X(x)| \leq \eta + \epsilon=\epsilon'.
		\end{aligned}
	\end{equation}
	Assume that $\alpha > \epsilon'$. Note that, conditional on the event $\sup_{z \in \mathbb{R}} \big( \hat{F}_Y(z) - F_X(z) \big) \leq \epsilon'$,
	the upper bound in Theorem~\ref{thm:cvar_bound_v2} holds deterministically. Consequently, the probability that
	\begin{equation}
		\CVaR_\alpha(X) \leq \frac{\epsilon'}{\alpha}\max(b_X,b_Y) 
		+ (1 - \frac{\epsilon'}{\alpha})\CVaR_{\alpha-\epsilon'}^{\hat{F}_Y}
	\end{equation}
	holds is equal to one. From the law of total probability, 
	\begin{equation}
		\begin{aligned}
			&P\Big(\CVaR_\alpha(X) \leq \frac{\epsilon'}{\alpha}\max(b_X,b_Y) + (1 - \frac{\epsilon'}{\alpha})\CVaR_{\alpha-\epsilon'}^{\hat{F}_Y}\Big) \\
			&=P\Big(\CVaR_\alpha(X) \leq \frac{\epsilon'}{\alpha}\max(b_X,b_Y) + (1 - \frac{\epsilon'}{\alpha})\CVaR_{\alpha-\epsilon'}^{\hat{F}_Y}
			\Big| \sup_{z \in \mathbb{R}} \big( \hat{F}_Y(z) - F_X(z) \big) \leq \epsilon'\Big) 
			P(\sup_{z \in \mathbb{R}} \big( \hat{F}_Y(z) - F_X(z) \big) \leq \epsilon') \\
			&+P\Big(\CVaR_\alpha(X) \leq \frac{\epsilon'}{\alpha}\max(b_X,b_Y) + (1 - \frac{\epsilon'}{\alpha})\CVaR_{\alpha-\epsilon'}^{\hat{F}_Y}
			\Big| \sup_{z \in \mathbb{R}} \big( \hat{F}_Y(z) - F_X(z) \big) > \epsilon'\Big) 
			P(\sup_{z \in \mathbb{R}} \big( \hat{F}_Y(z) - F_X(z) \big) > \epsilon')\\
			&=P(\sup_{z \in \mathbb{R}} \big( \hat{F}_Y(z) - F_X(z) \big) \leq \epsilon') \\
			&+P\Big(\CVaR_\alpha(X) \leq \frac{\epsilon'}{\alpha}\max(b_X,b_Y) + (1 - \frac{\epsilon'}{\alpha})\CVaR_{\alpha-\epsilon'}^{\hat{F}_Y}
			\Big| \sup_{z \in \mathbb{R}} \big( \hat{F}_Y(z) - F_X(z) \big) > \epsilon'\Big) 
			P(\sup_{z \in \mathbb{R}} \big( \hat{F}_Y(z) - F_X(z) \big) > \epsilon') \\
			&\geq P(\sup_{z \in \mathbb{R}} \big( \hat{F}_Y(z) - F_X(z) \big) \leq \epsilon')
		\end{aligned}
	\end{equation}
	From DKW \cite{DKW_inequality} inequality the following inequalities can be derived \cite{pmlr-v97-thomas19a}
	\begin{equation}\label{eq:helper_2}
		\Pr \left( \sup_{x \in \mathbb{R}} \big( \hat{F}(x) - F(x) \big) 
		\leq \sqrt{\frac{\ln(1/\delta)}{2n}} \right) \geq 1 - \delta, \qquad \Pr \left( \sup_{x \in \mathbb{R}} \big( \hat{F}(x) - F(x) \big) 
		\geq \sqrt{\frac{\ln(1/\delta)}{2n}} \right) \geq 1 - \delta.
	\end{equation}
	
	\begin{equation}
		\begin{aligned}
			P(\sup_{z \in \mathbb{R}} \big( \hat{F}_Y(z) - F_X(z) \big) \leq \epsilon') 
			&\geq P(\sup_{z \in \mathbb{R}} \big( \hat{F}_Y(z) - F_Y(z) \big) + \sup_{z \in \mathbb{R}} \big( F_Y(z) - F_X(z) \big)\leq \epsilon') \\
			&\geq P(\sup_{z \in \mathbb{R}} \big( \hat{F}_Y(z) - F_Y(z) \big) + \epsilon \leq \epsilon + \eta) \\
			&=P(\sup_{z \in \mathbb{R}} \big( \hat{F}_Y(z) - F_Y(z) \big)\leq \eta) > 1-\delta.
		\end{aligned}
	\end{equation}
	The first inequality follows from the triangle inequality; the second holds since $\epsilon$ bounds the distributional discrepancy between $X$ and $Y$; and the third follows from \eqref{eq:helper_2}.
	
	As with the preceding equations, all bounds in Theorem~\ref{thm:cvar_bound_v2} hold deterministically for a given distributional discrepancy, where the discrepancy between $\hat{F}_Y$ and $F_X$ is $\epsilon'$. Consequently, the remaining probabilistic guarantees hold with probability at least $1 - \delta$.
\end{proof}

\section{Comparison with the Concentration Bounds of Thomas et al.~\texorpdfstring{\cite{pmlr-v97-thomas19a}}{}}\label{sec:thomas_comparison}
\begin{figure}[htbp]
	\centering
	\includegraphics[width=\textwidth]{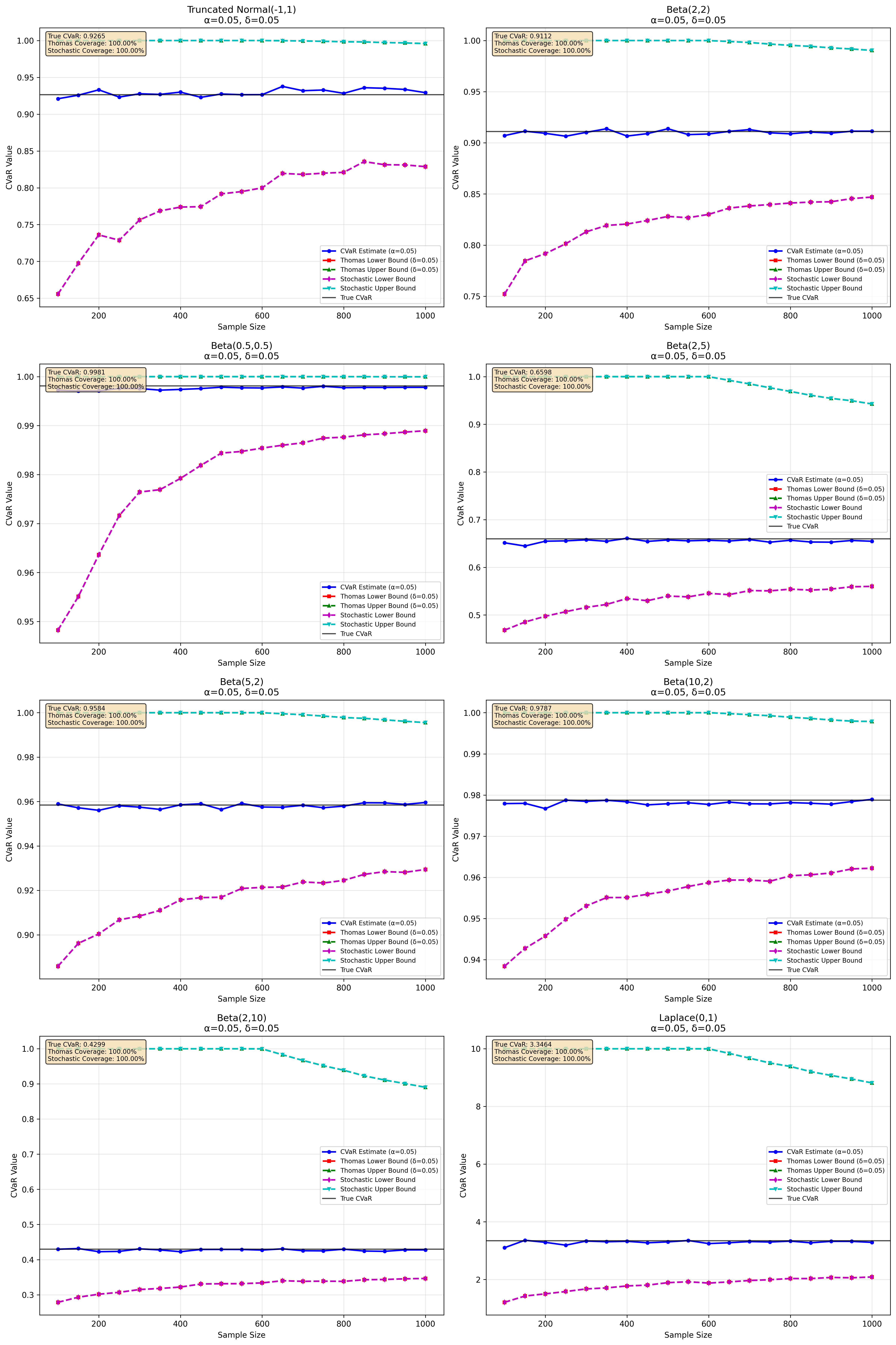}
	\caption{Comparison of concentration inequalities for $\CVaR_\alpha(X)$ established by Theorem~\ref{thm:ecdf_cvar_bound} (this work) and by \cite{pmlr-v97-thomas19a}, evaluated on $\operatorname{Beta}(2,2)$, $\operatorname{Beta}(0.5,0.5)$, $\operatorname{Beta}(2,5)$, $\operatorname{Beta}(5,2)$, $\operatorname{Beta}(10,2)$, $\operatorname{Beta}(2,10)$, and $\operatorname{Laplace}(0,1)$ distributions. The two sets of bounds coincide for all tested distributions.}
	\label{fig:cvar_bounds_comparison}
\end{figure}
See Figure \ref{fig:cvar_bounds_comparison}.

\section{Static CVaR Simplification Proofs}

\begin{theorem}\label{prf:simplification_bound_over_belief_distribution}
	\begin{equation}
		\begin{aligned}
			&\sup_{l\in \mathbb{R}}|P(R_{k:T} \leq l|b_k, a_k,\pi) \! -\! P_s(R_{k:T} \leq l|b_k, a_k,\pi)| \!
			\leq \!\! \! \sum_{t=k}^{T-1} \mathbb{E}^s[\Delta^s(b_t, a_t)|b_k, a_k],
		\end{aligned}
	\end{equation}
	where $\Delta^s$ is the TV distance that is defined by
	\begin{equation}
		\begin{aligned}
			&\Delta^s(b_{t-1},a_{t-1}) \triangleq 
			\int_{b_t\in B} |P(b_t|b_{t-1},a_{t-1})-P_s(b_t|b_{t-1},a_{t-1})|db_t.
		\end{aligned}
	\end{equation}
	
\end{theorem}
\begin{proof}
	By expanding the belief path from time $k+1$ to time $T$ we get
	\begin{equation}
		\begin{aligned}
			P(\sum_{t=k}^T c(b_t,a_t)\leq l|b_{k},\pi) 
			& = \int_{b_{k+1:T}\in B^{T-k}} P(R_{k:T}\leq l|b_{k:T},\pi) 
			 \prod_{i=k+1}^T P(b_i|b_{i-1}, \pi, a_i)db_{k+1:T}\\
			& = \int_{b_{k+1:T}\in B^{T-k}} 1_{R_{k:T}\leq l}\prod_{i=k+1}^T P(b_i|b_{i-1}, \pi, a_i)db_{k+1:T},
		\end{aligned}
	\end{equation}
	where the second equality holds because $R_{k:T}$ is constant given $\pi$ and $b_{k:T}$. Hence, the subtraction between the simplified and original return CDFs can be exhibited using the difference between the multiplication of the belief transition models.
	\begin{equation}
		\begin{aligned}
			&P(R_{k:T}\leq l|b_k,\pi)-P_s(R_{k:T}\leq l|b_k,\pi)
			=\int_{b_{k+1:T}\in B^{T-k}} 1_{R_{k:T}\leq l} [\prod_{i=k+1}^T P(b_i|b_{i-1}, \pi, a_i) 
			- \prod_{i=k+1}^T P_s(b_i|b_{i-1}, \pi, a_i)]db_{k+1:T}
		\end{aligned}
	\end{equation} 
	By applying the triangle inequality we get the following bound
	\begin{equation}\label{eq:prf_first_belief_cdf_bound}
		\begin{aligned}
			&|P(R_{k:T}\leq l|b_k,\pi)-P_s(R_{k:T}\leq l|b_k,\pi)|
			\leq \int_{b_{k+1:T}\in B^{T-k}} 1_{R_{k:T}\leq l} |\prod_{i=k+1}^T P(b_i|b_{i-1}, \pi, a_i) 
			- \prod_{i=k+1}^T P_s(b_i|b_{i-1}, \pi, a_i)|db_{k+1:T} \\
			&\leq \int_{b_{k+1:T}\in B^{T-k}} |\prod_{i=k+1}^T P(b_i|b_{i-1}, \pi, a_i) 
			- \prod_{i=k+1}^T P_s(b_i|b_{i-1}, \pi, a_i)|db_{k+1:T} = g(T),
		\end{aligned}
	\end{equation}
	where 
	\begin{equation}
		\begin{aligned}
			g(t) &\triangleq \int_{b_{k+1:t}\in B^{t-k+1}} |\prod_{i=k+1}^t P(b_i|b_{i-1}, \pi, a_i) 
			- \prod_{i=k+1}^t P_s(b_i|b_{i-1}, \pi, a_i)|db_{k+1:t}.
		\end{aligned}
	\end{equation}
	We will prove that
	\begin{equation}
		g(t+1) \leq g(t) + \mathbb{E}^s[\Delta^s (b_{t}, a_t)|b_k, a_k]=\epsilon_t
	\end{equation}
	for $t \in \{k+1,\dots, T\}$, where $\epsilon_k=0$, and therefore 
	\begin{equation}
		|P(R_{k:t}\leq l|b_k,\pi)-P_s(R_{k:t}\leq l|b_k,\pi)|\leq \epsilon_t.
	\end{equation}
	For the case where $t=k$, $b_t$ is constant, and therefore the $c(b_t, a_t)$ does not depend on the belief transition model. Hence, $P_s(c(b_t, a_t)\leq l|b_t, a_t)=P(c(b_t, a_t)\leq l|b_t, a_t)$. Therefore we get
	\begin{equation}
		P(c(b_k, a_k)\leq l|b_k, a_k) - P_s(c(b_k, a_k)\leq l|b_k, a_k)|=0=\epsilon_k.
	\end{equation}
	The rest of the proof will use induction over $t\in \{k+1, \dots, T\}$.
	
	\textbf{Base case, $t=k+1$:}
	\begin{equation}
		\begin{aligned}
			&g(k+1)=\int_{b_{k+1}\in B} |\prod_{i=k+1}^{k+1} P(b_i|b_{i-1}, \pi, a_i) 
			- \prod_{i=k+1}^{k+1} P_s(b_i|b_{i-1}, \pi, a_i)|db_{k+1:k+1} \\
			&= \int_{b_{k+1}\in B} |P(b_{k+1}|b_{k}, \pi, a_{k}) - P_s(b_{k+1}|b_{k}, \pi, a_k)|db_{k+1} 
			= \Delta^s(b_{k}, a_k)=\mathbb{E}^s[\Delta^s(b_{k}, a_k)|b_k,a_k] \\
			&= \epsilon_{k+1}
		\end{aligned}
	\end{equation} 
	
	\textbf{Induction step:} Assume that the claim is true for $t\in \{k+1, \dots, T-1\}$ and prove for $t+1$.
	
	Our proof's strategy is to express $g(t+1)$ recursively, using $g(t)$, and separate the bound at time $t+1$ into 2 components - one at time $t$ and one at time $t+1$. We then bound the term at time $t+1$, and use the induction step to bound the recursive part of $g$ at time $t$. 
	
	For all $x_i,y_i\in \mathbb{R}$,
	\begin{equation}
		\begin{aligned}
			&\prod_{i=k+1}^{t+1} x_i-\prod_{i=k+1}^{t+1}y_i
			=\prod_{i=k+1}^{t+1} x_i - x_{t+1} \prod_{i=k+1}^t y_i + x_{t+1} \prod_{i=k+1}^t y_i - \prod_{i=k+1}^{t+1}y_i 
			=x_{t+1}(\prod_{i=k+1}^t x_i - \prod_{i=k+1}^t y_i) + (x_{t+1} - y_{t+1}) \prod_{i=k+1}^{t} y_i.
		\end{aligned}
	\end{equation}
	By denoting $x_i=P(b_i|b_{i-1},a_{i-1})$ and $y_i=P_s(b_i|b_{i-1},a_{i-1})$ we get
	\begin{equation}
		\begin{aligned}
			&\prod_{i=k+1}^{t+1} P(b_i|b_{i-1}, \pi, a_i) - \prod_{i=k+1}^{t+1} P_s(b_i|b_{i-1}, \pi, a_i)\\
			&= [P(b_{t+1}|b_{t},a_{t}) - P_s(b_{t+1}|b_{t},a_{t})] \prod_{i=k+1}^{t}P_s(b_i|b_{i-1},a_{i-1}) 
			+ P(b_{t+1}|b_{t},a_{t})[\prod_{i=k+1}^{t} P(b_i|b_{i-1}, \pi, a_i) 
			- \prod_{i=k+1}^{t} P_s(b_i|b_{i-1}, \pi, a_i)].
		\end{aligned}
	\end{equation}
	By combining the equation above with the triangle inequality, we get 
	\begin{equation}
		g(t+1)\leq A_1 + A_2
	\end{equation}
	for 
	\begin{equation}
		\begin{aligned}
			A_1 &\triangleq \int\limits_{b_{k+1:t+1}\in B^{t-k}} |P(b_{t+1}|b_{t},a_{t}) - P_s(b_{t+1}|b_{t},a_{t})| 
			\prod_{i=k+1}^{t}P_s(b_i|b_{i-1},a_{i-1})b_{k+1:t},
		\end{aligned}
	\end{equation}
	\begin{equation}
		\begin{aligned}
			A_2 &\triangleq \int\limits_{b_{k+1:t+1}\in B^{t-k}} P(b_{t+1}|b_{t},a_{t}) |\prod_{i=k+1}^t P(b_i|b_{i-1}, \pi, a_i) 
			- \prod_{i=k+1}^t P_s(b_i|b_{i-1}, \pi, a_i)|db_{k+1:t} 
		\end{aligned}
	\end{equation}
	$A_1$ is the expectation over the TV-distance of the belief at time t.
	\begin{equation}
		\begin{aligned}
			&A_1 = \mathbb{E}^s[\Delta^s (b_{t}, a_t)|b_k, a_k].
		\end{aligned}
	\end{equation}
	As we see below, $A_2$ does not depend on the belief integration at time $t+1$, and therefore equals to $g(t)$. From the induction assumption, we get that $g(t)\leq \epsilon_t$.
	\begin{equation}
		\begin{aligned}
			&A_2 = \int\limits_{b_{k+1:t}\in B^{t-k}} P(b_{t+1}|b_{t},a_{t})  |\prod_{i=k+1}^t P(b_i|b_{i-1}, \pi, a_i) 
			- \prod_{i=k+1}^t P_s(b_i|b_{i-1}, \pi, a_i)|db_{k:t+1} \\
			&= \int\limits_{b_{k+1:t}\in B^{t-k}} |\prod_{i=k+1}^t P(b_i|b_{i-1}, \pi, a_i) - \prod_{i=k+1}^t P_s(b_i|b_{i-1}, \pi, a_i)| 
			\underbrace{\int_{b_{t+1}}P(b_{t+1}|b_{t},a_{t})db_{t+1}}_{=1} db_{k+1:t} \\
			&=g(t)\leq \epsilon_{t}
		\end{aligned}
	\end{equation}
	By combining the computations of $A_1$ and $A_2$,
	\begin{equation}
		\begin{aligned}
			g(t+1)&\leq A_1 + A_2 = g(t) + \mathbb{E}^s[\Delta^s (b_{t}, a_t)|b_k, a_k].
		\end{aligned}
	\end{equation}
	The last equation completes the induction proof. Utilizing the recursive relation of $\epsilon_t$, we get the bound we want to prove.
	\begin{equation}
		\begin{aligned}
			&|P(R_{k:T} \leq l|b_k,\pi)-P_s(R_{k:T} \leq l|b_k,\pi)|\leq g(T)
			\leq g(T-1) + \mathbb{E}^s[\Delta^s (b_{T-1}, a_{T-1})|b_k, a_k] 
			\leq \sum_{t=k}^{T-1} \mathbb{E}^s[\Delta^s(b_t, a_t)|b_k, a_k].
		\end{aligned}
	\end{equation}
\end{proof}

\begin{theorem}\label{prf:observation_model_simplification_bounds}
	In the case where the simplified and original observation models are denoted by $q_z$ and $p_z$ respectively, it holds that
	\begin{equation}
		\begin{aligned}
			&\sup_{l\in \mathbb{R}}|P(R_{k:T} \leq l|b_k,\pi) \! -\! P_s(R_{k:T} \leq l|b_k,\pi)| \! 
			\leq \!\! \! \sum_{t=k}^{T-1} \mathbb{E}^s[\Delta^s(x_{t+1})|b_k, a_k],
		\end{aligned}
	\end{equation}
	where $\Delta^s$ is the TV distance that is defined by
	\begin{equation}\label{eq:TVdistState_appendix}
		\Delta^s(x_{t}) \triangleq \int_{z_{t}\in Z} |p(z_{t}|x_t)-q(z_t|x_t)|dz_t.
	\end{equation}
\end{theorem}
\begin{proof}
	We show that for each $t\in \{k, \dots, T-1\}$,
	\begin{equation}
		\mathbb{E}^s[\Delta^s(b_t, a_t)|b_k, a_k]\leq \mathbb{E}^s[\Delta^s(x_{t+1})|b_k, a_k],
	\end{equation}
	and therefore 
	\begin{equation}
		\begin{aligned}
			&\sup_{l\in \mathbb{R}}|P(R_{k:T} \leq l|b_k,\pi) \! -\! P_s(R_{k:T} \leq l|b_k,\pi)| \! 
			\leq \!\! \! \sum_{t=k}^{T-1} \mathbb{E}^s[\Delta^s(b_t, a_t)|b_k, a_k] 
			\leq \sum_{t=k}^{T-1} \mathbb{E}^s[\Delta^s(x_{t+1})|b_k, a_k].
		\end{aligned}
	\end{equation}
	
	We start by expressing the belief transition model using the observation and state-transition models. Here we assume that $b_t$ is known. That is, $b_t=\{(x_i, w_i)\}_{i=1}^n$, where $x_i$ and $w_i$ are constants.
	\begin{equation}
		\begin{aligned}
			&P(b_{t+1}|b_{t},a_{t})=\int_{x_t\in X, x_{t+1}\in X, z_{t+1}\in Z} 
			P(b_{t+1}|b_{t}, a_{t}, z_{t+1}, x_t, x_{t+1})P(z_{t+1}, x_t, x_{t+1}|b_t,a_t) 
			dz_{t+1}dx_{t+1}dx_t \\
			&=\int_{x_t\in X, x_{t+1}\in X, z_{t+1}\in Z} 
			P(b_{t+1}|b_{t}, a_{t}, z_{t+1}, x_t, x_{t+1})p(z_{t+1}|x_{t+1})
			\times P(x_{t+1}|x_t, a_t)b(x_t)dz_{t+1}dx_{t+1}dx_t. 
		\end{aligned}
	\end{equation}
	Similarly,
	\begin{equation}
		\begin{aligned}
			&P_s(b_{t+1}|b_{t},a_{t}) 
			=\int_{x_t\in X, x_{t+1}\in X, z_{t+1}\in Z} P_s(b_{t+1}|b_{t}, a_{t}, z_{t+1}, x_t, x_{t+1}) 
			 q(z_{t+1}|x_{t+1})P(x_{t+1}|x_t, a_t)b(x_t)dz_{t+1}dx_{t+1}dx_t. 
		\end{aligned}
	\end{equation}
	Note that $P(b_{t+1}|b_{t}, a_{t}, z_{t+1}, x_t, x_{t+1})=P_s(b_{t+1}|b_{t}, a_{t}, z_{t+1}, x_t, x_{t+1})$, as both terms do not depend on the observation model. By subtracting the simplified and original belief transition models, we can separate the simplified and original observation models.
	\begin{equation}\label{eq:belief_transition_bound_with_simp_obs}
		\begin{aligned}
			&|P(b_{t+1}|b_{t},a_{t})-P_s(b_{t+1}|b_{t},a_{t})| \\
			&=|\int_{x_t\in X, x_{t+1}\in X, z_{t+1}\in Z} P_s(b_{t+1}|b_{t}, a_{t}, z_{t+1}, x_t, x_{t+1}) 
			P(x_{t+1}|x_t, a_t)b(x_t) 
			(p(z_{t+1}|x_{t+1}) - q(z_{t+1}|x_{t+1}))dz_{t+1}dx_{t+1}dx_t| \\
			&\leq \int_{x_t\in X, x_{t+1}\in X, z_{t+1}\in Z} P_s(b_{t+1}|b_{t}, a_{t}, z_{t+1}, x_t, x_{t+1}) 
			P(x_{t+1}|x_t, a_t)b(x_t) 
			|p(z_{t+1}|x_{t+1}) - q(z_{t+1}|x_{t+1})|dz_{t+1}dx_{t+1}dx_t.
		\end{aligned}
	\end{equation}
	where the inequality holds from the triangle inequality. It holds that \begin{equation}
		\begin{aligned}
			&\Delta^s(b_{t}, a_{t})=\int_{b_{t+1}\in B} |P(b_{t+1}|b_{t},a_{t})-P_s(b_{t+1}|b_{t},a_{t})|db_{t+1} \\
			&\leq \int_{b_{t+1}\in B} \int_{x_t\in X, x_{t+1}\in X, z_{t+1}\in Z} P_s(b_{t+1}|b_{t}, a_{t}, z_{t+1}, x_t, x_{t+1}) 
			P(x_{t+1}|x_t, a_t)b(x_t) 
			|p(z_{t+1}|x_{t+1}) - q(z_{t+1}|x_{t+1})|dz_{t+1}dx_{t+1}dx_tdb_{t+1} \\
			&=\int_{x_t\in X, x_{t+1}\in X} P(x_{t+1}|x_t, a_t)b(x_t) 
			\int_{z_{t+1}\in Z} |p(z_{t+1}|x_{t+1}) - q(z_{t+1}|x_{t+1})| 
			\underbrace{\int_{b_{t+1}\in B}P_s(b_{t+1}|b_{t}, a_{t}, z_{t+1}, x_t, x_{t+1}) db_{t+1}}_{= 1} dz_{t+1}dx_{t+1}dx_t \\
			&= \int_{b_{t+1}\in B} \int_{x_t\in X, x_{t+1}\in X} P(x_{t+1}|x_t, a_t)b(x_t) \Delta^s(x_{t+1}) dx_{t+1}dx_t \\
			&=\mathbb{E}^s[\Delta^s(x_{t+1})|b_t, a_t]
		\end{aligned}
	\end{equation}
	where the first inequality holds from \eqref{eq:belief_transition_bound_with_simp_obs}. From the last equation and expectation monotonicity, we get that $\mathbb{E}^s[\Delta(b_{t}, a_{t})]\leq \mathbb{E}^s[\Delta^s(x_{t+1})|b_t, a_t]$, and therefore the claim holds for each $t$. Summing over $t$ and applying the tower property yields the result.
	\begin{equation}
		\begin{aligned}
			\mathbb{E}^s[\Delta^s(b_{t}, a_{t})|b_k, a_k]&=\mathbb{E}^s[\mathbb{E}^s[\Delta^s(b_{t}, a_{t})|b_t, a_t]|b_k, a_k] 
			\leq \mathbb{E}^s[\mathbb{E}^s[\Delta^s(x_{t+1})|b_t, a_t]|b_k, a_k]
			=\mathbb{E}^s[\Delta^s(x_{t+1})|b_k, a_k].
		\end{aligned}
	\end{equation}
\end{proof}

\begin{theorem}\label{prf:uniform_lower_and_upper_bounds_for_v_and_q}
	Let $\delta_{d_{min}}\in [0, 1], \delta_{d_{max}}\in [0, 1]$ and denote 
	\begin{equation}
		\epsilon \triangleq \epsilon(b_k, a_k) \triangleq \min(\sum_{i=k}^{T-1} \mathbb{E}^s[\Delta^s(b_i, a_i)|b_k, a_k,\pi], 1),
	\end{equation}
	\begin{enumerate}
		\item \textbf{Upper Bound:} assume that $P(R_{k:T}\leq d_{max})>1-\delta_{d_{max}}$ and $P_s(R_{k:T}\leq d_{max})>1-\delta_{d_{max}}$, then \begin{enumerate}
			\item If $\epsilon<\alpha$, then
			$U_s\triangleq (1-\frac{\epsilon}{\alpha}) Q^\pi_{M_s}(b_k,a_k,\alpha-\epsilon) + \frac{\epsilon}{\alpha}d_{max}$
			\item If $\epsilon \geq \alpha$ then $U_s\triangleq d_{max}$
		\end{enumerate}
		\item \textbf{Lower Bound:} assume that $P(R_{k:T}\geq d_{min})>1-\delta_{d_{min}}$ and $P_s(R_{k:T}\geq d_{min})>1-\delta_{d_{min}}$, then \begin{enumerate}
			\item If $\epsilon + \alpha < 1$ then 
			$
			L_s\triangleq (1+\frac{\epsilon}{\alpha}) Q^\pi_{M_s}(b_k,a_k,\alpha+\epsilon) - \frac{\epsilon}{\alpha} Q^\pi_{M_s}(b_k,a_k, \epsilon)
			$
			\item If $\epsilon + \alpha \geq 1$ then 
			$
			L_s\triangleq\frac{1}{\alpha}[-(\alpha + \epsilon - 1)
			d_{min}+Q^\pi_{M_s}(b_k,a_k)-\epsilon Q^\pi_{M_s}(b_k,a_k,\epsilon)]
			$
		\end{enumerate}
	\end{enumerate}
	Then $P(L_s \leq Q_M^\pi(b_k, a_k,\alpha)) > 1-\delta_{d_{min}}$ and $P(Q_M^\pi(b_k, a_k,\alpha)\leq U_s)> 1-\delta_{d_{max}}$.
\end{theorem}
\begin{proof}
	From the law of total probability,
	\begin{equation}
		\begin{aligned}
			P(Q_M^\pi(b_k, a_k,\alpha)\leq U_s) 
			&= P(Q_M^\pi(b_k, a_k,\alpha) \leq U_s| R_{k:T}\leq d_{max})P(R_{k:T}\leq d_{max}) \\
			&+P(Q_M^\pi(b_k, a_k,\alpha) \leq U_s| R_{k:T} > d_{max})P(R_{k:T} > d_{max})\\
			&\geq P(Q_M^\pi(b_k, a_k,\alpha) \leq U_s| R_{k:T}\leq d_{max})P(R_{k:T}\leq d_{max}) \\
			&>P(Q_M^\pi(b_k, a_k,\alpha) \leq U_s| R_{k:T}\leq d_{max})(1-\delta_{d_{max}}).
		\end{aligned}
	\end{equation}
	Similarly, 
	\begin{equation}
		P(Q_M^\pi(b_k, a_k,\alpha) \geq L_s) > P(Q_M^\pi(b_k, a_k,\alpha) \geq L_s | R_{k:T}\geq d_{min})(1-\delta_{d_{min}}).
	\end{equation}
	From Theorem \ref{prf:simplification_bound_over_belief_distribution}, for all $l\in \mathbb{R}$
	\begin{equation}
		\begin{aligned}
			&|P(R_{k:T} \leq l|b_k,\pi)-P_s(R_{k:T} \leq l|b_k,\pi)| 
			\leq \sum_{i=k+1}^{T-1} \mathbb{E}_{b_{{k+1}:i}}
			[\Delta^s(b_i, a_i)|b_k,\pi]=\epsilon,
		\end{aligned}
	\end{equation}
	The action-value function is defined as the CVaR of the return. Let $X$ and $Y$ be random variables distributed according to the original and simplified models $P$ and $P_s$, respectively. Theorem~\ref{thm:cvar_bound_v2} establishes bounds on $\CVaR_\alpha(X)$ in terms of $\CVaR$ of $Y$, under a known discrepancy between the cumulative distribution functions of $X$ and $Y$. In our setting, this distributional discrepancy is bounded by $\epsilon$, and thus the result applies directly.
\end{proof}

\begin{corollary}\label{prf:cvar_tv_dist_bound_simp_obs}
	The bound from Theorem \ref{thm:uniform_lower_and_upper_bounds_for_v_and_q} holds for $\epsilon(b_k, a_k)=\sum_{t=k}^{T-1} \mathbb{E}^s[\Delta^s(x_{t+1})|b_k, a_k]$.
\end{corollary}
\begin{proof}
	Theorem~\ref{thm:uniform_lower_and_upper_bounds_for_v_and_q} establishes bounds on the value and action-value functions under a quantified distributional discrepancy $\epsilon$ between the original and simplified return distributions. As long as this discrepancy between their CDFs remains bounded by $\epsilon$, the result remains valid. Consequently, by bounding the belief-based discrepancy $\epsilon$ in Theorem~\ref{thm:uniform_lower_and_upper_bounds_for_v_and_q} via the state-based discrepancy bound established in Theorem~\ref{prf:observation_model_simplification_bounds}, we recover the same form of the bound for the newly defined $\epsilon$.
\end{proof}

\begin{theorem}\label{prf:m_i_sum_bound_guarantees}
	Let $\eta>0,\delta\in (0,1)$. If $\hat{\Delta}$ is unbiased, then
	\begin{equation}
		\begin{aligned}
			&P(|\hat{\epsilon}_t(\bar{b}_t, a)-\epsilon_t(\bar{b}_t, a)|\geq \eta)
			\leq 2\exp\Big( -\frac{2\eta^2 N_b}{D^2(T-k)^2(R_{\max}-R_{\min})^2} \Big)
		\end{aligned}
	\end{equation}
	for $D=\max_{i\in \{k+1,\dots,T-1\}} D_i$ where $D_i=\sup_{x_{i+1}}P(x_{i+1}|x_i, a_i)/Q_0(x_{i+1})$.
\end{theorem}
\begin{proof}
	We start by proving that $\hat{m}(\bar{b}_t, a)$ is an unbiased estimator for $m(\bar{b}_t, a)$.
	\begin{equation}
		\begin{aligned}
			&\mathbb{E}[\hat{m}(\bar{b}_t, a)]=\sum_{i=1}^{N_p} \mathbb{E}[\hat{m}(x_t^i, a)]\tilde{w}^i_t
			=\sum_{i=1}^{N_p} \mathbb{E}[\frac{P(x_n^\Delta|x_t^i, a)}{Q_0(x_n^\Delta)}\Delta(x_n^\Delta)]\tilde{w}^i_t=\sum_{i=1}^{N_p} m(x_t^i, a)\tilde{w}^i_t  \\
			&=m(\bar{b}_t, a).
		\end{aligned}
	\end{equation}
	It follows that the estimator $\hat{\epsilon}(\bar{b}_t, a)$ of the total distributional discrepancy is an unbiased estimator of $\epsilon(\bar{b}_t, a)$.
	\begin{equation}
		\begin{aligned}
			\mathbb{E}[\hat{\epsilon}(\bar{b}_t, a)]&=\sum_{\tau=t}^{T-1} \mathbb{E}[\hat{\mathbb{E}}[m(\bar{b}_\tau, \pi_\tau)|\bar{b}_t, a]|\bar{b}_t, a]
			=\sum_{\tau=t}^{T-1} \mathbb{E}[\hat{m}(\bar{b}_\tau^1, \pi(\bar{b}_\tau^1))|\bar{b}_t, a]
			=\sum_{\tau=t}^{T-1} \mathbb{E}[\mathbb{E}[\hat{m}(\bar{b}_\tau^1, \pi(\bar{b}_\tau^1))|\bar{b}_\tau^1]|\bar{b}_t, a] \\
			&=\sum_{\tau=t}^{T-1} \mathbb{E}[m(\bar{b}_\tau^1, \pi(\bar{b}_\tau^1))|\bar{b}_t, a]=\epsilon(\bar{b}_t, a)
		\end{aligned}
	\end{equation}
	
	Define
	\begin{equation}
		X_i \triangleq \frac{1}{N_b} \sum_{\tau=t}^{T-1} \hat{m}(\bar{b}_\tau^i, \pi(\bar{b}_\tau^i)),
	\end{equation}
	and note that
	\begin{equation}
		\sum_{i=1}^{N_b} X_i = \hat{\epsilon}_t(\bar{b}_t, a).
	\end{equation}
	Since the total variation distance is bounded by 2, each $X_i$ satisfies the bound
	\begin{equation}
		\frac{2D(T - t) R_{\min}}{N_b} \leq X_i \leq \frac{2D(T - t) R_{\max}}{N_b}.
	\end{equation}
	Moreover, the $X_i$ are independent and identically distributed, and thus Hoeffding’s inequality applies.
	\begin{equation}
		\begin{aligned}
			P(|\hat{\epsilon}_t(\bar{b}_t, a)-\epsilon_t(\bar{b}_t, a)|\geq \eta)
			=P(|\sum_{i=1}^{N_b}X_i-\epsilon_t(\bar{b}_t, a)|\geq \eta)
			&\leq 2\exp(-\frac{2\eta^2}{D^2(T-k)^2(R_{\max}-R_{\min})^2 / N_b^2 \times N_b}) \\
			&\leq 2\exp(-\frac{2\eta^2 N_b}{D^2(T-k)^2(R_{\max}-R_{\min})^2}).
		\end{aligned}
	\end{equation}
\end{proof}

\begin{theorem}\label{prf:d_guarantees}
	For the estimators $\hat{d}_{\max}^\pi(b_k) = \max_{i \in \{1,\dots,N_b\}} R^i$ and $\hat{d}_{\min}^\pi(b_k) = \min_{i \in \{1,\dots,N_b\}} R^i$, where $R^1,\dots,R^{N_b} \sim P_s(\cdot | b_k, \pi)$, and for $\epsilon$ as defined in \eqref{eq:epsilon}, the following bounds hold. 
	\begin{equation}
		P_s(R_{k:T} \geq \hat{d}_{min}^\pi(b_k))\geq\frac{N_b}{N_b+1}, \qquad P_s(R_{k:T} \leq \hat{d}_{max}^\pi(b_k))\geq\frac{N_b}{N_b+1}
	\end{equation}
	\begin{equation}
		P\!\left(R_{k:T} \geq \hat{d}_{min}^\pi(b_k) \Big| b_k ,\pi \right)
		\geq \frac{N_b}{N_b+1} - \epsilon, \qquad
		P\!\left(R_{k:T} \leq \hat{d}_{max}^\pi(b_k) \Big| b_k ,\pi \right)
		\geq \frac{N_b}{N_b+1} - \epsilon.
	\end{equation}
\end{theorem}
\begin{proof}
	Note that $\{R^i\}_{i=1}^{N_b} \sim P_s(\cdot | b_k, \pi)$ are sampled from the simplified distribution. Observe that $R_{k:T}, R^1,\dots, R^{N_b}$ constitute an i.i.d.\ sample of size $N_b+1$ from the simplified distribution. For each $j \in \{0,1,\dots,N_b\}$, define the event $A_j = \{X_j \text{ is strictly the largest among all } N_b+1 \text{ samples}\}$, where $X_0 = R_{k:T}$ and $X_i = R^i$ for $i \geq 1$. These events are mutually exclusive, and by the i.i.d.\ assumption $P_s(A_j)$ is the same for all $j$. Since $\sum_{j=0}^{N_b} P_s(A_j) \leq 1$, it follows that
	\begin{equation}
		P_s\!\left(R_{k:T} > \max_{i \in \{1,\dots, N_b\}} R^i \Big| b_k ,\pi \right)
		= P_s(A_0)
		\leq \frac{1}{N_b+1}.
	\end{equation}
	Similarly, $P_s(R_{k:T}<\min_{i \in \{1,\dots, N_b\}} R^i | b_k, \pi)\leq\frac{1}{N_b+1}$. Therefore,
	\begin{equation}
		P_s(R_{k:T} \leq \max_{i \in \{1,\dots, N_b\}} R^i | b_k ,\pi )\geq\frac{N_b}{N_b+1}, \qquad
		P_s(R_{k:T} \geq \min_{i \in \{1,\dots, N_b\}} R^i | b_k ,\pi )\geq\frac{N_b}{N_b+1}.
	\end{equation}
	
	When bounding
	$P\left(R_{k:T} \le \max_{i \in {1,\dots, N_b}} R^i\right)$,
	observe that $R_{k:T}$ is drawn from the original distribution $P$, whereas the samples $\{R^i\}_{i=1}^{N_b}$ are drawn from the simplified distribution $P_s$. Consequently, the symmetry argument used in the previous bound does not apply. To address this mismatch, we employ the bound from Theorem \ref{thm:simplification_bound_over_belief_distribution} to bound the CDF of the return computed with the original model, using the CDF computed using the simplified belief-transition model.
	\begin{equation}
		\begin{aligned}
			&P\!\left(R_{k:T} \leq \max_{i \in \{1,\dots, N_b\}} R^i \Big| b_k ,\pi \right)
			=\int_{R^{1},\dots,R^{N_b}} P\!\left(R_{k:T} \leq \max_{i \in \{1,\dots, N_b\}} R^i \Big| R^1,\dots,R^{N_b}, b_k ,\pi \right) P_s(R^1,\dots,R^{N_b}|b_k ,\pi) dR^1\dots dR^{N_b} \\
			&\geq \int_{R^{1},\dots,R^{N_b}} \Big( P_s\!\left(R_{k:T} \leq \max_{i \in \{1,\dots, N_b\}} R^i \Big| R^1,\dots,R^{N_b}, b_k ,\pi \right)- \epsilon \Big) P_s(R^1,\dots,R^{N_b}|b_k ,\pi) dR^1,\dots,dR^{N_b} \\
			&=\mathbb{E}_s \Big[ P_s\!\left(R_{k:T} \leq \max_{i \in \{1,\dots, N_b\}} R^i \Big| R^1,\dots,R^{N_b} \right) - \epsilon  \Big| b_k ,\pi \Big] \\
			&= P_s\!\left(R_{k:T} \leq \max_{i \in \{1,\dots, N_b\}} R^i\right) - \epsilon \\
			&\geq\frac{N_b}{N_b+1} - \epsilon.
		\end{aligned}
	\end{equation}
	The first inequality follows directly from Theorem \ref{thm:simplification_bound_over_belief_distribution}, and the last inequality holds from the symmetry argument used above.
	
	Similarly, 
	\begin{equation}
		P\!\left(R_{k:T} \leq \min_{i \in \{1,\dots, N_b\}} R^i \Big| b_k ,\pi \right)
		\geq \frac{N_b}{N_b+1} - \epsilon.
	\end{equation}	
\end{proof}

\begin{theorem}\label{prf:trivial_d_guarantees}
	Let $\delta,\delta_{d_{\min}}, \delta_{d_{\max}} \in [0,1]$, and let $b_k$ denote the initial belief. Consider samples $R^1,\dots,R^{N_b} \sim P_s(\cdot \mid b_k,\pi)$ generated using the simplified belief-transition model. Define
	\begin{equation}
		\epsilon \triangleq \epsilon(b_k, a_k) \triangleq \min(\sum_{i=k}^{T-1} \mathbb{E}^s[\Delta^s(b_i, a_i)|b_k, a_k,\pi], 1),
	\end{equation}
	\begin{equation}
		\eta = \sqrt{\ln(1/\delta)/(2N_b)}, \qquad \epsilon' = \epsilon + \eta.
	\end{equation}
	\begin{enumerate}
		\item \textbf{Upper Bound:} assume that $P(R_{k:T}\leq d_{max})>1-\delta_{d_{max}}$ and $P_s(R_{k:T}\leq d_{max})>1-\delta_{d_{max}}$, then \begin{enumerate}
			\item If $\epsilon'<\alpha$, then
			$U_s\triangleq (1-\frac{\epsilon'}{\alpha}) \hat{Q}^\pi_{M_s}(b_k,a_k,\alpha-\epsilon') + \frac{\epsilon'}{\alpha}d_{max}$
			\item If $\epsilon' \geq \alpha$ then $U_s\triangleq d_{max}$
		\end{enumerate}
		\item \textbf{Lower Bound:} assume that $P(R_{k:T}\geq d_{min})>1-\delta_{d_{min}}$ and $P_s(R_{k:T}\geq d_{min})>1-\delta_{d_{min}}$, then \begin{enumerate}
			\item If $\epsilon' + \alpha < 1$ then
			$
			L_s\triangleq (1+\frac{\epsilon'}{\alpha}) \hat{Q}^\pi_{M_s}(b_k,a_k,\alpha+\epsilon') - \frac{\epsilon'}{\alpha} \hat{Q}^\pi_{M_s}(b_k,a_k, \epsilon')
			$
			\item If $\epsilon' + \alpha \geq 1$ then 
			$
			L_s\triangleq\frac{1}{\alpha}[-(\alpha + \epsilon' - 1)
			d_{min}+\hat{Q}^\pi_{M_s}(b_k,a_k)-\epsilon' \hat{Q}^\pi_{M_s}(b_k,a_k,\epsilon')]
			$
		\end{enumerate}
	\end{enumerate}
	Then $P(L_s \leq Q_M^\pi(b_k, a_k,\alpha)) > (1-\delta)(1-\delta_{d_{min}})$ and $P(Q_M^\pi(b_k, a_k,\alpha)\leq U_s)> (1-\delta)(1-\delta_{d_{max}})$.
\end{theorem}
\begin{proof}
	From the law of total probability,
	\begin{equation}
		\begin{aligned}
			P(Q_M^\pi(b_k, a_k,\alpha)\leq U_s) 
			&= P(Q_M^\pi(b_k, a_k,\alpha) \leq U_s| R_{k:T}\leq d_{max})P(R_{k:T}\leq d_{max}) \\
			&+P(Q_M^\pi(b_k, a_k,\alpha) \leq U_s| R_{k:T} > d_{max})P(R_{k:T} > d_{max})\\
			&\geq P(Q_M^\pi(b_k, a_k,\alpha) \leq U_s| R_{k:T}\leq d_{max})P(R_{k:T}\leq d_{max}) \\
			&>P(Q_M^\pi(b_k, a_k,\alpha) \leq U_s| R_{k:T}\leq d_{max})(1-\delta_{d_{max}}).
		\end{aligned}
	\end{equation}
	Similarly, 
	\begin{equation}
		P(Q_M^\pi(b_k, a_k,\alpha) \geq L_s) > P(Q_M^\pi(b_k, a_k,\alpha) \geq L_s | R_{k:T}\geq d_{min})(1-\delta_{d_{min}}).
	\end{equation}
	From Theorem \ref{prf:simplification_bound_over_belief_distribution}, for all $l\in \mathbb{R}$
	\begin{equation}
		\begin{aligned}
			&|P(R_{k:T} \leq l|b_k,\pi)-P_s(R_{k:T} \leq l|b_k,\pi)| 
			\leq \sum_{i=k+1}^{T-1} \mathbb{E}_{b_{{k+1}:i}}
			[\Delta^s(b_i, a_i)|b_k,\pi]=\epsilon,
		\end{aligned}
	\end{equation}
	The action-value function is defined as the CVaR of the return. Let $X$ and $Y$ be random variables distributed according to the original and simplified models $P$ and $P_s$, respectively. Theorem~\ref{thm:ecdf_y_bound_cvar_x} establishes bounds on $\CVaR_\alpha(X)$ in terms of $\CVaR$ of a sample of $Y$, under a known discrepancy between the cumulative distribution functions of $X$ and $Y$. In our setting, this distributional discrepancy is bounded by $\epsilon$, and thus the result applies directly.
\end{proof}

\begin{theorem}(Bound Guarantees)\label{prf:guarantees}
	Let $\delta \in (0,0.5)$, and let $b_k$ denote the initial belief. Consider samples $R^1,\dots,R^{N_b} \sim P_s(\cdot \mid b_k,\pi)$ generated using the simplified belief-transition model. Define
	\begin{equation}
		\epsilon \triangleq \epsilon(b_k, a_k) \triangleq \min(\sum_{i=k}^{T-1} \mathbb{E}^s[\Delta^s(b_i, a_i)|b_k, a_k,\pi], 1),
	\end{equation}
	\begin{equation}
		\eta = \sqrt{\ln(1/\delta)/(2N_b)}, \qquad \epsilon' = \epsilon + \eta,
	\end{equation}
	\begin{equation}
		\hat{d}_{\max}^\pi(b_k) = \max_{i \in \{1,\dots,N_b\}} R^i, \qquad \hat{d}_{\min}^\pi(b_k) = \min_{i \in \{1,\dots,N_b\}} R^i.
	\end{equation}
	\begin{enumerate}
		\item \textbf{Upper Bound:} \begin{enumerate}
			\item If $\epsilon'<\alpha$, then
			$U_s\triangleq (1-\frac{\epsilon'}{\alpha}) \hat{Q}^\pi_{M_s}(b_k,a_k,\alpha-\epsilon') + \frac{\epsilon'}{\alpha}\hat{d}_{\max}^\pi(b_k)$
			\item If $\epsilon' \geq \alpha$ then $U_s\triangleq \hat{d}_{\max}^\pi(b_k)$
		\end{enumerate}
		\item \textbf{Lower Bound:} \begin{enumerate}
			\item If $\epsilon' + \alpha < 1$ then
			$
			L_s\triangleq (1+\frac{\epsilon'}{\alpha}) \hat{Q}^\pi_{M_s}(b_k,a_k,\alpha+\epsilon') - \frac{\epsilon'}{\alpha} \hat{Q}^\pi_{M_s}(b_k,a_k, \epsilon')
			$
			\item If $\epsilon' + \alpha \geq 1$ then
			$
			L_s\triangleq\frac{1}{\alpha}[-(\alpha + \epsilon' - 1)
			\hat{d}_{\min}^\pi(b_k)+\hat{Q}^\pi_{M_s}(b_k,a_k)-\epsilon' \hat{Q}^\pi_{M_s}(b_k,a_k,\epsilon')]
			$
		\end{enumerate}
	\end{enumerate}
	Then $P(L_s \leq Q_M^\pi(b_k, a_k,\alpha)) > (1-\delta)\Big(\frac{N_b}{N_b+1} - \epsilon\Big)$ and $P(Q_M^\pi(b_k, a_k,\alpha)\leq U_s)> (1-\delta)\Big(\frac{N_b}{N_b+1} - \epsilon\Big)$.
\end{theorem}
\begin{proof}
	We prove the upper bound; the lower bound follows by an analogous argument applied to~$\hat{d}_{\min}^\pi(b_k)$.
	By Theorem~\ref{thm:d_guarantees},
	\begin{equation}
		P_s(R_{k:T} \leq \hat{d}_{\max}^\pi(b_k) \mid b_k, \pi) \geq \frac{N_b}{N_b+1}, \qquad
		P(R_{k:T} \leq \hat{d}_{\max}^\pi(b_k) \mid b_k, \pi) \geq \frac{N_b}{N_b+1} - \epsilon.
	\end{equation}
	Setting $d_{\max} = \hat{d}_{\max}^\pi(b_k)$ and $\delta_{d_{\max}} = \frac{1}{N_b+1} + \epsilon$ in Theorem~\ref{thm:trivial_d_guarantees}, both of its conditions are satisfied, and we obtain
	\begin{equation}
		P(Q^\pi_M(b_k, a_k, \alpha) \leq U_s) > (1-\delta)(1-\delta_{d_{\max}}) = (1-\delta)\Big(\frac{N_b}{N_b+1} - \epsilon\Big).
	\end{equation}
	An analogous argument with $d_{\min} = \hat{d}_{\min}^\pi(b_k)$ and $\delta_{d_{\min}} = \frac{1}{N_b+1} + \epsilon$ yields the lower bound guarantee.
\end{proof}

\section{Environment Configurations}\label{sec:env_configs}

The following subsections describe the three POMDP environments used throughout the experiments, including the action elimination experiments of Figure~\ref{fig:agent_path}. Each domain provides an \emph{original} environment (used for sampling rollouts) and a \emph{simplified} surrogate (used for belief updates and CVaR bound computation). The two models differ only in their observation and reward models, as summarised in Table~\ref{tab:arch_overview}.

\begin{table}[ht]
	\centering
	\small
	\setlength{\tabcolsep}{6pt}
	\begin{tabular}{|l|l|l|}
		\hline
		\textbf{Aspect} & \textbf{Original} & \textbf{Simplified} \\
		\hline
		Observation noise & Gaussian Mixture Model (GMM) & Isotropic Gaussian \\
		\hline
		Reward noise & Stochastic (Push, Laser Tag) & Deterministic \\
		\hline
	\end{tabular}
	\caption{Architectural differences between the original and simplified models across all domains.}
	\label{tab:arch_overview}
\end{table}

\subsection{Light-Dark}

\textbf{State space}: $\mathbf{s} = (x, y) \in \mathbb{R}^2$ (2D continuous position).
\textbf{Action space}: Discrete — $\{\text{up}, \text{down}, \text{left}, \text{right}\}$, mapped to unit vectors in $\mathbb{R}^2$.
\textbf{Observation space}: $\mathbf{o} \in \mathbb{R}^2$ — noisy 2D position.

\begin{table}[ht]
	\centering
	\small
	\setlength{\tabcolsep}{4pt}
	\begin{tabular}{|l|c|}
		\hline
		\textbf{Parameter} & \textbf{Value} \\
		\hline
		Grid size & $7 \times 7$ \\
		\hline
		Start state & $(1,\;1)$ \\
		\hline
		Goal state & $(6,\;6)$; radius $1.5$ \\
		\hline
		State transition covariance & $0.06 \cdot I_2$ \\
		\hline
		Obs.\ noise (far from beacon) & $\mathcal{N}(0,\;0.06\cdot I_2)$ / GMM($K{=}2500$) \\
		\hline
		Obs.\ noise (near beacon) & $\mathcal{N}(0,\;0.03\cdot I_2)$ / GMM($K{=}2500$, scaled) \\
		\hline
		Beacons & $(1,1),\;(1,6),\;(6,1),\;(6,6)$; radius $1.0$ \\
		\hline
		Obstacle & $(5,2)$; radius $3.0$; hit prob $1.0$ \\
		\hline
		Goal reward & $+10$ \\
		\hline
		Obstacle penalty & $-10$ \\
		\hline
		Step cost (fuel) & $-2$ \\
		\hline
		Discount factor $\gamma$ & $0.95$ \\
		\hline
		Planning horizon $H$ & $9$ steps \\
		\hline
	\end{tabular}
	\caption{Configuration for the 2D Light-Dark POMDP.}
	\label{tab:ld_full_config}
\end{table}

\textbf{Original}: Two independent 1D GMMs (one per axis), $K = 2500$ components, approximating the target Normal distribution.
\textbf{Simplified}: Standard Gaussian $\mathcal{N}(0,\;\sigma^2 I_2)$ with the same $\sigma$.

\subsection{Continuous Push}

\textbf{State space}: $\mathbf{s} = (r_x, r_y, o_x, o_y, t_x, t_y) \in \mathbb{R}^6$ — robot, object, and target positions.
\textbf{Action space}: Discrete — $\{\text{up}, \text{down}, \text{left}, \text{right}\}$, mapped to unit vectors.
\textbf{Observation space}: $\mathbf{o} \in \mathbb{R}^6$ — robot and target positions observed exactly; object position is noisy: $\hat{o} = o + \epsilon,\;\epsilon \sim \mathcal{N}(0, \sigma^2 I_2)$.

\begin{table}[ht]
	\centering
	\small
	\setlength{\tabcolsep}{4pt}
	\begin{tabular}{|l|c|}
		\hline
		\textbf{Parameter} & \textbf{Value} \\
		\hline
		Grid size & $6 \times 6$ \\
		\hline
		Initial state & $(0.5,\;0.5,\;1.0,\;0.5,\;5.0,\;5.0)$ \\
		\hline
		Robot transition covariance & $\mathrm{diag}(0.005,\;0.005)$ \\
		\hline
		Observation noise $\sigma$ & $0.1$ (on object position) \\
		\hline
		Push threshold & $1.0$ \\
		\hline
		Friction coefficient & $0.3$ \\
		\hline
		Max push force & $2.0$ \\
		\hline
		Robot radius & $0.3$ \\
		\hline
		Dangerous area & $(5.0,\;2.0)$; radius $1.0$; penalty $-20$ \\
		\hline
		Discount factor $\gamma$ & $0.95$ \\
		\hline
		Planning horizon $H$ & $9$ steps \\
		\hline
	\end{tabular}
	\caption{Configuration for the Continuous Push POMDP.}
	\label{tab:push_full_config}
\end{table}

\textbf{Original}: One 1D GMM ($K = 7000$ components) applied independently to each object-position axis; additionally adds stochastic reward noise $\mathcal{N}(0,\;25)$.
\textbf{Simplified}: Gaussian noise $\mathcal{N}(0,\;0.01\cdot I_2)$ on object position; deterministic reward.

\subsection{Continuous Laser Tag}

\textbf{State space}: $\mathbf{s} = (r_x, r_y, o_x, o_y, \text{term}) \in \mathbb{R}^4 \times \{0,1\}$ — robot and opponent positions plus terminal flag.
\textbf{Action space}: Discrete — $\{\text{up}, \text{down}, \text{left}, \text{right}, \text{tag}\}$.
\textbf{Observation space}: $\mathbf{o} \in \mathbb{R}^8$ — 8-directional laser range measurements with additive noise; terminal states return $\mathbf{o} = -\mathbf{1}$.

\begin{table}[ht]
	\centering
	\small
	\setlength{\tabcolsep}{4pt}
	\begin{tabular}{|l|c|}
		\hline
		\textbf{Parameter} & \textbf{Value} \\
		\hline
		Arena & $11 \times 7$ continuous grid \\
		\hline
		Walls & 8 axis-aligned wall cells \\
		\hline
		Initial state & $(2.0,\;1.0,\;8.0,\;5.0,\;0)$ \\
		\hline
		Robot transition covariance & $0.0125 \cdot I_2$ \\
		\hline
		Opponent transition covariance & $0.00625 \cdot I_2$ \\
		\hline
		Opponent pursuit speed & $0.6$ \\
		\hline
		Laser noise $\sigma$ & $1.0$ (per measurement) \\
		\hline
		Robot / opponent radius & $0.3$ \\
		\hline
		Tag radius & $0.5$ \\
		\hline
		Tag reward & $+10$ \\
		\hline
		Failed-tag penalty & $-10$ \\
		\hline
		Step cost & $-1$ \\
		\hline
		Dangerous areas & $(5.0,\;3.0),\;(7.0,\;1.0)$; radius $1.0$; penalty $-300$ \\
		\hline
		Discount factor $\gamma$ & $0.95$ \\
		\hline
		Planning horizon $H$ & $8$ steps (7 moves + tag) \\
		\hline
	\end{tabular}
	\caption{Configuration for the Continuous Laser Tag POMDP.}
	\label{tab:laser_tag_full_config}
\end{table}

\textbf{Original}: One 1D GMM with $K = 290$ components ($\sigma_{\text{comp}} = 0.097$), applied independently to all 8 laser measurements; additionally adds stochastic reward noise $\mathcal{N}(0,\;0.25)$.
\textbf{Simplified}: Gaussian noise $\mathcal{N}(0,\;1.0)$ per laser measurement; deterministic reward.

\subsection{Shared Planner Parameters}

\begin{table}[ht]
	\centering
	\small
	\setlength{\tabcolsep}{4pt}
	\begin{tabular}{|l|c|}
		\hline
		\textbf{Parameter} & \textbf{Value} \\
		\hline
		CVaR level $\alpha$ & $0.1$ \\
		\hline
		Belief particles $N_p$ & $10$ \\
		\hline
		Return samples $N_r$ & $600$ \\
		\hline
		$\Delta$-estimation states & $100$ \\
		\hline
		Samples per state (for $\Delta$) & $2000$ \\
		\hline
		$k$-NN neighbours & $10$ \\
		\hline
		Discount factor $\gamma$ & $0.95$ \\
		\hline
		Confidence parameter $\delta$ & $0.05$ \\
		\hline
		Random seed & $42$ \\
		\hline
	\end{tabular}
	\caption{Planner parameters shared across all three environments.}
	\label{tab:shared_planner_config}
\end{table}

These parameters are used in the open-loop bound evaluation experiments (Figures~\ref{fig:agent_path}, \ref{fig:acceleration_sensitivity}, and \ref{fig:horizon_bounds}) and in the closed-loop policy evaluation experiments (Figures~\ref{fig:closed_loop_policy_and_trajectories}, \ref{fig:closed_loop_sensitivity}, and \ref{fig:closed_loop_speedup}).

\section{Simulations}\label{sec:simulations_appendix}

The planning configurations below correspond to the open-loop experiments in Figures~\ref{fig:agent_path}, \ref{fig:acceleration_sensitivity}, and \ref{fig:horizon_bounds}.

Figure \ref{fig:agent_path} presents the estimated bounds and CVaR values for pairs of action sequences across all three environments, computed using 600 sampled belief trajectories and a confidence parameter of $\delta = 0.05$.

The simplified planner samples $N^\Delta = 100$ states offline from the square region defined by the corners $(0, 0)$ and $(7, 7)$. For each sampled state, the distributional discrepancy $\Delta$ between the simplified and original observation models is estimated using $N_z=2000$ samples. CVaR is computed at confidence level $\alpha = 0.5$, and the probabilistic bounds are constructed with confidence level $\delta = 0.05$. We set the number of sampled belief trajectories to $N_b=600$.

\subsubsection{Continuous Laser Tag -- Planning Configuration}

\begin{table}[ht]
	\centering
	\small
	\setlength{\tabcolsep}{4pt}
	\begin{tabular}{|l|c|}
		\hline
		\textbf{Parameter} & \textbf{Value} \\
		\hline
		Number of Episodes & 300 \\
		\hline
		Steps per Episode & 30 (max) \\
		\hline
		Planning Depth & 2 \\
		\hline
		Return Samples ($N_b$) & 600 \\
		\hline
		CVaR $\alpha$ & 0.5 \\
		\hline
		Number of Particles ($N_p$) & 10 \\
		\hline
		Discount Factor ($\gamma$) & 0.95 \\
		\hline
		GMM Components & 290 \\
		\hline
	\end{tabular}
	\caption{Planning configuration for the Continuous Laser Tag experiments.}
	\label{tab:laser_tag_config}
\end{table}

\subsubsection{Light-Dark -- Planning Configuration}

\begin{table}[ht]
	\centering
	\small
	\setlength{\tabcolsep}{4pt}
	\begin{tabular}{|l|c|}
		\hline
		\textbf{Parameter} & \textbf{Value} \\
		\hline
		Number of Episodes & 300 \\
		\hline
		Steps per Episode & 30 (max) \\
		\hline
		Planning Depth & 2 \\
		\hline
		Return Samples ($N_b$) & 600 \\
		\hline
		CVaR $\alpha$ & 0.5 \\
		\hline
		Number of Particles ($N_p$) & 10 \\
		\hline
		Discount Factor ($\gamma$) & 0.95 \\
		\hline
		GMM Components & 7000 \\
		\hline
	\end{tabular}
	\caption{Planning configuration for the Light-Dark experiments.}
	\label{tab:light_dark_config}
\end{table}

\subsubsection{Continuous Push -- Planning Configuration}

The planning parameters used for Continuous Push experiments are summarized in Table~\ref{tab:push_config}; see Figure~\ref{fig:push_env} for representative agent trajectories.

\begin{table}[ht]
	\centering
	\small
	\setlength{\tabcolsep}{4pt}
	\begin{tabular}{|l|c|}
		\hline
		\textbf{Parameter} & \textbf{Value} \\
		\hline
		Number of Episodes & 300 \\
		\hline
		Steps per Episode & 30 (max) \\
		\hline
		Planning Depth & 2 \\
		\hline
		Return Samples ($N_b$) & 600 \\
		\hline
		CVaR $\alpha$ & 0.5 \\
		\hline
		Number of Particles ($N_p$) & 10 \\
		\hline
		Discount Factor ($\gamma$) & 0.95 \\
		\hline
		GMM Components & 7000 \\
		\hline
	\end{tabular}
	\caption{Planning configuration for the Continuous Push experiments.}
	\label{tab:push_config}
\end{table}

\section{BetaZero Policy Configuration (Light-Dark)}\label{sec:betazero_config}

The closed-loop policy evaluation uses BetaZero~\cite{moss2024betazero}, an AlphaZero-style algorithm for POMDPs. A neural network maps a belief summary vector $\phi(b)$ to a scalar value estimate $V(b)$ and a policy prior $\pi(a \mid b)$, which guide MCTS over particle beliefs.

\subsection{Belief Representation}

Particles are summarised into a fixed-size feature vector using the mean-and-standard-deviation representation:
\begin{equation}
    \phi(b) = \bigl[\bar{\mu}_x,\ \bar{\mu}_y,\ \bar{\sigma}_x,\ \bar{\sigma}_y\bigr] \in \mathbb{R}^4,
\end{equation}
where $\bar{\mu}$ and $\bar{\sigma}$ are the weighted mean and standard deviation of the particle set.

\begin{table}[ht]
    \centering
    \small
    \setlength{\tabcolsep}{4pt}
    \begin{tabular}{|l|c|}
        \hline
        \textbf{Parameter} & \textbf{Value} \\
        \hline
        State dimension & 2 \\
        \hline
        Belief feature dimension & 4 (= $2 \times$ state dim) \\
        \hline
        Particles per belief $N_p$ & 500 \\
        \hline
    \end{tabular}
    \caption{Belief representation parameters for BetaZero on Light-Dark.}
    \label{tab:bz_belief}
\end{table}

Input normalisation is applied using online statistics collected during training:

\begin{table}[ht]
    \centering
    \small
    \setlength{\tabcolsep}{4pt}
    \begin{tabular}{|l|c|c|}
        \hline
        \textbf{Feature} & \textbf{Mean} & \textbf{Std} \\
        \hline
        $\bar{\mu}_x$ & 3.56 & 1.91 \\
        \hline
        $\bar{\mu}_y$ & 4.12 & 1.91 \\
        \hline
        $\bar{\sigma}_x$ & 0.302 & 0.407 \\
        \hline
        $\bar{\sigma}_y$ & 0.303 & 0.408 \\
        \hline
        Value output & $-24.34$ & 36.64 \\
        \hline
    \end{tabular}
    \caption{Input normalisation statistics for the BetaZero network.}
    \label{tab:bz_normalization}
\end{table}

\subsection{Network Architecture}

Two-headed MLP with a shared trunk and separate value and policy heads:

\begin{table}[ht]
    \centering
    \small
    \setlength{\tabcolsep}{4pt}
    \begin{tabular}{|l|c|}
        \hline
        \textbf{Layer} & \textbf{Size} \\
        \hline
        Input & 4 \\
        \hline
        Hidden 1 & 128 \\
        \hline
        Hidden 2 & 128 \\
        \hline
        Value head output & 1 (scalar) \\
        \hline
        Policy head output & 4 (one logit per action) \\
        \hline
    \end{tabular}
    \caption{BetaZero network architecture.}
    \label{tab:bz_network}
\end{table}

\subsection{MCTS Parameters}

\begin{table}[ht]
    \centering
    \small
    \setlength{\tabcolsep}{4pt}
    \begin{tabular}{|l|c|}
        \hline
        \textbf{Parameter} & \textbf{Value} \\
        \hline
        Search depth & 10 \\
        \hline
        Simulations per step (training) & 100 \\
        \hline
        Simulations per step (inference) & 1000 \\
        \hline
        Exploration constant $c$ & 219.13 (tuned) \\
        \hline
        Temperature $\tau$ & 0.319 (tuned) \\
        \hline
        Value weight $z_q$ & 1.0 \\
        \hline
        Visit count weight $z_n$ & 1.0 \\
        \hline
        DPW action widening $k_a$, $\alpha_a$ & 1.0,\ 0.5 \\
        \hline
        DPW observation widening $k_o$, $\alpha_o$ & 1.0,\ 0.5 \\
        \hline
        Discount factor $\gamma$ & 0.95 \\
        \hline
    \end{tabular}
    \caption{MCTS parameters for BetaZero on Light-Dark.}
    \label{tab:bz_mcts}
\end{table}

\subsection{Training Loop}

\begin{table}[ht]
    \centering
    \small
    \setlength{\tabcolsep}{4pt}
    \begin{tabular}{|l|c|}
        \hline
        \textbf{Parameter} & \textbf{Value} \\
        \hline
        Outer iterations & 50 \\
        \hline
        Episodes per iteration & 50 \\
        \hline
        Total training episodes & 2500 \\
        \hline
        Episode length & 20 steps \\
        \hline
        Data collection & Full MCTS rollouts \\
        \hline
        Replay buffer & 5 iterations \\
        \hline
    \end{tabular}
    \caption{BetaZero training loop parameters.}
    \label{tab:bz_training}
\end{table}

At each episode, a fresh initial belief is sampled: mean $\sim \mathcal{U}([-1,8]^2)$, std $\sim \mathcal{U}(0.01,\;2.0)$, then $N_p = 500$ particles are drawn from $\mathcal{N}(\text{mean},\;\text{std}^2 I)$. This diversity prevents the training distribution from collapsing to a uniform loop.

\subsection{Optimizer and Loss}

\begin{table}[ht]
    \centering
    \small
    \setlength{\tabcolsep}{4pt}
    \begin{tabular}{|l|c|}
        \hline
        \textbf{Parameter} & \textbf{Value} \\
        \hline
        Optimizer & Adam \\
        \hline
        Learning rate & $10^{-4}$ \\
        \hline
        Batch size & 1024 \\
        \hline
        Training epochs per iteration & 50 \\
        \hline
        Weight decay & $10^{-5}$ \\
        \hline
        Loss & Cross-entropy (policy) + MSE (value) \\
        \hline
    \end{tabular}
    \caption{Optimizer and loss configuration for BetaZero training.}
    \label{tab:bz_optimizer}
\end{table}

\subsection{Hyperparameter Search}

The exploration constant $c$ and temperature $\tau$ were selected via Optuna (TPE sampler, 20 trials) with $n_\text{sim}=100$ during search:

\begin{table}[ht]
    \centering
    \small
    \setlength{\tabcolsep}{4pt}
    \begin{tabular}{|l|c|c|}
        \hline
        \textbf{Hyperparameter} & \textbf{Search range} & \textbf{Best value} \\
        \hline
        $c$ (exploration constant) & $[1.0,\;10\times\Delta r] \approx [1,\;220]$ & 219.13 \\
        \hline
        $\tau$ (temperature) & $[0.05,\;0.5]$ & 0.319 \\
        \hline
        $z_q$,\ $z_n$ & fixed & 1.0,\ 1.0 \\
        \hline
    \end{tabular}
    \caption{Hyperparameter search results for BetaZero on Light-Dark.}
    \label{tab:bz_hparam}
\end{table}

Search used a MedianPruner (3 startup trials, 5 warmup steps) and a TPE sampler with seed 42. The objective minimised the final policy loss (cross-entropy vs.\ MCTS visit counts); the baseline uniform policy corresponds to $\ln 4 \approx 1.386$. The best trial was retrained at $n_\text{sim}=1000$ for the final model used in all experiments.

\section{Hardware Specifications}\label{sec:hardware}
All simulations were conducted on a machine equipped with an 11th Gen Intel\textsuperscript{\textregistered} Core\textsuperscript{TM} i9-11900K CPU running at 3.50\,GHz and 64\,GB of RAM.

\end{document}